\documentclass[a4paper,10pt]{article}
\usepackage{psfrag}
\usepackage{amsmath, amsfonts, amscd, amssymb, amsthm, enumerate}
\usepackage[all]{xy}
\usepackage[francais,english]{babel}
\usepackage[latin1]{inputenc}

\title{Equivariant K-theory for proper actions of non-compact Lie groups}
\author{Clément de Seguins Pazzis
\footnote{e-mail adress: dsp.prof@gmail.com}
\footnote{This work is part of the author's PhD thesis that he completed at the Institut Galil\'ee in Universit\'e Paris Nord,
99 avenue Jean-Baptiste Cl\'ement, 93430 Villetaneuse, FRANCE}}

\def\defterm{\textbf}
\def\N{\mathbb{N}}

\def\R{\mathbb{R}}
\def\C{\mathbb{C}}
\def\H{\mathbb{H}}
\def\calC{\mathcal{C}}
\def\calD{\mathcal{D}}
\def\GL{\text{GL}}
\newcommand{\End}{\operatorname{End}}
\newcommand{\Hom}{\operatorname{Hom}}

\newcommand{\vEc}{\operatorname{Vec}}
\newcommand{\VEct}{\operatorname{\mathbb{V}ect}}
\newcommand{\Func}{\operatorname{Func}}
\newcommand{\im}{\operatorname{Im}}
\newcommand{\tr}{\operatorname{Tr}}
\newcommand{\ktimes}{\operatorname{\underset{k}{\times}}}

\newcommand{\Ob}{\operatorname{Ob}}
\newcommand{\Mor}{\operatorname{Hom}}
\newcommand{\In}{\operatorname{In}}
\newcommand{\Fin}{\operatorname{Fin}}
\newcommand{\Id}{\operatorname{Id}}
\newcommand{\id}{\operatorname{id}}
\newcommand{\Comp}{\operatorname{Comp}}
\newcommand{\Rep}{\operatorname{Rep}}
\newcommand{\res}{\operatorname{res}}
\newcommand{\Prol}{\operatorname{ext}}
\newcommand{\Sk}{\operatorname{Sk}}

\newcommand{\Ker}{\operatorname{Ker}}
\newcommand{\sub}{\operatorname{sub}}

\newcommand{\Mod}{\operatorname{\text{-}mod}}
\newcommand{\Frame}{\operatorname{\text{-}frame}}
\newcommand{\Bdl}{\operatorname{\text{-}Bdl}}
\newcommand{\imod}{\operatorname{\text{-}imod}}
\newcommand{\iframe}{\operatorname{\text{-}iframe}}
\newcommand{\iBdl}{\operatorname{\text{-}iBdl}}
\newcommand{\smod}{\operatorname{\text{-}smod}}
\newcommand{\sframe}{\operatorname{\text{-}sframe}}
\newcommand{\sBdl}{\operatorname{\text{-}sBdl}}
\newcommand{\Sim}{\operatorname{GU}}

\newcommand{\colim}{\operatorname{colim}}
\newcommand{\Indlim}{\operatorname{\underset{\longrightarrow}{\colim}}}

\renewcommand{\setminus}{\smallsetminus}
\renewcommand{\epsilon}{\varepsilon}

\newcommand{\longleft}[1]{\;{\leftarrow%
    \count255=0 \loop \mathrel{\mkern-6mu}%
    \relbar\advance\count255 by1\ifnum\count255<#1\repeat}\;}
\newcommand{\longright}[1]{\;{\count255=0 \loop \relbar\mathrel{\mkern-6mu}%
    \advance\count255 by1\ifnum\count255<#1\repeat\rightarrow}\;}
\newcommand{\Right}[2]{\overset{#2}{\longright#1}}
\newcommand{\RIGHT}[3]{\mathrel{\mathop{\kern0pt\longright#1}
   \limits^{#2}_{#3}}}
\newcommand{\Left}[2]{{\buildrel #2 \over {\longleft#1}}}
\newcommand{\dRIGHT}[3]{\mathrel{%
   \mathop{\vcenter{\baselineskip=0pt\hbox{$\kern0pt\longright#1$}%
   \hbox{$\kern0pt\longright#1$}}}\limits^{#2}_{#3}}}
\newcommand{\LRIGHT}[3]{\mathrel{%
   \mathop{\vcenter{\baselineskip=0pt\hbox{$\kern0pt\longleft#1$}%
   \hbox{$\kern0pt\longright#1$}}}\limits^{#2}_{#3}}}

\theoremstyle{plain}
\newtheorem{theo}{Theorem}[section]
\newtheorem{prop}[theo]{Proposition}
\newtheorem{cor}[theo]{Corollary}
\newtheorem{lemme}[theo]{Lemma}

\newtheorem{atheo}{Theorem}[section]

\newtheorem{aprop}[atheo]{Proposition}

\theoremstyle{definition}
\newtheorem{Def}[theo]{Definition}

\theoremstyle{remark}
\newtheorem{Rems}{Remarks}
\newtheorem{Rem}[Rems]{Remark}

\begin{document}

\maketitle

\begin{abstract}
Generalizing a construction of Lück and Oliver \cite{Bob2}, we
define a good equivariant cohomology theory on the category of proper $G$-CW complexes
when $G$ is an arbitrary Lie group (possibly non-compact). This is done by constructing an appropriate
classifying space that arises from a $\Gamma-G$-space. It is proven that this theory
effectively generalizes Segal's equivariant $K$-theory when $G$ is compact.
\end{abstract}

\section{Introduction}

\subsection{The problem}

Topological K-theory is a generalized cohomology theory which was
developed by Atiyah and Hirzebruch in the early 1960's:
starting from a topological space $X$, one looks at the monoid
of isomorphism classes of (complex) finite-dimensional vector bundles bundles over $X$, and
after using a standard algebraic construction (the Grothendieck group associated to a commutative monoid),
one recovers an abelian group $\mathbb{K}(X)$, called the K-theory of $X$, and extends it to a functor $\mathbb{K}(-)$.

This was generalized to spaces by Segal \cite{SegalKtheory}
by replacing vector bundles with $G$-vector bundles.
It is now well known that his construction
gives rise to a good equivariant cohomology theory
in the case of actions of a \emph{compact} group on
compact spaces or on $G$-CW-complexes, and that Bott periodicity still holds.

However, it is still possible to \emph{define} the equivariant K-theory group $\mathbb{K}_G(X)$
whenever $X$ is a $G$-space. Some properties \cite{Bob1} of
$\mathbb{K}_G(-)$ still hold in the general case of a Lie group $G$
acting properly on $G$-CW-complexes  (e.g.
two $G$-homotopic maps have the same image by the functor
$\mathbb{K}_G(-)$). However, even in simple cases,
 $\mathbb{K}_G(-)$ is not a good cohomology theory (excision may fail, cf.\ \cite{Phillips} and \cite{Bob1}).

A first positive generalization to non-compact groups was given by Phillips using tools from function analysis \cite{Phillips}.
However, it would seem reasonable to define equivariant K-theory by means of homotopy theory.
In the following paper, we will show that a construction of Lück and Oliver that was featured in \cite{Bob2}
for discrete groups can be generalized to an arbitrary Lie group. In a
following paper, we will relate our equivariant cohomology theory to Phillips' \cite{Ktheo2}.

\subsection{Structure of the paper}

Let $G$ be a Lie group and $F$ denote one of the fields $\R$ or $\C$.
What we want to construct is a good equivariant cohomology theory,
which will be written $KF_G(-)$ for convenience, defined on a subcategory of the one of proper $G$-spaces
which contains at least the category of finite proper $G$-CW-complexes (see the next paragraph for the definitions).
For such a theory to deserve the label ``equivariant K-theory'', we impose a set of conditions.
First of all, we want to have product maps, i.e. natural homomorphisms
$KF_G(X) \otimes KF_H(Y) \longrightarrow KF_{G \times H}(X \times Y)$.
We also want to have Bott homomorphisms (depending on $F$) that are linear with respect to products, and we want to have Bott
periodicity. We finally want a connection between Segal's ``naive'' equivariant K-theory and our functors. More precisely,
we want to have a natural transformation $\mathbb{K}F_G(-) \longrightarrow KF_G(-)$ that
yields isomorphisms for the $G$-spaces of the type
$(G/H) \times Y$, where $H$ is a compact subgroup of $G$, and $Y$ is a reasonable space on which $G$ acts trivially
(say a compact space, a CW-complex, or a finite CW-complex). When $Y$ is a sphere,
we recover the equality of the so-called ``coefficients'' of our equivariant K-theory with those of Segal's.
The last condition is that the various natural transformations
$\mathbb{K}F_G(-) \longrightarrow KF_G(-)$ should be compatible with product maps,
and it will follow that they are compatible with the Bott homomorphisms.

The first step (Section \ref{3}) consists in constructing a classifying space
for the functor $\mathbb{V}\text{ect}_G^F(-)$ which maps every $G$-CW-complex $X$ to the
monoid of isomorphism classes of finite dimensional $G$-vector bundles over $X$.
In Section \ref{4}, we will construct a $\Gamma-G$-space
$\underline{\vEc}_G^{F,\infty}$ such that $\vEc_G^{F,\infty}$ has the equivariant homotopy type of that classifying space.
The $G$-space $KF_G^{[\infty]}:=\Omega B \vEc_G^{F,\infty}$ will then
be used in Section \ref{6} as a classifying space to define K-theory $KF^*_G(-)$
in negative degrees. Following Lück and Oliver, product structures and Bott homomorphisms
are constructed and then used to define equivariant $K$-theory in positive degrees.
In Section \ref{comparesec}, we will finally show that our construction generalizes both Segal's and Lück-Oliver's.

During the course of the construction, we will also define two other classifying spaces along the way:
one that is naturally suited for finite-dimensional $G$-Hilbert bundles
and the other for finite-dimensional $G$-simi-Hilbert bundles (i.e. vector bundle with
an added structure that is related to the group of similarities). They will be needed in \cite{Ktheo2} to
relate our equivariant K-theory with Phillips'.

A final word: this work features two theorems with very technical proofs.
So as not to distract the reader, we have relegated those proofs in Sections \ref{Bundleproof} and \ref{Proofhomotype} of the appendix.
Section \ref{C} is meant to set things straight on a common misconception on $\Gamma$-spaces.

\subsection{The main framework}

\subsubsection{$G$-CW-complexes}

A $G$-space $X$ is called a $G$\defterm{-CW-complex} when it is obtained as the direct limit of a
sequence $(X_{(n)})_{n \in \mathbb{N}}$ of subspaces for which there exists, for every  $n \in \mathbb{N}$,
a set $I_n$, a family $(H_i)_{i \in I_n}$ of closed subgroups of $G$ and a push-out square
$$\begin{CD}
\underset{i \in I_n}{\coprod} (G/H_i) \times S^{n-1} @>>> \underset{i \in I_n}{\coprod} (G/H_i) \times D^n \\
@VVV @VVV \\
X_{(n-1)} @>>> X_{(n)}
\end{CD}$$ in the category of $G$-spaces
(where we have a trivial action of $G$ on both the $(n-1)$-sphere $S^{n-1}$ and the closed $n$-disk $D^n$), with the convention that
$X_{-1}=\emptyset$. The spaces $(G/H_i) \times \overset{\circ}{D^n}$
are called the equivariant cells (or $G$-cells) of $X$.
A $G$-CW-complex is \defterm{proper} when all its isotropy subgroups are compact,
i.e. all the groups $H_i$ in the preceding description are compact.
Relative $G$-CW-complexes are defined accordingly. However, by a \defterm{proper} relative $G$-CW-complex,
we mean a relative $G$-CW-complex $(X,A)$ such that $X \setminus A$ is a proper $G$-space
(whereas $X$ itself may not be proper).

Given a topological group $G$,
a pair of $G$-spaces $(X,A)$ is said to be a \defterm{$G$-CW-pair} when $A$ and $X$ are $G$-CW-complexes
and $(X,A)$ is a relative $G$-CW-complex.
A $G$-CW-pair $(X,A)$ is said to be \defterm{proper} when $X$ is a proper $G$-space
(i.e. its isotropy groups are compact subgroups of $G$).

A \defterm{pointed proper $G$-CW-complex} is a relative $G$-CW-complex $(X,*)$, with $*$ a point, such that the $G$-space $X \setminus *$ is proper.
Notice that whenever $(X,A)$ is a relative $G$-CW-complex such that $X \setminus A$ is proper, the $G$-space $X/A$
inherits a natural structure of pointed proper $G$-CW-complex.

\subsubsection{$G$-fibre bundles}

Let $G$ be a topological group. Given
a $G$-space $X$, we call \defterm{pseudo-$G$-vector bundle}
(resp.\ $G$-\defterm{vector bundle}) over $X$ the data consisting of
a pseudo-vector bundle (resp.\ a vector bundle)
$p : E \rightarrow X$ over $X$ and of a (left) $G$-action on $E$, such that
$p$ is a $G$-map, and, for all $g \in G$ and $x \in X$, the map $E_x \rightarrow E_{g.x}$
induced by the $G$-action on $E$ is a linear isomorphism.

Given an integer $n \in \N$ and a $G$-space $X$,
$\mathbb{V}\text{ect}_G^{F,n}(X)$ will denote the set of isomorphism classes of
$n$-dimensional $G$-vector bundles over $X$. Accordingly,
$\mathbb{V}\text{ect}_G^{F}(X)$ will denote the abelian monoid
of isomorphism classes of finite-dimensional $G$-vector bundles over $X$. \label{Gvectorbundle}

Given another topological group $H$, a \defterm{$(G,H)$-principal bundle} is
an $H$-principal bundle $\pi : E \rightarrow X$ with structures of $G$-spaces on $E$ and $X$
for which $\pi$ is a $G$-map and $\forall (g,h,x)\in G \times H \times E, \; g.(x.h)=(g.x).h$.
Notice then that we recover a structure of $(G \times H^{\text{op}})$-space on $E$
for which the isotropy subgroups are closed subgroup which intersect $\{1\} \times H^{\text{op}}$ trivially.

\subsubsection{The category of compactly-generated $G$-spaces}

Let $G$ be a topological group.
A \defterm{$G$-pointed k-space} consists of a $G$-space which is a k-space (i.e. compactly-generated and Hausdorff\footnote{For Section \ref{7}, and in general whenever smash products are concerned, one should loosen up this definition
and define k-spaces as topological spaces that are compactly-generated and \emph{weak-Hausdorff}. The reader will check this
bears no additional complexity.})
together with a point in it which is fixed by the action of $G$.

The category $CG_G^{\bullet}$ is the one whose objects are the
$G$-pointed k-spaces and whose morphisms are the pointed $G$-maps.
The category $CG_G^{h\bullet}$ is the category with the same objects as
$CG_G^{\bullet}$, and whose morphisms are the equivariant pointed homotopy classes
of $G$-maps between objects (i.e. $CG_G^{h\bullet}$ it is the homotopy category of $CG_G^{\bullet})$.
Given two $G$-spaces (resp. two pointed $G$-spaces) $X$ and $Y$, we let $[X,Y]_G$ (resp. $[X,Y]_G^\bullet$)
denote the set of equivariant homotopy classes of $G$-maps (resp. pointed $G$-maps) from $X$ to $Y$.

Let $f: X \rightarrow Y$ be a morphism in $CG_G$ and $\mathcal{F}$
be a set of subgroups of $G$. We say that $f$ is an \defterm{$\mathcal{F}$-weak equivalence}
when the restriction $f^H:X^{H} \rightarrow Y^{H}$
is a weak equivalence
for every $H \in \mathcal{F}$. We say that $f$ is a \defterm{$G$-weak equivalence} when $f$ is a $\mathcal{K}$-weak equivalence
for the set $\mathcal{K}$ of all \emph{compact} subgroups of $G$.
Given a set $\mathcal{F}$ of subgroups of $G$, every morphism
that is equivariantly homotopic to an $\mathcal{F}$-weak equivalence is itself an $\mathcal{F}$-weak equivalence.

We finally define $W_G$ as the class of morphisms in $CG_G^{h\bullet}$
which have $G$-weak equivalences as representative maps.
We can then consider the category of fractions $CG_G^{h\bullet}[W_G^{-1}]$, with
its usual universal property. The following properties are then folklore and will be used throughout the paper
(see \cite{Bob2} for proofs):

\begin{prop}\label{extentionlemma2}
Let $G$ be a Lie group, $\mathcal{F}$ a family of subgroups of $G$ stable by conjugation
and $Y \overset{f}{\rightarrow} Y'$ an $\mathcal{F}$-weak equivalence.
Then, for every relative $G$-CW-complex $(X,A)$ such that all the isotropy subgroups of $X \setminus A$
belong to $\mathcal{F}$, and for every commutative square
$$\begin{CD}
A @>{\alpha _0}>> Y \\
@V{i}VV @V{f}VV \\
X @>{\alpha}>> Y'
\end{CD}$$ in $CG_G$, there exists a $G$-map
$\tilde{\alpha}:X \rightarrow Y$ such that $\tilde{\alpha} \circ i=\alpha _0$ and
$[f \circ \tilde{\alpha}] =[\alpha]$ in $CG^h_G$. Moreover, this map is unique up to an
equivariant homotopy rel.\ $A$.
\end{prop}

\begin{prop}\label{pointedweakeq}
Let $Y \overset{f}{\rightarrow} Y'$ be a $G$-weak equivalence between
pointed $G$-spaces. \\
Then, for every proper pointed $G$-CW-complex $X$, the map $f$ induces a bijection
$$f_*: [X,Y]_G^\bullet \longrightarrow [X,Y']_{G.}^\bullet$$
\end{prop}

\begin{cor}\label{weakeq}
Let $\mathcal{F}$ be a class of subgroups of $G$, and
$Y \overset{f}{\rightarrow} Y'$ an $\mathcal{F}$-weak equivalence between
$G$-spaces. Then, for every $G$-CW-complex $X$ whose isotropy subgroups all belong to $\mathcal{F}$,
the map $f$ induces a bijection:
$$f_*: [X,Y]_G \longrightarrow [X,Y']_G.$$
\end{cor}

\begin{cor}
For every proper pointed $G$-CW-complex $X$, the functor
$$F_X: \begin{cases} CG_G^{h\bullet} & \longrightarrow \text{Sets} \\
Y & \longmapsto [X,Y]_G^\bullet
\end{cases}$$
factorizes through
$$\xymatrix{CG_G^{h\bullet} \ar[r] \ar[rd]_{F_X} & CG_G^{h\bullet}[W_G^{-1}] \ar@{.>}[d]^{\tilde{F}_X} \\
& \text{Sets}
}.$$
\end{cor}

\begin{Rem}
We may replace $CG_G^{h\bullet}$ by the category $\text{H-G}^{h\bullet}$
whose objects are compactly-generated $G$-spaces which have an equivariant H-space structure,
and the morphisms are equivariant pointed homotopy classes
of continuous morphisms of equivariant H-spaces,
i.e. $X \overset{f}{\rightarrow} Y$ is such a morphism if and only
if it is continuous, equivariant, pointed, and the square
$$\begin{CD}
X \times X @>{f \times f}>> Y \times Y \\
@V{\times _X}VV @V{\times _Y}VV \\
X @>{f}>> Y
\end{CD}$$ is commutative in $CG_G^{h\bullet}$. Let $W_{\text{H-G}}$ denote
the class of morphisms in
$\text{H-G}^{h\bullet}$ which are $G$-weak equivalences.
Then, for every pointed $G$-CW-complex $X$, we recover
a functor
$$G_X:\begin{cases}
\text{H-G}^{h\bullet} & \longrightarrow \mathcal{G}r \\
Y & \longmapsto [X,Y]_{G.}^\bullet
\end{cases}
$$
\end{Rem}

\begin{cor}
For every pointed proper $G$-CW-complex $X$, the functor $G_X$
factorizes through:
$$\xymatrix{ \text{H-G}^{h\bullet} \ar[r] \ar[rd]_{G_X} &
\text{H-G}^{h\bullet}[W_{\text{H-G}}^{-1}] \ar@{.>}[d]^{\tilde{G_X}}\\
& \mathcal{G}r.
}$$
\end{cor}

\subsubsection{$\Gamma$-spaces}

The simplicial category is denoted by $\Delta$ (cf.\ \cite{G-Z}).
Recall that the category $\Gamma$ (see \cite{Segal-cat}) has
the finite sets as objects, a morphism from $S$ to $T$ being a map from
$\mathcal{P}(S)$ to $\mathcal{P}(T)$ which preserves disjoint unions (with obvious composition of morphisms);
this is equivalent to having a map $f$ from $S$ to $\mathcal{P}(T)$
such that $f(s) \cap f(s') =\emptyset$ whenever $s \neq s'$.

For every $n \in \N$, we set $\mathbf{n}:=\{1,\dots,n\}$ and $[n]:=\{0,\dots,n\}$.
Recall the canonical functor $\Delta \rightarrow \Gamma$ obtained
by mapping $[n]$ to $\mathbf{n}$ and the morphism $\delta: [n] \rightarrow [m]$
to $$\begin{cases}
\mathbf{n} & \longrightarrow \mathcal{P}(\mathbf{m}) \\
k & \longmapsto \{j \in \N: \bigl\{\delta(k-1) < j \leq \delta(k)\bigr\}.
\end{cases}$$

By a $\Gamma$-space, we mean a
\emph{contravariant} functor $\underline{A}: \Gamma \rightarrow \text{CG}$ such that $\underline{A}(\mathbf{0})$ is a well-pointed contractible space.
The space $\underline{A}(\mathbf{1})$ is then simply denoted by $A$.
We say that $\underline{A}$ is a  \defterm{good} $\Gamma$-space when, in addition, for all
$n \in \mathbb{N}^*$, the continuous map $\underline{A}(\mathbf{n}) \rightarrow \underset{i=1}{\overset{n}{\prod}} A$, induced by all morphisms
$\mathbf{1} \rightarrow \mathbf{n}$ which map $1$ to $\{i\}$, is a homotopy equivalence.
From now on, when we talk of $\Gamma$-spaces, we will actually mean good $\Gamma$-spaces.

When $\underline{A}$ is a $\Gamma$-space, composition with the previously defined functor
$\Delta \rightarrow \Gamma$ yields a simplicial space, which we still write $\underline{A}$,
and we can take its thick geometric realization (as defined in appendix A of \cite{Segal-cat}), which we write $BA$.
Since $\underline{A}(\mathbf{0})$ is well-pointed and contractible, we have a map $A \rightarrow \Omega B\underline{A}$
that is ``canonical up to homotopy''. Recall that we have an H-space structure on
$A$ by composing the map $\underline{A}(\mathbf{2}) \rightarrow A$ induced by
$\begin{cases}
\{1\} & \rightarrow \mathcal{P}(\mathbf{2}) \\
1 & \mapsto \{1,2\}
\end{cases}$ and a homotopy inverse of the map $\underline{A}(\mathbf{2}) \rightarrow A \times A$ mentioned earlier.
Under suitable assumptions on $A$, one may prove (cf.\ § 4 of \cite{Segal-cat}) that the map
$A \rightarrow \Omega BA$ is in some sense the ``group completion'' of the H-space $A$.

Here, we will be dealing with equivariant $\Gamma$-spaces, or \defterm{$\Gamma-G$-spaces}:
given a topological group $G$, a $\Gamma-G$-space is a contravariant functor
$\underline{A}: \Gamma \rightarrow \text{CG}_G$ such that:
\begin{enumerate}[(i)]
\item $\underline{A}(\mathbf{0})$ is equivariantly well-pointed and equivariantly contractible;
\item For any $n \in \N^*$, the canonical map $\underline{A}(\mathbf{n}) \rightarrow \underset{i=1}{\overset{n}{\prod}} A$
is an equivariant homotopy equivalence.
\end{enumerate}
Notice that this definition should be more constraining than the one featured in \cite{Bob2}.
Moreover, whenever $H$ is a subgroup of $G$ and $\underline{A}$ is a $\Gamma-G$-space,
the contravariant functor $\underline{A}^H: \Gamma \rightarrow CG$ obtained by restricting $\underline{A}$ to
the fixed point sets for $H$ is in fact a $\Gamma$-space.
When $\underline{A}$ is a $\Gamma-G$-space, we may define as before
a $G$-map $A \rightarrow \Omega BA$ which is ``canonical up to homotopy'' and is a ``group completion''
of the equivariant H-space $A$.

\subsection{Additional definitions and notation}

We set $\mathbb{R}_+:=\{ t \in \mathbb{R}: t \geq 0\}$ and
$\mathbb{R}_+^*:=\{ t \in \mathbb{R}: t > 0\}$. The standard segment is denoted by $I:=[0,1]$.

\paragraph{}
By $F$, we will always denote one of the fields $\R$, $\C$ or $\H$.
Given two vector spaces $E$ and $E'$ over $F$, $L(E,E')$ will denote the
set of linear maps from $E$ to $E'$.
When $n \in \N$ and $F=\R$ or $\C$,
we denote by $U_n(F)$ the unitary subgroup of $\GL_n(F)$, and by $\Sim_n(F)=F^*\,U_n(F)$ the subgroup of similarities of $F^n$.

Let $k \in \N^*$ and $F=\R$ or $\C$. Then $F^k$ comes with a canonical structure
of Hilbert space. We have a canonical sequence of isometries $F^1 \hookrightarrow
F^2 \hookrightarrow \dots \hookrightarrow F^k \hookrightarrow F^{k+1} \hookrightarrow \dots$, where
$F^k \hookrightarrow F^{k+1}$ maps $(x_1,\dots,x_k)$ to $(x_1,\dots,x_k,0)$.
We let $F^{(\infty)}$ denote the corresponding direct limit $\underset{\underset{k \in \mathbb{N}^*}{\longrightarrow}}{\lim}F^k$,
with its natural structure of topological vector space, and its natural structure of
inner product space. Notice that $F^{(\infty)}$ actually corresponds to the space
$F^\infty$ in \cite{Bob2}.

When $\mathcal{H}$ is an inner product space, and $n \in \mathbb{N}^*$,
$B_n(\mathcal{H}) \subset \mathcal{H}^n$
will denote the space of linearly independent
$n$-tuples of elements of $\mathcal{H}$ (with the convention $B_0(\mathcal{H})=*$), while
$V_n(\mathcal{H}) \subset \mathcal{H}^n$ will denote the space of orthonormal $n$-tuples
of elements of $\mathcal{H}$ (with the convention $V_0(\mathcal{H})=*$).
We also let $B_{\mathcal{H}}$ denote the unit ball of $\mathcal{H}$, $\sub(\mathcal{H})$ denote the
set of closed linear subspaces of $\mathcal{H}$, and, for $n \in \mathbb{N}, \sub_n(\mathcal{H})$
the set of $n$-dimensional linear subspaces of $\mathcal{H}$.

\paragraph{}
When $X$ and $Y$ are compactly-generated spaces, $Y^X$ will denote the
space of continuous maps from $X$ to $Y$, with the k-space topology associated to the compact-open topology.
When a group $G$ acts on $X$ and $H$ is a subgroup of $G$, $X^H$ standardly denotes the subspace of $X$ consisting
of the points of $X$ that are fixed by $H$. We are well aware of the possible conflicts
but the context will always help the reader avoid them.

When we have two morphisms $X \rightarrow Z$ and $Y \rightarrow Z$
of topological spaces, $X \underset{Z}{\times} Y$ will denote the limit of the diagram
$X \rightarrow Z \leftarrow Y$, and, when no confusion is possible, we will
write $X \underset{\vartriangle}{\times} Y$ instead of $X \underset{Z}{\times} Y$.

\paragraph{}
If $\mathcal{C}$ is a small category, then:
\begin{itemize}
\item $\Ob(\mathcal{C})$ (resp. $\Mor(\mathcal{C})$) will denote its set of objects (resp. of morphisms);
\item The structural maps of $\calC$ i.e. the initial, final, identity and composition maps are respectively denoted by
$$\In_{\mathcal{C}}: \Mor(\mathcal{C}) \rightarrow \Ob(\mathcal{C}) \quad ; \quad
\Fin_{\mathcal{C}}: \Mor(\mathcal{C}) \rightarrow \Ob(\mathcal{C});$$
$$\Id_{\mathcal{C}}: \Ob(\mathcal{C}) \rightarrow \Mor(\mathcal{C}) \quad \text{and} \quad
\Comp_{\mathcal{C}}: \Mor(\mathcal{C}) \underset{\vartriangle}{\times} \Mor(\mathcal{C}) \rightarrow \Mor(\mathcal{C});$$
\item The nerve of $\calC$ is denoted by $\mathcal{N}(\calC)$, whilst
$\mathcal{N}(\calC)_m$ will denote its $m$-th component for any $m\in \N$.
\end{itemize}
By a \defterm{k-category}, we mean a small category with k-space topologies on the sets
of objects and spaces, such that the structural maps induce continuous maps in the category of k-spaces.
To every topological category $\calC$, we assign a k-category whose space of objects and space of morphisms are
respectively $\Ob(\calC)_{(k)}$ and $\Mor(\calC)_{(k)}$.

To a k-category, we may assign its \emph{nerve} in the category of k-spaces,
and then take one of the two geometric realizations $\| \quad \|$ (the ``thick realization") or $| \quad |$
(the ``thin realization") of it in the category of k-spaces (see \cite{Segal-cat}).
When $\calC$ and $\calD$ are two k-categories, we may define another category, denoted by
$\Func(\calC,\calD)$, whose objects are the topological
functors from $\calC$ to $\calD$, and
whose morphisms are the continuous natural transformations between continuous functors from
$\calC$ to $\calD$. The structural maps are obvious.
The set $\Ob(\Func(\calC,\calD))$
is a subset of $\Mor(\calD)^{\Mor(\calC)}$, and is given
the topology induced by the k-space topology
of $\Mor(\mathcal{D})^{\Mor(\mathcal{C})}$. The set $\Mor(\Func(\mathcal{C},\mathcal{D}))$
is a closed subset of the product
$\Ob(\Func(\mathcal{C},\mathcal{D})) \ktimes \Ob(\Func(\mathcal{C},\mathcal{D}))
\ktimes \Mor(\mathcal{D})^{\Ob(\mathcal{C})}$ and is equipped with the topology induced
by the product topology in the category of k-spaces.
From the properties of k-spaces (more precisely the adjunction homeomorphisms), it is easy to check that the structural maps
of $\Func(\mathcal{C},\mathcal{D})$
are continuous, hence $\Func(\mathcal{C},\mathcal{D})$ is equipped with a structure of k-category.

When $M$ is a topological monoid, we let
$\mathcal{B}M$ denote the $k$-category with one object $*$, such that
$\Mor(*,*)=M$ as a topological space, with $x \circ y=x.y$ for any $(x,y)\in M^2$.

Top will denote the category of topological spaces, TopCat the one
of topological categories, and kCat the one of k-categories.

\paragraph{}
Given a topological space $X$, we will occasionally use
$X$ to denote the topological category with $X$ as space of objects \emph{and} space of morphisms
and only identity morphisms.

\paragraph{List of important notation} ${}$
\vskip 1mm
\def\8[#1;#2]{\hbox to \textwidth{#1 \hfill p. #2}\vskip2truemm}
\8[$\VEct_G^F(X)$;\pageref{3}]
\8[$\varphi \Frame$;\pageref{defframe}]
\8[$\varphi \Mod$;\pageref{defmod}]
\8[$\varphi \Bdl$;\pageref{defBdl}]
\8[$\vEc_G^\varphi=|\Func(\mathcal{E} G,\varphi \Mod)|$;\pageref{defvec}]
\8[$\widetilde{\vEc}_G^\varphi=|\Func(\mathcal{E} G,\varphi \Frame)|$;\pageref{defvec}]
\8[$E\vEc_G^\varphi=|\Func(\mathcal{E} G,\varphi \Bdl)|$;\pageref{defvec}]
\8[$\theta: \tilde{E} \times F^n  \longrightarrow E$;\pageref{thetapage}]
\8[$\gamma_n^k(F):E_n(F^k)\to G_n(F^k)$ (universal vector bundle);\pageref{3.4.2}]
\8[$\widetilde\gamma_n^k(F):B_n(F^k)\to G_n(F^k)$ (universal frame bundle);\pageref{3.4.2}]
\8[$\gamma_n(F):E_n(F^{(\infty)})\to G_n(F^{(\infty)})$,
$\widetilde\gamma_n(F):B_n(F^{(\infty)})\to G_n(F^{(\infty)})$;\pageref{3.4.2}]
\8[$\gamma^{(m)}(F)=\underset{n \in \mathbb{N}}{\coprod}\gamma_n^{mn}(F)$,
$\gamma(F)=\underset{n \in \mathbb{N}}{\coprod}\gamma_n(F)$;\pageref{3.4.2}]
\8[$i\VEct_G^F(X),s\VEct_G^F(X)$;\pageref{hilbertbundles}]
\8[$i\vEc_G^\varphi$, $i\widetilde{\vEc}_G^\varphi$, $Ei\vEc_G^\varphi$,
$s\vEc_G^\varphi$, $s\widetilde{\vEc}_G^\varphi$, $Es\vEc_G^\varphi$;\pageref{hilbertdef}]
\8[$\Gamma\text{-Fib}_F$;\pageref{4.1.1}]
\8[$\mathcal{O}_\Gamma^F:\Gamma\text{-Fib}_F\rightarrow \Gamma$;\pageref{4.1.2}]
\8[$\Mod, \imod, \smod: \Gamma \text{-Fib}_F\rightarrow \text{kCat}$;\pageref{4.2.2}]
\8[$\Bdl, \iBdl, \sBdl: \Gamma \text{-Fib}_F\rightarrow \text{kCat}$;\pageref{4.2.5}]
\8[Hilbert $\Gamma$-bundles $\varphi:\Gamma\rightarrow \Gamma \text{-Fib}_F$;\pageref{4.4}]
\8[$\underline{\vEc}_G^\varphi:\quad
\underline{\vEc}_G^\varphi(S)=\vEc_G^{\varphi(S)}.\qquad
\underline{i\vEc}_G^\varphi,\ \underline{s\vEc}_G^\varphi$;\pageref{gammaspaces}]
\8[$\text{Fib}^{\mathcal{H}}(S)$;\pageref{4.5}]
\8[$\vEc_G^{F,m}=\vEc_G^{\text{Fib}^{F^m}(\mathbf{1})} =
\underline{\vEc}_G^{\text{Fib}^{F^m}}(\mathbf{1}), \textup{ etc.}$;\pageref{deftypique}]
\8[$KF_G^{[m]}=\Omega B\vEc_G^{F,m}$, $iKF_G^{[m]}$, $sKF_G^{[m]}$;\pageref{6}]
\8[$\gamma: \mathbb{K}F_G(-) \longrightarrow KF_G(-)$;\pageref{6.2}]
\8[$KF_G^\varphi=\Omega B\vEc_G^\varphi$ ($\varphi$ a Hilbert
$\Gamma$-bundle);\pageref{7.2}]

\section{Classifying spaces for $\mathbb{V}\text{ect}_G^{F,n}(-)$}\label{3}

In Section \ref{3.1}, we introduce the notion of a perfect k-categories, which will be ubiquitous in the rest of the section.
In Section \ref{3.2}, we start from an $n$-dimensional
vector bundle $\varphi$, and use it to construct various categories. These categories are then used to produce, under suitable conditions on $\varphi$, a universal $n$-dimensional $G$-vector bundle $E\vEc_G^\varphi  \rightarrow \vEc_G^\varphi$ for $G$-CW-complexes, when $G$ is a Lie group.
Suitable $\varphi$'s are produced in the end of Section \ref{3.4}.
In the end, we consider similar results when $G$-vector bundles are replaced by finite-dimensional
$G$-Hilbert bundles or $G$-simi-Hilbert bundles.

\subsection{Perfect k-categories}\label{3.1}

\begin{Def}A non-empty k-category $\mathcal{C}$ is called \defterm{perfect}
when the structural map
$$\begin{cases}
\Mor(\mathcal{C}) &  \longrightarrow \Ob(\mathcal{C}) \ktimes \Ob(\mathcal{C}) \\
f & \longmapsto (\In (f),\Fin (f))
\end{cases}$$
is a homeomorphism.
\end{Def}

Let $X$ be a k-space. We then let $\mathcal{E}X$ denote the perfect k-category with $X$ as space of objects,
$X \ktimes X$ as space of morphisms, and obvious structural maps.
Clearly, $\mathcal{E}X$ is a perfect k-category.
This yields an equivalence of categories $\mathcal{E}: \text{CG} \rightarrow \text{kCat}$.
The following lemmas are easy to check.

\begin{lemme}\label{perfectsub}
Let $\mathcal{C}$ be a perfect category.
Then every full, non-empty, closed subcategory of $\mathcal{C}$
is perfect.
\end{lemme}

\begin{lemme}\label{contractibility}
Let $\mathcal{C}$ be a perfect category.
Then its geometric realization $|\mathcal{C}|$ is contractible.
\end{lemme}

\begin{proof}
This is proven in the same way as the contractibility of the geometric realization of a
(discrete) category which has an initial object. Extra care needs to be taken
about the continuity of the involved maps, but this is straightforward.
\end{proof}

\begin{prop}
Let $\mathcal{C}$ be a non-empty k-category, and $\mathcal{D}$
a perfect k-category. Then the k-category $\Func(\mathcal{C},\mathcal{D})$ is perfect.
\end{prop}

\begin{proof}
Since $\Func(\mathcal{C},\mathcal{D})$ is non-empty, we need to prove that the structural map
$$\alpha: \Mor(\Func(\mathcal{C},\mathcal{D})) \longrightarrow
\Ob(\Func(\mathcal{C},\mathcal{D})) \ktimes
\Ob(\Func(\mathcal{C},\mathcal{D}))$$
is a homeomorphism. Injectivity is straightforward.
In order to prove surjectivity,
let $f$ and $g$  be two objects in $\Func(\mathcal{C},\mathcal{D})$.
The map
$(f \rightarrow g): \begin{cases}
\Ob (\mathcal{C}) & \longrightarrow \Mor(\mathcal{D}) \\
x & \longmapsto (f(x),g(x))
\end{cases}$ is continuous, since $\mathcal{D}$ is a perfect k-category.
Then $(f,g,f \rightarrow g)$ is a morphism from $f$ to $g$, because all diagrams are commutative in any perfect k-category.
It follows that $\alpha$ is a bijection.
It thus suffices to prove the continuity of its inverse, and this is a consequence of the
continuity of the composite map:

$$\Mor(\mathcal{D})^{\Mor(\mathcal{C})} \ktimes
\Mor(\mathcal{D})^{\Mor(\mathcal{C})} \overset{i \times i}{\longrightarrow}
\Ob(\mathcal{D})^{\Ob(\mathcal{C})} \ktimes
\Ob(\mathcal{D})^{\Ob(\mathcal{C})} \longrightarrow
(\Ob(\mathcal{D})\ktimes \Ob(\mathcal{D}))^{\Ob(\mathcal{C})},
$$
where $i$ is given by right composition by $\Id_{\mathcal{C}}$ and left composition by $\In _{\mathcal{D}}$.
The usual properties of k-spaces \cite{Steenrod} show that those maps are continuous.
\end{proof}

\subsection{Topological categories associated to vector bundles}\label{3.2}

We fix an integer $n \in \mathbb{N}$ for the rest of the section.
Let $\varphi: E \rightarrow X$ be an $n$-dimensional vector bundle over $F$, and $\tilde{\varphi}: \tilde{E} \rightarrow X$
the $\GL_n(F)$-principal bundle canonically associated to it (by considering $\tilde{E}$ as a subspace of $E^{\oplus n}$).
We will assume that $X$ is a locally-countable CW-complex. Hence, $\tilde{E}$, $E$ and $X$
are k-spaces, and any finite cartesian product of copies of them also is.

\label{defframe} We set
$$\varphi\Frame:=\mathcal{E}\tilde{E} .$$
The (right-)action of $\GL_n(F)$ on $\tilde{E}$ induces the diagonal action of $\GL_n(F)$
on $\Mor(\varphi\Frame)=\tilde{E} \times \tilde{E}$.
Therefore $\GL_n(F)$ acts on the category $\varphi\Frame$. \label{defmod}

We set $$\varphi\Mod:=\varphi\Frame/\GL_n(F),$$
which is obviously a k-category with $\Ob(\varphi\Mod)\cong X$.
An object of $\varphi\Mod$ corresponds to a point of $X$, and
a morphism from $x$ to $y$ (with $(x,y) \in X^2$) corresponds to
a linear isomorphism $E_x \overset{\cong}{\rightarrow} E_y$. The composite of
two morphisms $x \overset{f}{\rightarrow} y$ and $y \overset{g}{\rightarrow} z$ then
corresponds to the composite of the corresponding linear isomorphisms $E_x \overset{f}{\rightarrow} E_y$ and
$E_y \overset{g}{\rightarrow} E_z$.

Finally, we define $\varphi\Bdl$ as the category whose space of objects is $E$,
and whose space of morphisms
is the (closed) subspace of
$E \ktimes E \ktimes \Mor(\varphi\Mod)$
consisting of those triples $(e,e',f)$ such that
$\varphi(e) \underset{f}{\longrightarrow} \varphi(e')$ and $f(e)=e'$;
the structural maps are obvious. Again, this is clearly a k-category.
 \label{defBdl}

\begin{Rem}When $X=*$ and $\varphi$ is the canonical vector bundle
$F^n \rightarrow *$,
it is clear from the definitions that
$\varphi \Frame \cong  \mathcal{E}\GL_n(F)$, and
$\varphi \Mod \cong \mathcal{B}\GL_n(F)$.
\end{Rem}

The definition of $\varphi \Mod$ clearly yields a factor
$\varphi \Frame \rightarrow \varphi \Mod$. On the other hand, we have a functor of k-categories $\varphi \Bdl
\rightarrow \varphi \Mod$ defined as $\varphi: E \rightarrow X$ on the spaces of objects
and as the third projection $\Mor(\varphi \Bdl) \rightarrow \Mor(\varphi \Mod)$
on the spaces of morphisms.

We have a right-action of $\GL_n(F)$ on $\varphi \Frame$.
Identifying $F^n$ with the $k$-category whose space of objects and space of
morphisms are $F^n$ (with identities as the only morphisms),
there is a (unique) continuous functor
$\bigl(\varphi \Frame\bigr) \times F^n \rightarrow \varphi \Bdl$
\label{thetapage}
whose restriction to objects is:
$$\theta: \begin{cases}
\tilde{E} \times F^n &  \longrightarrow E \\
\bigl((e_i)_{1\leq i \leq n},(\lambda_i)_{1\leq i \leq n}\bigr) & \longmapsto \underset{i=1}{\overset{n}{\sum}} \lambda_i.e_i,
\end{cases}
$$
and which makes the square
$$\begin{CD}
\varphi \Frame \times F^n @>>> \varphi \Bdl \\
@VVV @VVV \\
\varphi \Mod \times F^n @>{\pi_1}>> \varphi \Mod
\end{CD}$$
commutative ($\pi_1$ denotes the projection onto the first factor).

\subsection{Universal $G$-vector bundles}\label{3.3}

\begin{Def}\label{defvec}
Given a Lie group $G$ and an $n$-dimensional vector bundle $\varphi$ over the field $F$,
we set
$$\vEc_G^\varphi:=\big|\Func(\mathcal{E}G,\varphi \Mod)\big|; \quad
\widetilde{\vEc}_G^\varphi :=\big|\Func(\mathcal{E}G,\varphi
\Frame)\big| \quad \text{and} \quad
E\vEc_G^\varphi :=\big|\Func(\mathcal{E}G,\varphi \Bdl)\big|.$$
\end{Def}

Notice that right-multiplication on the objects of $\mathcal{E}G$ induces a right-action of $G$ on $\mathcal{E}G$.
The three functors $\varphi \Frame \longrightarrow \varphi \Mod$, $\varphi \Bdl \longrightarrow \varphi \Mod$
and $\varphi \Frame \times F^n \longrightarrow \varphi \Bdl$ thus induce, by composition,
equivariant continuous functors
$$\Func (\mathcal{E}G,\varphi \Frame) \longrightarrow \Func (\mathcal{E}G,\varphi \Mod), \quad
\Func (\mathcal{E}G,\varphi \Bdl) \longrightarrow \Func (\mathcal{E}G,\varphi \Mod),$$
and
$$\Func (\mathcal{E}G,\varphi \Frame) \times F^n \longrightarrow \Func (\mathcal{E}G,\varphi \Bdl).$$
Taking geometric realizations everywhere yields three $G$-maps\footnote{Local-compactedness of $G$ is actually needed here to ensure that the geometric realizations considered here are really $G$-spaces.}
whose properties are summed up in the next theorem:

\begin{theo}\label{universal}
Let $X$ be a locally-countable CW-complex,
$\varphi: E \rightarrow X$ be an $n$-dimensional vector bundle over $X$, and
$G$ be a Lie group.
Then: \begin{enumerate}[(i)]
\item $\widetilde{\vEc}_G^\varphi \longrightarrow \vEc_G^\varphi$ is a $(G,\GL_n(F))$-principal bundle;
\item $E\vEc_G^\varphi \longrightarrow \vEc_G^\varphi$ is an $n$-dimensional $G$-vector bundle;
\item The canonical map
$$\widetilde{\vEc}_G^\varphi \times _{\GL_n(F)} F^n \longrightarrow
E\vEc_G^\varphi$$
is an isomorphism of $G$-vector bundles over
$\vEc_G^\varphi$.
\end{enumerate}
\end{theo}

\begin{Rem}
When $G=\{1\}$ and $X=*$, we recover the usual universal $n$-dimensional vector
bundle $E\GL_n(F) \times _{\GL_n(F)} F^n \longrightarrow B\GL_n(F)$.
\end{Rem}

In \cite{Bob2}, Theorem \ref{universal} was claimed to be true with no proof on
why the involved maps should be fibre bundles rather than just pseudo-fiber bundles
(in there, only the case $X$ is discrete was actually considered). However, the proof of this
is long, tedious, and definitely non-trivial. The details however are not necessary to understand the rest of
the paper, so we will wait until Section \ref{Bundleproof} of the appendix to give a complete proof.

\subsection{Classifying spaces for $\mathbb{V}\text{ect}^{F,n}_G(-)$}\label{3.4}

\subsubsection{General results}\label{3.4.1}

In the upcoming Proposition \ref{classifyingspace}, we will see which additional requirements on $G$ and $\varphi$
are sufficient to ensure that the
$G$-vector bundle $E\vEc_G^\varphi \longrightarrow \vEc_G^\varphi$ is \emph{universal}.
Let us dig deeper first into the topology of $\widetilde{\vEc}_G^\varphi$. Recall that $\GL_n(F)$ acts freely on
$\widetilde{\vEc}_G^\varphi$ by a right-action
that is compatible with the left-action of $G$.
Thus $G \times \GL_n(F)$ acts on $\widetilde{\vEc}_G^{\varphi}$ by a left-action
(of course, $\GL_n(F)$ is considered with the opposite group structure from now on).

\begin{prop}\label{fixedpoints}
Let $G$ be a Lie group of dimension $m$, and
$\varphi: E \rightarrow X$ be an $n$-dimensional vector bundle
(with underlying field $F$) such that $\tilde{E}$
is $(m-1)$-connected.
Then for every compact subgroup $K \subset G \times \GL_n(F)$:
\begin{itemize}
\item if $K \cap (\{1\} \times \GL_n(F)) \neq \{1\}$
then $(\widetilde{\vEc}_G^\varphi)^K=\emptyset$;
\item if $K \cap (\{1\} \times \GL_n(F)) = \{1\}$,
then $(\widetilde{\vEc}_G^\varphi)^K \simeq *$.
\end{itemize}
\end{prop}

\begin{proof}
Let $K \subset G \times \GL_n(F)$ be a closed subgroup.
In the case $ K \cap (\{ 1 \} \times \GL_n(F)) \neq \{1\}$, we have
$(\widetilde{\vEc}_G^\varphi)^K=\emptyset$ since
$\GL_n(F)$ acts freely on $\widetilde{\vEc}_G^\varphi$. \\
Assume now that $ K \cap (\{ 1 \} \times \GL_n(F)) = \{1\}$.
First, standard arguments on Lie groups\footnote{
Define indeed $H$ as the image of $K$ by the canonical projection $\pi_1 :
G \times \GL_n(F) \longrightarrow G$;
the assumption on $K$ shows that $LK \cap L(G \times \{1\}) =\{0\}$,
hence the exponential map yields a neighborhood $U$ of $1_G$ in $G$ such that
$U \cap H$ is closed in $U$. It follows that $H$ is a closed subgroup of $G$, hence a Lie group.
Thus the restriction of $\pi_1$ to $K$ induces a continuous bijection
$\alpha : K \overset{\cong}{\longrightarrow} H$. Since both $K$ and $H$ are Lie groups,
$\alpha$ is actually a diffeomorphism, and we may then define $\psi$ as the composite of
$\alpha^{-1}$ with the projection $\pi_2 : G \times \GL_n(F) \rightarrow \GL_n(F)$.}
prove that there is a closed subgroup $H$ of $G$ and a continuous group homomorphism $\psi: H \rightarrow \GL_n(F)$ such that $K=\{(h,\psi(h)) \mid h \in H \}$.

However $(\widetilde{\vEc}_G^\varphi)^K=|\mathcal{F}^K|$. By Lemma \ref{perfectsub}, $\mathcal{F}^K$ is either empty or perfect
as a full subcategory of the perfect k-category $\mathcal{F}$.
We now show that $\mathcal{F}^K$ is non-empty, i.e. we produce a functor $F: \mathcal{E}G \rightarrow \varphi \Frame$
which is invariant by the restriction of the $(G \times \GL_n(F))$-action to $K$.

On the one hand, we have a left-action of $H$ on $\mathcal{E}G$, by right-multiplication of the inverse.
On the other hand, we have a left-action of $H$ on $\varphi \Frame$, defined by
$h.x=x.\psi(h)$ for every $(x,h)\in \tilde{E} \times H$. Then a
functor $\mathcal{E}G \rightarrow \varphi \Frame$ is $K$-invariant if and only if it is $H$-equivariant.
However, an $H$-equivariant functor is nothing but an $H$-map $G \rightarrow \tilde{E}$
since $\varphi \Frame$ is perfect.

However $G$ has the structure of an $H$-CW-complex (for the preceding right-action)
of dimension $\leq m$ (cf.\ \cite{Illman} theorem II).
Also, the only isotropy subgroup for the action of $H$ on $G$ is trivial.
Finally, the map $\tilde{E} \rightarrow *$
induces isomorphisms on the homotopy groups
of dimension $i \leq m-1$. It classically follows that the map
$[G,\tilde{E}]_H \longrightarrow [G,*]_H$
induced by $\tilde{E} \rightarrow *$ is surjective (use the same line of reasoning as in the proof of
Lemma \ref{extentionlemma2}). Since $[G,*]_H \neq \emptyset$, there is at least one
$H$-map from $G$ to $\tilde{E}$, and we deduce that $\mathcal{F}^K$ is a perfect k-category.
The result then follows from Lemma \ref{contractibility}.
\end{proof}

\begin{prop}\label{classifyingspace}
Let $G$ be an $m$-dimensional Lie group, $X_1$ be a
$G$-CW-complex, and
$\varphi: E \rightarrow X$ be an $n$-dimensional vector bundle (with ground field $F$) such that $\tilde{E}$
is $(m-1)$-connected and $X$ is a locally-countable CW-complex.
Pulling back the $G$-vector bundle $E\vEc_G^\varphi \rightarrow \vEc_G^\varphi$
then gives rise to a bijection
$$[X_1,\vEc_G^\varphi]_G \overset{\cong}{\longrightarrow}
\VEct^{F,n} _G(X_1).$$
\end{prop}

\begin{proof}
Let $\varphi _1: E_1 \rightarrow X_1$
be an $n$-dimensional $G$-vector bundle over $X_1$.
We let $\tilde{\varphi _1}: \tilde{E_1} \rightarrow X$ denote the
$(G,\GL_n(F))$-principal bundle canonically associated to $\varphi_1$. Then $\tilde{E_1}$
is a $(G \times \GL_n(F))$-CW-complex.

By Proposition \ref{fixedpoints}, the $G$-map $\widetilde{\vEc}_G^\varphi \rightarrow *$
is a $\mathcal{K}$-weak equivalence, where
$\mathcal{K}$ is the class consisting of all closed subgroups $K$ of $G \times \GL_n(F)$
such that $K \cap(\{1\}\times \GL_n(F))=\{1\}$. Since $\tilde{E_1}$
is a $(G \times \GL_n(F))$-CW-complex, all the isotropy subgroups of which belong
to $\mathcal{K}$ (as was shown earlier), we have:
$$[\tilde{E_1},\widetilde{\vEc}_G^\varphi]_{G \times \GL_n(F)}
\overset{\cong}{\longrightarrow}
[\tilde{E_1},*]_{G \times \GL_n(F)} \cong *.$$

Choose $\tilde{f}$ in the unique class of
$[\tilde{E_1},\widetilde{\vEc}_G^\varphi]_{G\times \GL_n(F)}$.
Then $\tilde{f}$ induces a strong morphism of $G$-vector bundles
$$\begin{CD}
E_1 @>{Ef}>> E\vEc_G^\varphi \\
@V{\varphi _1}VV @V{\pi '}VV \\
X_1 @>{f}>> \vEc_G^\varphi.
\end{CD}$$
We deduce that the map
$[X_1,\vEc_G^\varphi]_G \longrightarrow
\mathbb{V}\text{ect}^{F,n} _G(X_1)$ is onto. We will finish by showing that it is one-to-one. \\
Let $f_1,f_2:X_1 \rightarrow \vEc_G^\varphi$
be two maps together with an isomorphism
$E_1=f_1^*(E\vEc_G^\varphi) \underset{g}{\overset{\cong}{\longrightarrow}}
f_2^*(E\vEc_G^\varphi)=E_2$
of $G$-vector bundles over $X_1$.
Let $\tilde{g}: \tilde{E_1} \rightarrow \tilde{E_2}$ be the associated
morphism of $(G,\GL_n(F))$-principal bundles.
Then $\tilde{g}$ is a $(G \times \GL_n(F))$-map from
$\tilde{E_1}$ to $\tilde{E_2}$. Since $[\tilde{E_1},\widetilde{\vEc}_G^\varphi]_{G \times \GL_n(F)}
\cong *$, the maps $\widetilde{f_2} \circ \widetilde{g}$ and $\widetilde{f_1}$ are $(G \times \GL_n(F))$-homotopic,
and it follows that $f_1$ and $f_2$ are $G$-homotopic. \end{proof}

\subsubsection{Fundamental examples}\label{3.4.2}

For any pair $(n,k) \in \mathbb{N}^2$, we let
$G_n(F^k)$ denote the space of $n$-dimensional linear subspaces of $F^k$,
$\gamma^k _n(F): E_n(F^k) \rightarrow G_n(F^k)$ the
canonical $n$-dimensional vector bundle over the Grassmanian manifold $G_n(F^k)$,
and $\tilde{\gamma}^k _n(F):B_n(F^k) \rightarrow G_n(F^k)$
the $\GL_n(F)$-principal bundle associated to it.
We also let
$\gamma_n(F): E_n(F^{(\infty)}) \rightarrow G_n(F^{(\infty)})$
denote the canonical $n$-dimensional vector bundle over $G_n(F^{(\infty)})$,
and $\tilde{\gamma}_n(F)$ the $\GL_n(F)$-principal bundle
associated to it. Recall that $G_n(F^k)$
is a finite CW-complex for every pair $(k,n)\in \mathbb{N}^2$,
and that $G_n(F^{(\infty)})$,
with the topology induced by the filtration $G_n(F^k) \subset G_n(F^{(\infty)})$,
is a locally-countable CW-complex for every $n \in \mathbb{N}$. We finally set $\gamma^{(m)}(F):=\underset{n \in \mathbb{N}}{\coprod}
\gamma_n^{mn}(F)$ for every $m \in \mathbb{N}^*$, and
$\gamma(F):=\underset{n \in \mathbb{N}}{\coprod}
\gamma_n(F)$.

Let $G$ be an $N$-dimensional Lie group. For every $m \geq \N$, the space
$G_k(F^{mk})$ is $N$-connected for every $k \in \N$
(cf.\ Theorem 5.1 of \cite{Husemoller}).
It follows from Proposition \ref{classifyingspace} that, for every $k \in \mathbb{N}$, every $n \geq N$,
and every proper $G$-CW-complex $X_1$, we have a bijection
$$[X_1,\vEc_G^{\gamma _k^{nk}(F)}]_G \overset{\cong}{\longrightarrow}
\mathbb{V}\text{ect}^{F,k} _G(X_1)$$ induced by pulling back the vector bundle
$E\vEc_G^{\gamma_k^{nk}(F)} \rightarrow \vEc_G^{\gamma_k^{nk}}$.
Moreover, for every $k \in \mathbb{N}$, we have a bijection
$$[X_1,\vEc_G^{\gamma _k(F)}]_G \overset{\cong}{\longrightarrow}
\mathbb{V}\text{ect}^{F,k} _G(X_1).$$
It follows that, for every $k \in \mathbb{N}$, every $n \geq N$,
and every connected $G$-CW-complex $X_1$, we have bijections
$$\left[X_1,\underset{k \in \mathbb{N}}{\coprod}\vEc_G^{\gamma_k^{nk}(F)}\right]_G
\overset{\cong}{\longrightarrow}
\mathbb{V}\text{ect}^F _G(X_1) \quad \text{and} \quad
\left[X_1,\underset{k \in \mathbb{N}}{\coprod}\vEc_G^{\gamma_k(F)}\right]_G
\overset{\cong}{\longrightarrow}
\mathbb{V}\text{ect}^F _G(X_1).$$
Obviously, these results still hold in the case $X_1$ is non-connected.

\subsection{Similar constructions using isometries or similarities}\label{3.5}

Here, $F=\mathbb{R}$ or $\mathbb{C}$.

\begin{Def}A \defterm{simi-Hilbert space} is a finite-dimensional vector space $V$ (with ground field $F$)
with a linear family $(\lambda \langle -,- \rangle)_{\lambda \in \mathbb{R}_+^*}$ of inner products on $V$.
\end{Def}

The relevant notion of isomorphisms between two simi-Hilbert spaces is that of similarities.
We do have a notion of orthogonality, but no notion of orthonormal families.
The relevant notion is that of \defterm{simi-orthonormal} families:
a family will be said to be simi-orthonormal when it is orthogonal and all its vectors share the same positive norm (for any inner product in the
linear family). Equivalently, a family of vectors is simi-orthonormal iff it is orthonormal
for some inner product in the linear family. We also have a \defterm{simi-orthonormalization process}:
if $(e_1,\dots,e_k)$ is a linearly independent
$k$-tuple in a simi-Hilbert space, there is a unique norm $\| - \|$ in the family such that $\|e_1\|=1$.
We then apply the orthonormalization process to $(e_1,\dots,e_k)$ with respect to this norm to obtain a simi-orthonormal family.
This process is compatible with similarities (i.e. if $u$ is a similarity of $V$ for some
inner product space in the family, $(e_1,\dots,e_k)$ is a linearly independent
$k$-tuple, and $(f_1,\dots,f_k)$ is obtained from it by the simi-orthonormalization process, then $(u(f_1),\dots,u(f_k))$ is obtained
from $(u(e_1),\dots,u(e_k))$ by the simi-orthonormalization process), and is continuous
with respect to the choice of the family.

\begin{Def}
Let $G$ be a topological group.
\begin{itemize}
\item For $n \in \mathbb{N}$, an $n$-dimensional $G$-Hilbert bundle is a $G$-vector bundle with fiber $F^n$ and structural group $U_n(F)$.
\item For $n \in \mathbb{N}$, an $n$-dimensional simi-$G$-Hilbert bundle is a $G$-vector bundle with fiber $F^n$ and structural group $\Sim_n(F)$.
\item A disjoint union of $k$-dimensional $G$-Hilbert bundles, for $k\in \mathbb{N}$, is called a $G$-\textbf{Hilbert bundle.}
\item A disjoint union of $k$-dimensional $G$-simi-Hilbert bundles, for $k\in \mathbb{N}$, is called a $G$-\textbf{simi-Hilbert bundle.}
\end{itemize}
\end{Def}

\begin{Rems}
When $G=\{1\}$, we simply speak of Hilbert bundles or simi-Hilbert bundles. Since $U_n(F) \subset \Sim_n(F) \subset \GL_n(F)$,
any $G$-Hilbert bundle is in particular a $G$-simi-Hilbert bundle, and every $G$-simi-Hilbert bundle is in particular
a $G$-vector bundle. \\
There is a broader definition of $G$-Hilbert bundles (see e.g. \cite{Phillips}) which encompasses bundles with infinite-dimensional fibers,
but we will not use it.
\end{Rems}

In a Hilbert bundle, we have a Hilbert space structure on every fiber, and we derive a notion of orthonormal basis
on each fiber. In a simi-Hilbert bundle, we only have an inner product on each fiber defined up to a positive scalar, i.e.\
a structure of simi-Hilbert space on each fiber. From the theory of fibre bundles, we derive familiar notions
of (strong) morphisms of $G$-Hilbert bundles (resp.\ of $G$-simi-Hilbert bundles).

If $\varphi: E \rightarrow X$ is an $n$-dimensional $G$-Hilbert bundle (resp.\ a $G$-simi-Hilbert bundle),
we may consider the subspace $i\tilde{E}\subset E^{\oplus n}$ (resp. $s\tilde{E}\subset E^{\oplus n}$)
consisting of orthonormal bases (resp.\  of simi-orthonormal bases) on the fibers of $\varphi$.
The map $i\tilde{E} \rightarrow X$ (resp.\  $s\tilde{E} \rightarrow X$) is easily shown to
yield a $(G,U_n(F))$-principal bundle (resp.\ a $(G,\Sim_n(F))$-principal bundle), so that
the $G$-fibre bundle with fiber $F^n$ and structural group $U_n(F)$ (resp.\ $\Sim_n(F)$)
canonically associated to it is isomorphic to $p$.

\label{hilbertbundles}
Let $X$ be a $G$-CW-complex. For every $n \in \mathbb{N}$, we define
$i\mathbb{V}\text{ect}_G^{F,n}(X)$ (resp.\ $s\mathbb{V}\text{ect}_G^{F,n}(X)$) as the set of isomorphism classes
of $n$-dimensional $G$-Hilbert bundles (resp.\ $G$-simi-Hilbert bundles) over $X$, and
$i\mathbb{V}\text{ect}_G^F(X)$ (resp.\ $s\mathbb{V}\text{ect}_G^F(X)$)
as the abelian monoid of isomorphism classes of $G$-Hilbert bundles  (resp.\ $G$-simi-Hilbert bundles) over $X$.

\label{hilbertdef}
Replacing $\GL_n(F)$-principal bundles by $U_n(F)$-principal
bundles (resp.\ $\Sim_n(F)$-principal bundles),
and starting from any $n$-dimensional Hilbert bundle (resp.\ simi-Hilbert bundle)
$\varphi: E \rightarrow X$, we may define the k-categories
$\varphi \iframe$, $\varphi \imod$ and $\varphi \iBdl$
(resp. $\varphi \sframe$, $\varphi \smod$
and $\varphi \sBdl$), with a construction that is essentially similar to that of
$\varphi \Frame$, $\varphi \Mod$ and $\varphi \Bdl$.
For every Lie group $G$, we then obtain
$G$-spaces $i\widetilde{\vEc}_G^\varphi$,
$i\vEc_G^\varphi$ and $Ei\vEc_G^\varphi$
(resp. $s\widetilde{\vEc}_G^\varphi$,
$s\vEc_G^\varphi$ and $Es\vEc_G^\varphi$).

We may then prove results that are similar
to Theorem \ref{universal} and to Proposition \ref{classifyingspace}.
To be more precise, on the one hand, if $\varphi:E \rightarrow X$ is an $n$-dimensional Hilbert bundle such that
$X$ is a locally-countable CW-complex, then $i\widetilde{\vEc}_G^\varphi \rightarrow
i\vEc_G^\varphi$ has a structure of $(G,U_n(F))$ principal bundle,
$Ei\vEc_G^\varphi \rightarrow
i\vEc_G^\varphi$ a structure
of $G$-Hilbert bundle, and the natural map $i\widetilde{\vEc}_G^\varphi \times_{U_n(F)}
F^n \rightarrow Ei\vEc_G^\varphi$ is an isomorphism of
Hilbert bundles.
On the other hand, if $\varphi:E \rightarrow X$ is an $n$-dimensional simi-Hilbert bundle such that
$X$ is a locally-countable CW-complex, then $s\widetilde{\vEc}_G^\varphi \rightarrow
s\vEc_G^\varphi$ has a structure of $(G,\Sim_n(F))$-principal bundle,
$Es\vEc_G^\varphi \rightarrow s\vEc_G^\varphi$ a structure of $G$-simi-Hilbert bundle,
and the natural map $s\widetilde{\vEc}_G^\varphi \times_{\Sim_n(F)} F^n \rightarrow Es\vEc_G^\varphi$ is an isomorphism of
$G$-simi-Hilbert bundles.

Moreover, if $\varphi: E \rightarrow X$ is an $n$-dimensional Hilbert bundle such that
$\tilde{E}$ is $(m-1)$-connected with $m=\dim G$, the map $[X_1,i\vEc_G^\varphi]_G \rightarrow i\mathbb{V}\text{ect}_G^{F,n}(X_1)$,
induced by pulling back the universal $G$-Hilbert bundle $Ei\vEc_G^\varphi \rightarrow
i\vEc_G^\varphi$, is a bijection for every $G$-CW-complex $X_1$.
Finally, if $\varphi: E \rightarrow X$ is an $n$-dimensional simi-Hilbert bundle such that
$\tilde{E}$ is $(m-1)$-connected with $m=\dim G$, then
the map $[X_1,s\vEc_G^\varphi]_G \rightarrow s\mathbb{V}\text{ect}_G^{F,n}(X_1)$,
induced by pulling back the $G$-simi-Hilbert bundle
$Es\vEc_G^\varphi \rightarrow s\vEc_G^\varphi$, is a bijection for every $G$-CW-complex $X_1$.

We finish this section with an easy result that establishes a relationship between the three constructions.
If $\varphi: E \rightarrow X$ is an $n$-dimensional Hilbert bundle, then the inclusion
$i\tilde{E} \subset s\tilde{E}$ induces a canonical functor $\varphi \iframe \rightarrow \varphi \sframe$.
If $\varphi: E \rightarrow X$ is an $n$-dimensional simi-Hilbert bundle, then $s\tilde{E} \subset\tilde{E}$ induces a functor
$\varphi \sframe \rightarrow \varphi \Frame$. All those functors induce $G$-maps between the $G$-spaces
$\vEc_G^\varphi$, $E\vEc_G^\varphi$, etc\ldots  that were previously defined using $\varphi$.
The next proposition, the proof of which is straightforward, sums up their properties:

\begin{prop}\label{3comparision} ${}$
\begin{enumerate}[(a)]
\item
Let $\varphi: E \rightarrow X$ be an $n$-dimensional Hilbert bundle, and $G$ be a Lie group.
Then the canonical diagram
$$\begin{CD}
Ei\vEc_G^\varphi @>>> Es\vEc_G^\varphi \\
@VVV @VVV \\
i\vEc_G^\varphi @>>> s\vEc_G^\varphi \\
\end{CD}
$$
defines a strong morphism of $G$-simi-Hilbert bundles.
\item
Let $\varphi: E \rightarrow X$ be an $n$-dimensional simi-Hilbert bundle, and $G$ be a Lie group.
Then, the canonical diagram
$$\begin{CD}
Es\vEc_G^\varphi @>>> E\vEc_G^\varphi \\
@VVV @VVV \\
s\vEc_G^\varphi @>>> \vEc_G^\varphi
\end{CD}$$
defines a strong morphism of $G$-vector bundles.
\end{enumerate}
\end{prop}

\section{A construction of equivariant $\Gamma$-spaces}\label{4}

In this section, $F=\R$ or $\C$.
Our goal here is to construct a $\Gamma-G$-space $\underline{A}$ such that
$\underline{A}(\mathbf{1})$ is a classifying space
for the monoid-valued functor $\VEct_G^F(-)$ on the category of proper $G$-CW-complexes.
In order to achieve this, we introduce pre-decomposed $G$-Hilbert bundles, i.e.\
families of Hilbert bundles (over the same base space)
indexed over a finite set: they will be defined in Section \ref{4.1} as the objects
of a certain category $\Gamma \text{-Fib}^F$. In Sections \ref{4.2} and \ref{4.3}, we extend
the $\Mod$, $\Frame$, $\Bdl$, etc\ldots constructions from Hilbert bundles
with fixed dimension to pre-decomposed Hilbert bundles, and then give basic results about those.
In Section \ref{4.4}, we define
so-called ``Hilbert $\Gamma$-bundles", a certain type of contravariant functors from $\Gamma$ to
$\Gamma \text{-Fib}^F$. Any such functor will induce, after composition with the $\Mod$ construction, a functor
from $\Gamma$ to kCat, and therefore, after composition with  $|\Func(\mathcal{E}G,-)|$,
we will finally recover an equivariant $\Gamma$-space. In Section \ref{4.5}, we produce
particular Hilbert $\Gamma$-bundles
$\text{Fib}^{F^m}$ for each $m \in \mathbb{N}^*\cup \{\infty\}$ and prove that the equivariant
$\Gamma$-spaces $\underline{\vEc}_G^{F,m}(-)$ derived from them
are such that $\underline{\vEc}_G^{F,m}(\mathbf{1})$ is homotopy equivalent
to $\vEc_G^{\gamma^{(m)}(F)}$. We deduce that
$\underline{\vEc}_G^{F,m}(\mathbf{1})$ is a classifying space for $\VEct_G^F(-)$ on the category of
proper $G$-CW-complexes whenever $m \geq \dim(G)$, and that $\underline{\vEc}_G^{F,m}(\mathbf{1})$, with a suitable
H-space structure, classifies $\VEct_G^F(-)$ as a monoid-valued functor. From Section \ref{4.1} to \ref{4.5}, we systematically
make the parallel constructions involving isometries and similarities, just like in Section \ref{3.5}.
In Section \ref{4.6}, the relationship between those
three classes of $\Gamma$-spaces is investigated upon. Finally,
we will show in Section \ref{4.7} that natural transformations of Hilbert
$\Gamma$-bundles induce homotopies between associated $\Gamma$-spaces, and apply this
to some obvious transformations between the Hilbert $\Gamma$-bundles
$\text{Fib}^{F^m}$ for $m \in \mathbb{N}^*\cup \{\infty\}$.

\subsection{The category $\Gamma \text{-Fib}_F$}\label{4.1}

\subsubsection{Definition}\label{4.1.1}

We define the category $\Gamma \text{-Fib}_F$ as follows:
\begin{itemize}
\item An object of $\Gamma \text{-Fib}_F$ consists of a finite
set $S$, a locally-countable CW-complex $X$,
and, for every $s \in S$, of a Hilbert bundle
$p_s:E_s \rightarrow X$ with ground field $F$.
Such an object is called an \textbf{$S$-object} over $X$.
If $S=\mathbf{n}$ for some $n \in \mathbb{N}$,
an $S$-object will be called an $n$-object.
\item A morphism $f:(S,X,(p_s)_{s\in S}) \longrightarrow
(T,Y,(q_t)_{t\in T})$ consists of a morphism $\gamma: T \rightarrow S$
in the category $\Gamma$, a continuous map
$\bar{f}:X \rightarrow Y$, and, for every $t \in T$,
a strong morphism of Hilbert bundles
$$\begin{CD}
\underset{s \in \gamma(t)}{\oplus}E_s @>{f_t}>> E'_t \\
@V{\underset{s \in \gamma(t)}{\oplus}p_s}VV @VV{q_t}V \\
X @>>{\bar{f}}> Y.
\end{CD}$$
\end{itemize}
If $f:(S,X,(p_s)_{s\in S}) \rightarrow (T,Y,(q_t)_{t\in T})$ is the morphism
in $\Gamma \text{-Fib}_F$ corresponding to $(\gamma,\bar{f},(f_t)_{t\in T})$, and
$g:(T,Y,(q_t)_{t\in T}) \rightarrow (U,Z,(r_u)_{u\in U})$ is the morphism
in $\Gamma \text{-Fib}_F$ corresponding to $(\gamma',\bar{g},(g_u)_{u\in U})$, then the composite
morphism $g \circ f: (S,X,(p_s)_{s\in S}) \rightarrow (U,Z,(r_u)_{u\in U})$
is the one which corresponds to the triple consisting of $\gamma \circ \gamma'$, $\bar{g} \circ \bar{f}$, and
the family
$\left(g_u \circ \left[\underset{t\in \gamma'(u)}{\oplus}f_t\right]\right)_{u \in U.}$

\subsubsection{The forgetful functor $\mathcal{O}^F_\Gamma: \Gamma \text{-Fib}_F \longrightarrow \Gamma$ }\label{4.1.2}

$$\mathcal{O}^F_\Gamma: \begin{cases}
\Gamma \text{-Fib}_F & \longrightarrow \Gamma \\
(S,X,(p_s)_{s\in S}) & \longmapsto S \\
f=(\gamma,\bar{f},(f_t)_{t\in T}):(S,X,(p_s)_{s\in S}) \rightarrow (T,Y,(q_t)_{t\in T}) & \longmapsto
(\gamma: T \rightarrow S)
\end{cases}$$

is a contravariant functor from $\Gamma \text{-Fib}_F$ to $\Gamma$.

\subsubsection{Sums in the category $\Gamma \text{-Fib}_F$}\label{4.1.3}

Given a finite set $T$, and, for every $t \in T$, an object
$x_t=(S_t,X_t,(p^t_s)_{s \in S_t})$ of $\Gamma \text{-Fib}_F$, we define:
$$\underset{t \in T}{\sum}x_t:=\left(\underset{t\in T}{\coprod}S_t,
\underset{t\in T}{\prod}X_t, \left(p_s^t \times \underset{t_1\in T \setminus \{t\}}{\prod}\id_{X_{t_1}}\right)_{t\in T,s \in S_t}\right).$$

In particular, if $p$ is a $1$-object, then $n.p$
is an $n$-object (with $\underset{i=1}{\overset{n}{\coprod}}\{1\}=\{1,\dots,n\}$).

For example, the sum of two $1$-objects $(X,p:E \rightarrow X)$ and $(Y,q:E' \rightarrow Y)$ of $\Gamma \text{-Fib}_F$ is
$(\{1,2\},X \times Y, (p_1,p_2))$, where
$$\begin{CD}
E \times Y @>>> E \\
@V{p_1}VV @VV{p}V \\
X \times Y @>>{\pi_1}> X
\end{CD} \hskip 8mm \text{and} \hskip 8mm
\begin{CD}
X \times E' @>>> E' \\
@V{p_2}VV @VV{q}V \\
X \times Y @>>{\pi_2}> Y
\end{CD}$$
are pull-back squares ($\pi_1$ and $\pi_2$ respectively denote the canonical projections).

\subsection{Various functors from $\Gamma \text{-Fib}_F$ to kCat}\label{4.2}

We wish to extend the constructions $\varphi \Mod$ and $\varphi \Bdl$ from
Section \ref{3} to functors $\Gamma \text{-Fib}_F \longrightarrow \text{kCat}$.

\subsubsection{The dimension over a $1$-object}\label{4.2.1}

Let $(X,p:E \rightarrow X)$ be a $1$-object of $\Gamma \text{-Fib}_F$.
The natural map $\dim_p:\begin{cases}
X & \longrightarrow \mathbb{N} \\
x  & \longmapsto \dim(E_x)
\end{cases}$ is continuous because $X$ is a CW-complex.
Thus, setting $X_{n}:=\dim_p^{-1}\{n\}$, $E_n:=p^{-1}(X_n)$, and
$p_n=p_{|E_n}: E_n \rightarrow X_{n}$, then $p=\underset{n \in \mathbb{N}}{\coprod}p_n.$

\subsubsection{The functor $\Mod$}\label{4.2.2}

For any $1$-object $p$ over $X$, i.e.\ any a Hilbert bundle, we define
$$p \Mod:=\underset{n \in \mathbb{N}}{\coprod}(p_n \Mod).$$
We then obtain a canonical functor $p \Mod \rightarrow \mathcal{E}X$ by identifying the spaces of objects.
For any $S$-object $\varphi=(S,X,(p_s:E_s \rightarrow X)_{s \in S})$, we define the category $\varphi \Mod$
as follows: an object of $\varphi \Mod$ is a point $x\in X$, and a morphism
$x \rightarrow y$ in $\varphi \Mod$ is a family $(\varphi_s)_{s\in S}$
of linear isomorphisms $\varphi_s: (E_s)_x \overset{\cong}{\rightarrow} (E_s)_y$. As
a topological category, $\varphi \Mod$ is defined
as the fiber product of the categories $p_s \Mod$ over $\mathcal{E}X$
for all $s\in S$.

Let $f: \varphi=(S,X,(p_s)_{s \in S}) \rightarrow (T,Y,(q_t)_{t \in T})$
be a morphism in $\Gamma \text{-Fib}_F$, with corresponding morphisms $\gamma:T \rightarrow S$,
$\bar{f}:X \rightarrow Y$ and $(f_t)_{t\in T}$.
We assign a morphism $f \Mod$ to $f$ as follows:
\begin{itemize}
\item For every object $x \in X$ of $\varphi \Mod$,
$(f \Mod)(x):=\bar{f}(x)$;
\item
Let $(\varphi_s)_{s \in S}$ be a morphism in $\varphi \Mod$.
For every $t \in T$, we set $$\psi_t:=f_t\circ
\left(\underset{s \in \gamma(t)}{\oplus} \varphi_s\right)
\circ (f_t)^{-1}$$ so that the squares
$$\xymatrix{
\underset{s\in \gamma(t)}{\oplus}(E_s)_x \ar[d]_{\underset{s\in \gamma(t)}{\oplus}\varphi_s}
 \ar[r]^{(f_t)_x} & (E_t)_{f(x)} \ar@{.>}[d]^{\psi_t} \\
\underset{s\in \gamma(t)}{\oplus}(E_s)_y \ar[r]^{(f_t)_y} & (E_t)_{f(y)}
}$$ are all commutative. We then define
$$(f \Mod)((\varphi_s)_{s \in S}):=(\psi_t)_{t\in T}.$$
\end{itemize}
It is easily checked that this definition is compatible with the composition of morphisms, and it thus
yields a functor:

$$\Mod: \Gamma \text{-Fib}_F \longrightarrow
\text{kCat}.$$

\subsubsection{The functors $\imod$ and $\smod$}\label{4.2.3}

For every $1$-object $p$, we set
$$p\imod:=\underset{n \in \mathbb{N}}{\coprod}(p_n \imod).$$
By a construction that is strictly similar to that of $\Mod$, we then recover a functor:
$$\imod: \Gamma \text{-Fib}_F \longrightarrow
\text{kCat}.$$
For any $S$-object $\varphi=(S,X,(p_s:E_s \rightarrow X)_{s \in S})$, an object of $\varphi \imod$ simply corresponds to a point $x\in X$, while
a morphism $x \rightarrow y$ in $\varphi \imod$ is a family $(\varphi_s)_{s\in S}$
of unitary morphisms $\varphi_s: (E_s)_x \overset{\cong}{\rightarrow} (E_s)_y$.

Let $p:E \rightarrow X$ be a $1$-object over $X$.
We set $$p \smod
:=\underset{n \in \mathbb{N}}{\coprod}(p_n \smod).$$
We obtain a functor
$p \smod \rightarrow \mathcal{E}X \times \mathcal{B}\mathbb{R}_+^*$ by assigning $(x,y,\|\varphi\|)$ to every
morphism $\varphi:E_x \rightarrow E_y$ (here, $\|\varphi\|$ denotes the norm of the similarity
$\varphi$ with respect to the respective inner product structures on $E_x$ and $E_y$):
this is compatible with the composition of morphisms, since we are dealing
with similarities here.

For any $S$-object $\varphi=(S,X,(p_s)_{s \in S})$,
$\varphi \smod$ is defined as the fiber product of the categories
$p_s \smod$ over $\mathcal{E}X \times \mathcal{B}\mathbb{R}_+^*$ for all $s\in S$.

For any $S$-object $\varphi=(S,X,(p_s:E_s \rightarrow X)_{s \in S})$, an object of
$\varphi \smod$ simply corresponds to a point $x\in X$, while
a morphism $x \rightarrow y$ in $\varphi \smod$ is a family $(\varphi_s)_{s\in S}$
of similarities $\varphi_s: (E_s)_x \overset{\cong}{\rightarrow} (E_s)_y$ \emph{which share the same norm}.
It is then easy to extend this construction to obtain a functor
$$\smod: \Gamma \text{-Fib}_F \longrightarrow
\text{kCat}.$$
The key point here is that the orthogonal direct sum of similarities is not necessarily a similarity.
The condition on the objects of $\varphi \smod$ that all similarities in the family share the same norm ensures
that any orthogonal direct sum of them is a similarity.

\subsubsection{The functor $\Bdl$}\label{4.2.5}

Let $p$ be a $1$-object over $X$.

We set
$$p \Bdl:=\underset{n \in \mathbb{N}}{\coprod}(p_n \Bdl).$$
As in the case of $p \Mod$, we have a canonical functor
$p \Bdl \rightarrow \mathcal{E}X$.

For every $S$-object $\varphi=(S,X,(p_s:E_s \rightarrow X)_{s\in S})$,
$\varphi \Bdl$ is defined as the fiber product of the categories
$p_s \Bdl$ over $\mathcal{E}X$ for all $s \in S$. An object of $\varphi \Bdl$ simply corresponds to a
family $(e_s)_{s\in S}$ such that, for some $x\in X$:  $\forall s \in S,\, e_s \in (E_s)_x$.
A morphism $(e_s)_{s\in S} \rightarrow (e'_s)_{s\in S}$ in $\varphi \Bdl$,
with $e_s \in (E_s)_x$ and $e'_s \in (E_s)_y$ for all $s\in S$, is a family $(\varphi_s)_{s\in S}$
of linear isomorphisms $\varphi_s: (E_s)_x \overset{\cong}{\rightarrow} (E_s)_y$ such that
$\varphi_s(e_s)=e_s'$ for all $s\in S$.

As in the case of $\Mod$, we can extend $\Bdl$ to obtain a functor
$$\Bdl: \Gamma \text{-Fib}_F \longrightarrow
\text{kCat}.$$
The functors $\varphi \Bdl \rightarrow \varphi \Mod$ of Section \ref{3} then induce a natural transformation:
$$\Bdl \longrightarrow \Mod.$$

\subsubsection{The functors $\iBdl$ and $\sBdl$}\label{4.2.6}

For any $1$-object $p$, we set
$$p \iBdl:=\underset{n \in \mathbb{N}}{\coprod}(p_n \iBdl).$$
As in the case of $\Bdl$, we recover a functor $\iBdl: \Gamma \text{-Fib}_F \longrightarrow
\text{kCat}$ together with a natural transformation $\iBdl \longrightarrow \imod$.

For any $1$-object $p$ over $X$, we set
$$p \sBdl:=\underset{n \in \mathbb{N}}{\coprod}(p_n \sBdl).$$
This time, we do not form the fiber product over
$\mathcal{E}X$, rather over $\mathcal{E}X \times \mathcal{B}\mathbb{R}_+^*$, as in the construction of $\smod$.
We obtain a functor
$\sBdl: \Gamma \text{-Fib}_F \longrightarrow \text{kCat}$
together with a natural transformation $\sBdl \longrightarrow \smod$.

\subsection{Fundamental results on the previous functors}\label{4.3}

\begin{prop}\label{modproduit}
Let $p:E \rightarrow X$ be a $1$-object and $n\in \mathbb{N}^*$.
Then \begin{enumerate}[(i)]
\item $(p \Mod)^n \cong (n.p) \Mod$.
\item $(p \imod)^n \cong (n.p) \imod$.
\item There are three continuous functors
$F^p: (n.p) \smod \rightarrow (p \smod)^n$, $G^p: (p \smod)^n \rightarrow (n.p) \smod$
and $H^p: (p \smod)^n \times I \longrightarrow (p \smod)^n$
such that
$$G^p\circ F^p=\id_{(n.p) \smod}, \;
H^p_{|(p \smod)^n \times \{0\}} =\id_{(p \smod)^n} \; \text{and} \;
H^p_{|(p \smod)^n \times \{1\}} =F^p \circ G^p.$$
\end{enumerate}
\end{prop}

\begin{Rem}
There are similar results for the functors $\Bdl$, $\iBdl$ and $\sBdl$.
\end{Rem}

\begin{proof}
The space of objects of $(n.p) \Mod$ is $X^n$, and this is precisely
the space of objects of $(p \Mod)^n$. A morphism in $(n.p) \Mod$ is defined
by two objects $(x_1,\dots,x_n)$ and $(y_1,\dots,y_n)$ and an $n$-tuple of
linear isomorphisms $\varphi_i:E_{x_i} \overset{\sim}{\rightarrow} E_{y_i}$ (for $i\in \{1,\dots,n\}$); this corresponds
naturally to a morphism in $(p \Mod)^n$. We thus have a canonical isomorphism between the categories
$(n.p) \Mod$ and $(p \Mod)^n$. Similarly, we have a canonical isomorphism
between the categories $(n.p) \imod$ and $(p \imod)^n$. However, this fails for similarities,
because of the norm condition in the definition of the morphisms of $p\smod$.

We define $F^p:(n.p) \smod
\longrightarrow (p \smod)^n$ by assigning
the $n$-tuple of morphisms $((x_1,y_1,\varphi_1),\dots,(x_n,y_n,\varphi_n))$ of $p \smod$
to every morphism \\
$((x_1,\dots,x_n),(y_1,\dots,y_n),(\varphi_1,\dots,\varphi_n))$ of $(n.p) \smod$.

We define $G^p: (p \smod)^n \longrightarrow
(n.p) \smod$ by assigning the morphism \\
$\left((x_1,\dots,x_n),(y_1,\dots,y_n),\left(\varphi_1,\frac{\|\varphi_1\|}
{\|\varphi_2\|}\cdot\varphi_2,\dots,\frac{\|\varphi_1\|}
{\|\varphi_n\|}\cdot\varphi_n\right)\right)$ of $(n.p) \smod$ to every $n$-tuple
$((x_1,y_1,\varphi_1),\dots,(x_n,y_n,\varphi_n))$ of morphisms of $p \smod$. We readily see that $G^p \circ F^p=\id_{n.p \smod}$.

Finally, we define $H^p: (p \smod)^n \times I \longrightarrow (p \smod)^n$ by assigning the morphism
$\left(\left(x_1,y_1,\varphi_1 \right),\left(x_2,y_2,\left(\frac{\|\varphi_1\|}{\|\varphi_2\|}\right)^t.
\varphi_2 \right),\dots, \left(x_n,y_n,\left(\frac{\|\varphi_1\|}{\|\varphi_n\|}\right)^t.
\varphi_n \right) \right)$ of $(p \smod)^n$
to every morphism
$\bigl((x_1,y_1,\varphi_1),\dots,(x_n,y_n,\varphi_n),\bigr)$ of $(p \smod)^n \times I$. It is then easy to check that $H^p$
is a functor, that its restriction to
$(p \smod)^n \times \{0\}$ is $\id_{(p \smod)^n}$ and that its restriction
to $(p \smod)^n \times \{1\}$ is
$F^p \circ G^p$. \end{proof}

\begin{prop}\label{naturaltransformations}
Let $p$ and $q$ be two Hilbert bundles together with two strong morphisms of Hilbert bundles
$$\begin{CD}
E @>{\tilde{f}}>> E' \\
@V{p}VV @VV{q}V \\
X @>>{f}> Y
\end{CD} \hskip 8mm \text{and} \hskip 8mm
\begin{CD}
E' @>{\tilde{g}}>> E \\
@V{p}VV @VV{q}V \\
Y @>>{g}> X
\end{CD} \quad .$$
Then these squares respectively induce
a functor $f \Mod: p \Mod \rightarrow q \Mod$
and a functor $g \Mod: q \Mod \rightarrow p \Mod$, together with
two natural transformations
$$\eta:\id_{p \Mod} \rightarrow (g \Mod) \circ (f \Mod) \quad \text{and} \quad
\epsilon:\id_{q \Mod} \rightarrow (f \Mod) \circ (g \Mod).$$
Similar results hold for the functors $\imod$ and $\smod$.
\end{prop}

\begin{cor}\label{naturalcor}
Let $S$ be a finite set, $\varphi$ and $\psi$ be two $S$-objects of
$\Gamma \text{-Fib}_F$, and $f:\varphi \rightarrow \psi$ and $g:\psi \rightarrow \varphi$ be
two morphisms such that
$\mathcal{O}^F_\Gamma(f)=\mathcal{O}^F_\Gamma(g)=\id_S$. \\
Then there are two natural transformations
$$\eta:\id_{\varphi \Mod} \rightarrow (g \Mod) \circ (f \Mod) \quad \text{and} \quad
\epsilon:\id_{\psi \Mod} \rightarrow (f \Mod) \circ (g \Mod).$$
Similar results hold for the functors $\imod$ and $\smod$.
\end{cor}

\begin{proof}[Proof of Proposition \ref{naturaltransformations}:]
In order to build $\eta$,
it suffices to construct, for every
$x \in X$, an isomorphism between $E_x$ et $E_{g(f(x))}$ (and to do this in a continuous way, of course).
This isomorphism is simply given by:
$\begin{cases}
E_x & \rightarrow E_{g(f(x))} \\
y & \mapsto \tilde{g}(\tilde{f}(y)).
\end{cases}$ The construction of $\epsilon$ is similar. \end{proof}

\subsection{Hilbert $\Gamma$-bundles}\label{4.4}

\subsubsection{Definition}\label{4.4.1}

\begin{Def}
A \defterm{Hilbert $\Gamma$-bundle} is a contravariant functor
$\varphi:\Gamma \longrightarrow \Gamma \text{-Fib}_F$
which satisfies the following conditions:

\begin{enumerate}[(i)]
\item $\mathcal{O}^F_\Gamma \circ \varphi=\id_{\Gamma}$;
\item $\varphi(\mathbf{0})=(\mathbf{0},*,\emptyset)$;
\item For every $n \in \mathbb{N}^*$, there exists a morphism $f_n:n.\varphi
(\mathbf{1}) \rightarrow \varphi(\mathbf{n})$ in $\Gamma \text{-Fib}_F$ such that
$\mathcal{O}^F_\Gamma(f_n)=\id_{\mathbf{n}}$.
\end{enumerate}
\end{Def}

\begin{Rem} Notice that in condition (iii), only the existence
of some morphisms is required, but they are not part of the structure
of a Hilbert $\Gamma$-bundle.
\end{Rem}

\subsubsection{Induced constructions}\label{4.4.2}

\begin{Def}
Let $\varphi$ be an object of $\Gamma-\text{Fib}_F$ and $G$ be a Lie group. We define
$$\vEc_G^\varphi:=\left|\Func(\mathcal{E}G,\varphi \Mod)\right|, \quad
i\vEc_G^\varphi:=\left|\Func(\mathcal{E}G,\varphi \imod)\right|, \quad \text{and} \quad
s\vEc_G^\varphi:=\left|\Func(\mathcal{E}G,\varphi \smod)\right|.$$
\end{Def}

In what follows, we let $\varphi$ be a Hilbert $\Gamma$-bundle and
$G$ be a Lie group.

\begin{Def}\label{gammaspaces}
We define three functors from $\Gamma$ to $\text{CG}_G$:

$$\underline{\vEc}_G^\varphi: S \longmapsto \vEc_G^{\varphi(S)}, \quad
i\underline{\vEc}_G^\varphi: S \longmapsto i\vEc_G^{\varphi(S)} \quad \text{and} \quad
s\underline{\vEc}_G^\varphi: S \longmapsto s\vEc_G^{\varphi(S)}.$$
\end{Def}

\begin{prop}\label{gammaspaces}
The three functors $\underline{\vEc}_G^\varphi$, $i\underline{\vEc}_G^\varphi$ and
$s\underline{\vEc}_G^\varphi$ define $\Gamma-G$-spaces.
\end{prop}

\begin{proof}
We only prove the proposition for $\underline{\vEc}_G^\varphi$ and
$s\underline{\vEc}_G^\varphi$, since the proof for
$i\underline{\vEc}_G^\varphi$ is similar to that for $\underline{\vEc}_G^\varphi$. \\
$\bullet$ \textbf{The case of $\vEc_G^\varphi$.}
We first notice that $\varphi(\mathbf{0}) \text{-mod}=*$, and
so $\underline{\vEc}_G^\varphi(\mathbf{0})=*$.

The maps from $\mathbf{1}$ to $\mathbf{n}$ induce
a morphism $g_n: \varphi(\mathbf{n}) \rightarrow \varphi(\mathbf{1})^n$
such that
$\mathcal{O}^F_{\Gamma}(g_n)=\id_{\mathbf{n}}$.
By assumption on $\varphi$, using Proposition \ref{modproduit} shows there is also a morphism
$f_n: \varphi(\mathbf{1})^n \longrightarrow \varphi(\mathbf{n})$
such that
$\mathcal{O}^F_\Gamma(f_n)=\id_{\mathbf{n}}$.
Using Corollary \ref{naturalcor}, we deduce that
$f_n$ and $g_n$ respectively induce equivariant continuous functors
$$\tilde{f_n}:\Func(\mathcal{E}G,\varphi(\mathbf{1})^n \Mod)
\longrightarrow \Func(\mathcal{E}G,\varphi(\mathbf{n}) \Mod)$$
and
$$\tilde{g_n}:\Func(\mathcal{E}G,\varphi(\mathbf{n}) \Mod)
\longrightarrow \Func(\mathcal{E}G,\varphi(\mathbf{1})^n \Mod),$$
and equivariant natural transformations
$\tilde{\eta}: \id \longrightarrow \tilde{f_n} \circ \tilde{g_n}$
and
$\tilde{\eta}: \id \longrightarrow \tilde{g_n} \circ \tilde{f_n}.$
Therefore, $g_n$ induces
an equivariant homotopy equivalence
$$\left|\Func(\mathcal{E}G,\varphi(\mathbf{n}) \Mod)\right|
\overset{\simeq}{\longrightarrow}
\left|\Func(\mathcal{E}G,\varphi(\mathbf{1})^n \Mod)\right|.$$

From Proposition \ref{modproduit}, we derive that
$(\varphi(\mathbf{1}) \Mod)^n \cong \varphi(\mathbf{1})^n \Mod$, and it follows that
$$\left|\Func(\mathcal{E}G,\varphi(\mathbf{1})^n \Mod)\right|
\simeq \left|\Func(\mathcal{E}G,(\varphi(\mathbf{1}) \Mod)^n)\right|
\cong \left|\Func(\mathcal{E}G,\varphi(\mathbf{1}) \Mod)\right|^n.$$

We thus have an equivariant homotopy equivalence
$$\vEc_G^\varphi(\mathbf{n})=
\left|\Func(\mathcal{E}G,\varphi(\mathbf{n}) \Mod)\right|
\overset{\simeq}{\longrightarrow}
\left|\Func(\mathcal{E}G,\varphi(\mathbf{1}) \Mod)\right|^n
=\vEc_G^\varphi(\mathbf{1})^n,$$
and it is easy to check that it is induced by the maps from
$\mathbf{1}$ to $\mathbf{n}$. We conclude that $\underline{\vEc}_G^\varphi$ is a $\Gamma-G$-space. \\
$\bullet$ \textbf{The case of $s\vEc_G^\varphi$.} As in the above case, we obtain an equivariant homotopy equivalence
$$\left|\Func(\mathcal{E}G,\varphi(\mathbf{n}) \smod)\right|
\overset{\simeq}{\longrightarrow}
\left|\Func(\mathcal{E}G,\left[n.\varphi(\mathbf{1})\right] \smod)\right|.$$
We set $p:=\varphi(\mathbf{1})$. With the notations from Proposition \ref{modproduit}, we obtain
functors $F^p:(n.p) \smod \rightarrow (p \smod)^n$,
$G^p:(p \smod)^n
\rightarrow (n.p) \smod$ and
$H^p: (p \smod)^n \times I
\longrightarrow (p \smod)^n$.

We thus recover three equivariant functors:

$$\tilde{F^p}
: \Func(\mathcal{E}G,(n.p) \smod) \longrightarrow
\Func(\mathcal{E}G,(p \smod)^n),$$
$$\tilde{G^p}
: \Func(\mathcal{E}G,(p \smod)^n) \times I \longrightarrow
\Func(\mathcal{E}G,n.p \smod),$$
and
$$\tilde{H^p}
: \Func(\mathcal{E}G,(p-\text{smod})^n) \times I \longrightarrow
\Func(\mathcal{E}G,(p-\text{smod})^n).$$
We easily find that
$\tilde{G^p} \circ \tilde{F^p}=\id$, and that
$\tilde{H^p}$ is an equivariant homotopy from
$\id$ to $\tilde{F^p} \circ \tilde{G^p}$. After taking the geometric realizations, we obtain that
$|\tilde{F^p}|$ is an equivariant homotopy equivalence
and $|\tilde{G^p}|$ is an equivariant homotopy inverse of $|\tilde{F^p}|$.
Therefore, the natural map
$\left|\Func(\mathcal{E}G,\varphi(\mathbf{n}) \smod)\right|
\overset{\simeq}{\longrightarrow}
\left|\Func(\mathcal{E}G,(\varphi(\mathbf{1}) \smod)^n)\right|$
is an equivariant homotopy equivalence.
However, we know that $\left|\Func(\mathcal{E}G,(\varphi(\mathbf{1}) \smod)^n)\right|
\cong \left|\Func(\mathcal{E}G,\varphi(\mathbf{1}) \smod)\right|^n$.
Therefore, the map
$$s\vEc_G^{\varphi}(\mathbf{n}) =
\left|\Func(\mathcal{E}G,\varphi(\mathbf{n}) \smod)\right|
\overset{\simeq}{\longrightarrow}
\left|\Func(\mathcal{E}G,\varphi(\mathbf{1}) \smod)\right|^n
= (s\vEc_G^\varphi(\mathbf{1}))^n,$$
induced by the maps from $\mathbf{1}$ to $\mathbf{n}$,
is an equivariant homotopy equivalence.
We conclude that $s\underline{\vEc}_G^\varphi$ is a $\Gamma-G$-space.
\end{proof}

\subsubsection{Fundamental examples of Hilbert $\Gamma$-bundles}\label{4.5}

\begin{Def}
For every finite set $S$, we let $\Gamma (S)$ denote the set of maps $f:\mathcal{P}(S)
\rightarrow \mathcal{P}(\mathbb{N})$ which respect disjoint unions (and in particular $f(\emptyset)=\emptyset$),
and such that $f(S)$ is finite. We will write $S \overset{f}{\rightarrow}
\mathbb{N}$ when $f\in \Gamma(S)$.

For an inner product space $\mathcal{H}$ (with underlying field $F$) of finite dimension or
isomorphic to $F^{(\infty)}$, and for a finite subset $A$ of $\mathbb{N}$, we may consider the inner product space
$\mathcal{H}^A$ as embedded in the Hilbert space $\mathcal{H}^\infty$.
We then define $G_A(\mathcal{H})$ as the set of subspaces of dimension $\# A$ of $\mathcal{H}^A$,
with the limit topology for the inclusion of $G_{\#A}(E)$, where $E$ ranges over
the finite dimensional subspaces of $\mathcal{H}$.
When $A$ is empty, we set $G_\emptyset(\mathcal{H})=*$.

We let $p_A(\mathcal{H}): E_A(\mathcal{H}) \longrightarrow G_A(\mathcal{H})$
denote the canonical Hilbert bundle of dimension $\#A$ over $G_A(\mathcal{H})$.
The set $E_A(\mathcal{H})$ is constructed as a subspace of
the product of $\mathcal{H}^A$ (with the limit topology described above)
with $G_A(\mathcal{H})$.
\end{Def}

\begin{Rem}
In any case, $G_A(\mathcal{H})$ is a countable CW-complex.
\end{Rem}

For every finite set $S$, we define the following object of $\Gamma \text{-Fib}_F$:
$$\text{Fib}^{\mathcal{H}}(S):=
\left(S,X^\mathcal{H}(S),(p^\mathcal{H}(s))_{s \in S}\right),$$
where
$$X^\mathcal{H}(S):=\underset{f \in \Gamma(S)}{\coprod}\left[\underset{s \in S}{\prod}
G_{f(s)}(\mathcal{H})\right]$$
and, for every $s \in S$,
$$p^\mathcal{H}(s):=
\underset{f \in \Gamma(S)}{\coprod}\left[p_{f(s)}(\mathcal{H})
\times \underset{s' \in S\setminus\{s\}}{\prod}
\id_{G_{f(s')}(\mathcal{H})}\right].
$$
Let $\gamma: S \rightarrow T$ be a morphism in $\Gamma$.
We define a morphism
$\text{Fib}^{\mathcal{H}}(\gamma):\text{Fib}^{\mathcal{H}}(T) \rightarrow
\text{Fib}^{\mathcal{H}}(S)$ of $\Gamma \text{-Fib}_F$ for which
$\mathcal{O}^F_\Gamma(\text{Fib}^{\mathcal{H}}(\gamma))=\gamma$, in the following way:
for every $f \in \Gamma(S)$, we consider the map
$$\begin{cases}
\underset{t \in \gamma(s)}{\prod}
G_{f(t)}(\mathcal{H}) & \longrightarrow
G_{f \circ \gamma (s)}(\mathcal{H}) \\
(E_t)_{t \in \gamma(s)} & \longmapsto \underset{t \in \gamma(s)}{\overset{\bot}{\oplus}}
E_t,
\end{cases}$$
and, for every $s \in S$ and $f \in \Gamma(S)$, we have a commutative
square:
$$\begin{CD}
\left[\underset{t \in \gamma(s)}{\oplus}
E_{f(t)}(\mathcal{H})\right]
\times \underset{t \in T \setminus \gamma(s)}{\prod}
G_{f(t)}(\mathcal{H}) @>>>
E_{f\circ \gamma(s)}(\mathcal{H})
\times \underset{s' \in S \setminus \{s\}}{\prod}
G_{f(s')}(\mathcal{H}) \\
@VVV @VVV \\
\underset{t \in T}{\prod}
G_{f(t)}(\mathcal{H}) @>>>
\underset{s_1 \in S}{\prod}
G_{f\circ \gamma(s_1)}(\mathcal{H}),
\end{CD}
$$
where the upper morphism is given by the previous map and the following one:
$$\begin{cases}
\underset{t \in \gamma(s)}{\oplus}
E_{f(t)}(\mathcal{H}) & \longrightarrow
E_{f \circ \gamma (s)}(\mathcal{H}) \\
(x_t)_{t \in \gamma(s)} & \longmapsto
\underset{t \in \gamma(s)}{\sum} x_t.
\end{cases}$$
The above squares turn out to define strong morphisms of Hilbert bundles,
and we deduce that
$\text{Fib}^{\mathcal{H}}(\gamma)$ is a morphism in
$\Gamma \text{-Fib}_F$ whose image by $\mathcal{O}^F_\Gamma$ is $\gamma$.
Checking that this is compatible with the composition of morphisms is an easy task,
and we conclude that we have defined a functor
$$\text{Fib}^{\mathcal{H}}: \Gamma \longrightarrow \Gamma \text{-Fib}_F$$
such that $\mathcal{O}^F_\Gamma \circ \text{Fib}^{\mathcal{H}} = \id_\Gamma$.

\begin{Rem} The underlying idea is that we want
to map any finite family of finite dimensional subspaces of $\mathcal{H}^{(\infty)}$ to
its orthogonal direct sum in $\mathcal{H}^{(\infty)}$, if possible. The problem is solved by
taking ``labeled" subspaces, i.e. pairs $(A,x)$ consisting of a finite subset
$A$ of $\mathbb{N}$ and an $\# A$-dimensional subspace $x$ of $\mathcal{H}^A$
(the condition on the dimension of $x$ is not necessary in this paper but will prove crucial in \cite{Ktheo2}).
We then only accept families $((A_i,x_i))_{i \in I}$ consisting of labeled subspaces such that the $A_i$'s are pairwise
disjoint. With this condition, the $x_i$'s are automatically pairwise orthogonal, and
their orthogonal direct sum may be seen as a subspace of $\mathcal{H}^{\underset{i \in I}{\cup} A_i}$
so that the pair $\Bigl(\underset{i \in I}{\cup} A_i,\underset{i \in I}{\overset{\bot}{\oplus}} x_i\Bigr)$ is
a labeled subspace.
\end{Rem}

\begin{prop}
Let $\mathcal{H}$ be an inner product space with ground field $F$. Assume that
$\mathcal{H}$ is finite-dimensional or isomorphic to $F^{(\infty)}$. Then
$\text{Fib}^{\mathcal{H}}$ is a Hilbert $\Gamma$-bundle.
\end{prop}

\begin{proof}
Condition (i) has already been checked. Moreover,
$\text{Fib}^{\mathcal{H}}(\mathbf{0})=(\mathbf{0},*,\emptyset)$, by construction.
It remains to prove that condition (iii) is satisfied.

We choose $n$ elements $f_1$, $f_2$,..., $f_n$ in $\Gamma(\{1\})$.
Every $f_i$ then corresponds to a finite subset of $\mathbb{N}$.
If $f_i(1)=\emptyset$, we set $\max(f_i(1))=-1$.
We then define
$$\underset{1 \leq i \leq n}{\coprod}f_i:
\begin{cases}
\{1,\dots,n\} & \longrightarrow \mathcal{P}(\mathbb{N}) \\
1 & \longmapsto f_1(1) \\
i>1 & \longmapsto f_i(1) + i-1 + \underset{1 \leq j \leq i-1}{\sum}
\max(f_j(1)).
\end{cases}$$
For every $k\in \{1,\dots,n\}$, there is an increasing bijection from
$f_k(1)$ to $\left(\underset{1 \leq i \leq n}{\coprod}f_i\right)(k)$,
and these bijections give rise to a continuous map:
$$\underset{1 \leq i \leq n}{\prod}G_{f_i(1)}(\mathcal{H})
\longrightarrow \underset{1 \leq i \leq n}{\prod}G_
{\left(\underset{1 \leq j \leq n}{\coprod}f_j\right)(i)}(\mathcal{H}),$$
and, for every $i \in \{1,\dots,n\}$, a strong morphism of Hilbert bundles:
$$\begin{CD}
E_{f_i(1)}(\mathcal{H}) @>>> E_{\left(\underset{1 \leq j \leq n}{\coprod}f_j\right)(i)}(\mathcal{H}) \\
@VVV @VVV \\
G_{f_i(1)}(\mathcal{H}) @>>>
G_{\left(\underset{1 \leq j \leq n}{\coprod}f_j\right)(i)}(\mathcal{H}).
\end{CD}$$

We have just constructed a morphism
$n.\text{Fib}^{\mathcal{H}}(1) \longrightarrow \text{Fib}^{\mathcal{H}}(n)$ of $\Gamma \text{-Fib}_F$,
whose image by $\mathcal{O}_\Gamma^F$ is $\id_{\mathbf{n}}$. Hence condition (iii) is satisfied.
\end{proof}

We now tackle the special case where $\mathcal{H}=F^m$ for
$1 \leq m \leq \infty$, and try to relate it to the constructions of Section \ref{3} (when $m=\infty$, $F^m$ should be
understood as $F^{(\infty)}$). We will build two morphisms $\text{Fib}^{F^m}(\mathbf{1})
\leftrightarrows
\gamma^{(m)}(F)$
the image of which is $\id_\mathbf{1}$ by $\mathcal{O}^F_\Gamma$ (cf.\ Section \ref{3.4.2} for the definition
of $\gamma^{(m)}(F)$).

For $n \in \mathbb{N}$, we can identify
$\gamma_n^{mn}(F)$ with $p_{\{0,\dots,n-1\}}(F^m)$.
We then obtain a strong morphism of Hilbert bundles:

$$\begin{CD}
\underset{n \in \mathbb{N}}{\coprod} E_n(F^{mn}) @>>>
\underset{f \in \Gamma(\mathbf{1})}{\coprod}E_{f(1)}(F^m) \\
@V{\gamma^{(m)}(F)}VV @VV{\underset{f \in \Gamma(\mathbf{1})}{\coprod}p_{f(1)}(F^m)}V \\
\underset{n \in \mathbb{N}}{\coprod} G_n(F^{mn})
@>>> \underset{f \in \Gamma(\mathbf{1})}{\coprod}G_{f(1)}(F^m),
\end{CD}$$
defining our first morphism $\gamma^{(m)}(F)
\longrightarrow \text{Fib}^{F^m}(\mathbf{1})$.

For any element $f \in \Gamma(\mathbf{1})$, we choose a bijection from $f(1)$ to
$\{1,\dots,n\}$ (for some $n \in \mathbb{N}$). This induces a strong morphism of Hilbert bundles:

$$\begin{CD}
E_{f(1)}(F^m) @>>> E_n(F^{mn}) \\
@VVV @VVV \\
G_{f(1)}(F^m) @>>> G_n(F^{mn}).
\end{CD}$$
Collecting those strong morphisms yields the second morphism $\text{Fib}^{F^m}(\mathbf{1})
\longrightarrow \gamma^{(m)}(F)$.\label{deftypique} We may now set:
$$\underline{\vEc}_G^{F,m}:=
\underline{\vEc}_G^{\text{Fib}^{F^m}} \quad; \quad
\vEc_G^{F,m}:=
\vEc_G^{\text{Fib}^{F^m}(\mathbf{1})}
\quad; \quad
E\vEc_G^{F,m}:=
E\vEc_G^{\text{Fib}^{F^m}(\mathbf{1})};
$$
$$i\underline{\vEc}_G^{F,m}:=
i\underline{\vEc}_G^{\text{Fib}^{F^m}} \quad; \quad
i\vEc_G^{F,m}:=
i\vEc_G^{\text{Fib}^{F^m}(\mathbf{1})}
\quad; \quad
Ei\vEc_G^{F,m}:=
Ei\vEc_G^{\text{Fib}^{F^m}(\mathbf{1})},
$$
and
$$s\underline{\vEc}_G^{F,m}:=
s\underline{\vEc}_G^{\text{Fib}^{F^m}} \quad; \quad
s\vEc_G^{F,m}:=
s\vEc_G^{\text{Fib}^{F^m}(\mathbf{1})}
\quad; \quad
Es\vEc_G^{F,m}:=
Es\vEc_G^{\text{Fib}^{F^m}(\mathbf{1})}.
$$
The $\Gamma$-space structure of $\underline{\vEc}_G^{F,m}$
induces a structure of equivariant H-space on
$\vEc_G^{F,m}=\underline{\vEc}_G^{F,m}(\mathbf{1})$.

\begin{prop}\label{classpace}
Let $G$ be a second-countable Lie group and $m \in \mathbb{N}^* \cup \{\infty\}$
such that $m \geq \dim G$.
Then $E\vEc_G^{F,m} \rightarrow \vEc_G^{F,m}$ is universal
for finite-dimensional $G$-vector bundles
over $G$-CW-complexes, and, for every $G$-CW-complex
$X$, the induced bijection
$$\Phi: [X,\vEc_G^{F,m}]_G \overset{\cong}{\longrightarrow}
\VEct_G^F(X)$$ is a homomorphism of abelian monoids.
\end{prop}

\begin{Rem}Similar results hold for
$i\vEc_G^{F,m}$ (respectively for $s\vEc_G^{F,m}$) by replacing the notion of finite-dimensional
$G$-vector bundle by the notion of $G$-Hilbert bundle
(resp. by the notion of $G$-simi-Hilbert bundle).
\end{Rem}

Before proving this, we need two lemmas (the proofs are straightforward hence omitted):

\begin{lemme}\label{strong2strong}
Let $$\begin{CD}
E @>>> E' \\
@V{p}VV @VV{q}V \\
X @>>> X'
\end{CD}$$ be a strong morphism of Hilbert bundles.
Then the induced square
$$\begin{CD}
E\vEc_G^p @>>> E\vEc_G^q \\
@VVV @VVV \\
\vEc_G^p @>>> \vEc_G^q
\end{CD}$$
is a strong morphism of finite-dimensional $G$-vector bundles.
\end{lemme}

\begin{lemme}\label{surjectivegamma}
Let $\varphi=(S,X,(p_s)_{s \in S}) \overset{f}{\longrightarrow}
\psi=(T,Y,(q_t)_{t \in T})$ be a morphism in $\Gamma-\text{Fib}_F$
such that $\mathcal{O}_\Gamma^F(f)(T)=S$.
Then $f$ induces a commutative square

$$\begin{CD}
|\Func(\mathcal{E}G,\varphi \Bdl)| @>>>
|\Func(\mathcal{E}G,\psi \Bdl)| \\
@VVV @VVV \\
|\Func(\mathcal{E}G,\varphi \Mod)| @>>>
|\Func(\mathcal{E}G,\psi \Mod)| \\
\end{CD}$$
which defines a strong morphism of finite-dimensional pseudo-$G$-vector bundles.
\end{lemme}

\begin{proof}[Proof of Proposition \label{classpace}]
From the previous constructions and Proposition \ref{naturaltransformations},
we deduce that the $G$-maps $\vEc_G^{F,m} \rightarrow \vEc_G^{\gamma^{(m)}(F)}$ and
$\vEc_G^{\gamma^{(m)}(F)} \rightarrow \vEc_G^{F,m}$ (induced by the above functors)
are equivariant homotopy equivalences that are inverse one to the other up to an equivariant homotopy.

Let $X$ be a $G$-CW-complex. We deduce that the map $[X,\vEc_G^{\gamma^{(m)}(F)}]_G \overset{\cong}{\longrightarrow}
 [X,\vEc_G^{F,m}]_G$ induced by $\vEc_G^{F,m} \rightarrow \vEc_G^{\gamma^{(m)}(F)}$ is a bijection.
However, since Lemma \ref{strong2strong} shows the square
$$\begin{CD}
E\vEc_G^{\gamma^{(m)}(F)}
@>>> E\vEc_G^{F,m} \\
@VVV @VVV \\
\vEc_G^{\gamma^{(m)}(F)}
@>>> \vEc_G^{F,m}
\end{CD}$$ is a strong morphism of finite-dimensional
$G$-vector bundles, the composite of $[X,\vEc_G^{\gamma^{(m)}(F)}]_G \overset{\cong}{\longrightarrow}
 [X,\vEc_G^{F,m}]_G$ with $\left[X,\vEc_G^{F,m}\right]_G
\longrightarrow \VEct_G^F(X)$ (induced by pulling back the $G$-vector bundle $E\vEc_G^{F,m} \rightarrow \vEc_G^{F,m}$)
is the map $\left[X,\vEc_G^{\gamma^{(m)}(F)}\right]_G \overset{\cong}{\longrightarrow} \VEct_G^F(X)$ induced by pulling back the $G$-vector bundle
$E\vEc_G^{\gamma^{(m)}(F)} \rightarrow \vEc_G^{\gamma^{(m)}(F)}$.

We know from Section \ref{3.4.2}
that this last map is a bijection, and we deduce that the $G$-vector bundle $E\vEc_G^{F,m} \rightarrow \vEc_G^{F,m}$
induces a bijection  $\Phi: \left[X,\vEc_G^{F,m}\right]_G \overset{\cong}{\longrightarrow} \VEct_G^F(X)$.
This proves the first part of the proposition.

In order to prove the second part, we consider the commutative square:
$$\begin{CD}
E\vEc_G^{F,m} \ktimes E\vEc_G^{F,m} @>>> \left|\Func(\mathcal{E}G,
(\text{Fib}^{F^m}(\mathbf{1}) \Bdl)^2)\right| \\
@VVV @VVV\\
\vEc_G^{F,m} \ktimes \vEc_G^{F,m} @>>> \left|\Func(\mathcal{E}G,
(\text{Fib}^{F^m}(\mathbf{1}) \Mod)^2)\right|,
\end{CD}$$
which defines a strong morphism of pseudo-$G$-vector bundles; and
$$\begin{CD}
\left|\Func(\mathcal{E}G,
(\text{Fib}^{F^m}(\mathbf{1}) \Bdl)^2)\right| @>>>
\left|\Func(\mathcal{E}G,
(2.\text{Fib}^{F^m}(\mathbf{1})) \Bdl)\right| \\
@VVV @VVV \\
\left|\Func(\mathcal{E}G,
(\text{Fib}^{F^m}(\mathbf{1}) \Mod)^2)\right| @>>>
\left|\Func(\mathcal{E}G,
(2.\text{Fib}^{F^m}(\mathbf{1})) \Mod)\right|,
\end{CD}$$
which also defines a strong morphism of pseudo-$G$-vector bundles.

Since $\text{Fib}^{F^m}$ is a Hilbert $\Gamma$-bundle, we have a functor
$(\text{Fib}^{F^m}(\mathbf{1}))^2 \longrightarrow \text{Fib}^{F^m}(\mathbf{2})$ (by condition (iii)), and
a functor $\text{Fib}^{F^m}(\mathbf{2}) \rightarrow \text{Fib}^{F^m}(\mathbf{1})$ induced by the morphism
$\{1\} \rightarrow \{1,2\}$ of $\Gamma$ which sends $1$ to $\{1,2\}$.
The composition of these functors induces a commutative square
$$\begin{CD}
\left|\Func(\mathcal{E}G,
(\text{Fib}^{F^m}(\mathbf{1}))^2 \Bdl)\right| @>>>
\left|\Func(\mathcal{E}G,
\text{Fib}^{F^m}(\mathbf{1}) \Bdl)\right| \\
@VVV @VVV \\
\left|\Func(\mathcal{E}G,
(\text{Fib}^{F^m}(\mathbf{1}))^2 \Mod)\right| @>>>
\left|\Func(\mathcal{E}G,
\text{Fib}^{F^m}(\mathbf{1}) \Mod)\right|,
\end{CD}$$
and Lemma \ref{surjectivegamma} shows that it defines a strong morphism of pseudo-$G$-vector bundles.

By composing these three squares, we finally obtain a square

$$\begin{CD}
E\vEc_G^{F,m}\ktimes E\vEc_G^{F,m} @>>> E\vEc_G^{F,m} \\
@VVV @VVV \\
\vEc_G^{F,m} \ktimes \vEc_G^{F,m} @>{\pi}>> \vEc_G^{F,m},
\end{CD}$$
which is a strong morphism of pseudo-$G$-vector bundles, and where, by construction, the map $\pi$ arises
from the equivariant H-space structure on $\vEc_G^{F,m}$ induced by the equivariant $\Gamma$-space
structure on $\underline{\vEc}_G^{F,m}$.

Let $p: E \rightarrow X$ and $q:E' \rightarrow X$ be two
$G$-vector bundles over a $G$-CW-complex $X$, and
$f:X \rightarrow \vEc_G^{F,m}$ (resp.\
$g:X \rightarrow \vEc_G^{F,m}$) be a map corresponding to
$[p]$ (resp.\ $[q]$)
by the isomorphism $\Phi: [X,\vEc_G^{F,m}]_G \overset{\cong}{\longrightarrow}
\mathbb{V}\text{ect}_G^F(X)$.
Then the square
$$\begin{CD}
E \oplus E' @>>> E\vEc_G^{F,m}\ktimes E\vEc_G^{F,m} \\
@V{p\oplus q}VV @VVV \\
X @>{(f,g)}>> \vEc_G^{F,m} \ktimes \vEc_G^{F,m}
\end{CD}
$$
defines a strong morphism of pseudo-$G$-vector bundles.
Right-composing it with the preceding square yields a new square
$$\begin{CD}
E \oplus E' @>>> E\vEc_G^{F,m} \\
@VVV @VVV \\
X @>{\pi \circ(f,g)}>> \vEc_G^{F,m}
\end{CD}
$$
which defines a strong morphism of finite-dimensional $G$-Hilbert bundles, and we deduce that $\Phi([\pi \circ(f,g)])=[p\oplus q]$.
We conclude that
$$\Phi([f]+[g])=\Phi([\pi \circ(f,g)])=
[p \oplus q]=[p]+[q]=\Phi([f])+\Phi([g]).$$
\end{proof}

\subsubsection{Relating the three constructions}\label{4.6}

Here, we will assume that the Lie group $G$ is second-countable, i.e.\ that $\pi_0(G)$ is countable.

The natural inclusions $U_n(F) \subset \Sim_n(F) \subset \GL_n(F)$
induce a commutative diagram of natural transformations:
$$\begin{CD}
\iBdl @>>> \sBdl @>>> \Bdl \\
@VVV @VVV @VVV \\
\imod @>>> \smod @>>> \Mod.
\end{CD}
$$
For every $m \in \mathbb{N}^* \cup \{\infty\}$, we have the following diagram:
$$\begin{CD}
Ei\vEc_G^{F,m} @>>> Es\vEc_G^{F,m} @>>> E\vEc_G^{F,m} \\
@VVV @VVV @VVV \\
i\vEc_G^{F,m} @>>> s\vEc_G^{F,m} @>>> \vEc_G^{F,m}.
\end{CD}$$
The left-hand square is a strong morphism of $G$-simi-Hilbert bundles,
whilst the right-hand square is a strong morphism of $G$-vector bundles. This yields
natural transformations
$$[-,i\vEc_G^{F,m}]_G \rightarrow [-,s\vEc_G^{F,m}]_G
\rightarrow [-,\vEc_G^{F,m}]_G$$
on the category of $G$-spaces. We therefore obtain a commutative diagram
\begin{equation}\label{compared}
\begin{CD}
[-,i\vEc_G^{F,m}]_G @>>> [-,s\vEc_G^{F,m}]_G
@>>> [-,\vEc_G^{F,m}]_G \\
@VV{i\Phi}V @VV{s\Phi}V @VV{\Phi}V \\
i\VEct_G^F(-) @>>> s\VEct_G^F(-)
@>>> \VEct_G^F(-)
\end{CD}
\end{equation}
By Proposition \ref{classpace}, $\Phi$, $i\Phi$ and $s\Phi$ are isomorphisms on the category of $G$-CW-complexes for any $m \geq \dim G$,
We want to prove that this diagram consists solely of isomorphisms on the category of \emph{proper} $G$-CW-complexes
whenever $m \geq \dim(G)$.

\begin{lemme}\label{hilbert=vect}
Let $X$ be a finite proper $G$-CW-complex, and $G$
be a second countable Lie group.
Then $$i\VEct_G^F(X) \overset{\cong}{\longrightarrow}
s\VEct_G^F(X) \overset{\cong}{\longrightarrow}
\VEct_G^F(X).$$
\end{lemme}

\begin{proof}  Let $X$ be a finite proper $G$-CW-complex.
It suffices to prove the following statements.
\begin{enumerate}[(i)]
\item Any $G$-vector bundle over
$X$ can be equipped with a compatible structure of $G$-Hilbert bundle.
\item Any $G$-simi-Hilbert bundle can be equipped with a compatible structure of $G$-Hilbert bundle.
\item If two $G$-Hilbert bundles
$E \rightarrow X$ and $E' \rightarrow X$
are isomorphic as $G$-vector bundles, then they are isomorphic as
$G$-Hilbert bundles.
\end{enumerate}
Statement (i) was proven by Phillips (cf.\ example 3.4 of \cite{Phillips})
in the case $F=\mathbb{C}$, since $X$, being a finite proper $G$-CW-complex, is locally compact.
The same technique also yields the case $F=\R$.

Statement (ii) may be derived from statement (i). Indeed, given an $n$-dimensional $G$-simi-Hilbert bundle
$E \rightarrow X$, we want to find an equivariant section of the bundle $\bar{E} \rightarrow X$, where $\bar{E}_x$ is the linear family of
inner products on $E_x$ given by the $G$-simi-Hilbert space structure, for every $x\in X$.
This is a $(G,\mathbb{R}_+^*)$-principal bundle. If $E' \rightarrow X$ is a $(G,\mathbb{R}_+^*)$-principal bundle, then putting a $G$-Hilbert bundle
structure on the associated $G$-vector bundle is the same as finding a section of $E' \rightarrow X$. Hence, (ii) follows from (i) applied to the real
case for $1$-dimensional $G$-vector bundles.

Finally, let $E \overset{\pi}{\rightarrow} X$ and $E' \overset{\pi'}{\rightarrow} X$
be two $G$-Hilbert bundles, and $f : E \rightarrow E'$ an isomorphism of $G$-vector bundles.
Then $f(f^*f)^{-\frac{1}{2}}: E \rightarrow E'$ is
an isomorphism of $G$-Hilbert bundles. This proves statement (iii). \end{proof}

\begin{prop}\label{3compared}
Let $G$ be a second-countable Lie group
and let $m\geq \dim(G)$. \\
Then the maps
$$i\vEc_G^{F,m} \rightarrow s\vEc_G^{F,m} \rightarrow \vEc_G^{F,m}$$
are $G$-weak equivalences.
\end{prop}

\begin{cor}\label{3comparedcor}
Let $G$ be a second-countable Lie group and $m\geq \dim(G)$.
Then diagram (\ref{compared}) consists entirely of isomorphisms (on the category of
proper $G$-CW-complexes).
\end{cor}

\begin{proof}[Proof of Corollary \ref{3comparedcor}:]
Proposition \ref{3compared} shows that, for every proper $G$-CW-complex $X$, there are two isomorphisms:
$$[X,i\vEc_G^{F,m}]_G \overset{\cong}{\longrightarrow}
[X,s\vEc_G^{F,m}]_G \overset{\cong}{\longrightarrow}
[X,\vEc_G^{F,m}]_G.$$
This is equivalent to the condition that the two upper horizontal
morphisms in diagram (\ref{compared}) be isomorphisms on the category
of proper $G$-CW-complexes.
Therefore, it is also true for the lower horizontal morphisms. \end{proof}

\begin{proof}[Proof of Proposition \ref{3compared}:]
By Lemma \ref{hilbert=vect}, the lower horizontal morphisms
in (\ref{compared}) are isomorphisms on the category of
finite proper $G$-CW-complexes.
We deduce that the upper horizontal morphisms in (\ref{compared}) are
isomorphisms on the category of finite proper $G$-CW-complexes.
Therefore, for every $n \in \mathbb{N}$ and every compact subgroup
$H$ of $G$,
$$[G/H \times S^n, i\vEc_G^{F,m}]_G \overset{\cong}{\longrightarrow}
[G/H \times S^n, s\vEc_G^{F,m}]_G \overset{\cong}{\longrightarrow}
[G/H \times S^n, \vEc_G^{F,m}]_G,$$
hence
$$\pi_n\bigl((i\vEc_G^{F,m})^H\bigr) \overset{\cong}{\longrightarrow}
\pi_n\bigl((s\vEc_G^{F,m})^H\bigr) \overset{\cong}{\longrightarrow}
\pi_n\bigl((\vEc_G^{F,m})^H\bigr)$$
for every $n \in \N$ and every compact subgroup $H$ of $G$.
\end{proof}

\subsubsection{How the constructions depend on $m$}\label{4.7}

If $m$ and $m'$ are distinct elements of $\mathbb{N}^*\cup \{\infty\}$,
there are linear injections of $F^m$ into $F^{m'}$. We investigate here what they induce on the preceding constructions.

Let $\varphi$ and $\varphi'$ be two Hilbert $\Gamma$-bundles.
Any natural transformation $\eta: \varphi \rightarrow \varphi'$
induces natural transformations
$\eta^*: \varphi \Mod \rightarrow \varphi' \Mod$,
$i\eta^*: \varphi \imod \rightarrow \varphi' \imod$
and
$s\eta^*: \varphi \smod \rightarrow \varphi' \smod$,
which respectively induce homomorphisms of equivariant $\Gamma$-spaces
$\eta^*: \underline{\vEc}_G^\varphi \rightarrow \underline{\vEc}_G^{\varphi '}$,
$i\eta^*: i\underline{\vEc}_G^\varphi \rightarrow i\underline{\vEc}_G^{\varphi '}$,
and
$s\eta^*: s\underline{\vEc}_G^\varphi \rightarrow s\underline{\vEc}_G^{\varphi '}$.

\begin{lemme}\label{homotopyidentity}
Let $\varphi$ be a $\Gamma$-Hilbert bundle and $\eta: \varphi \rightarrow \varphi$
be a natural transformation such that $\mathcal{O}_\Gamma^F(\eta)=\id_{\Gamma}$.
Then there exists a natural transformation
$$\underline{\vEc}_G^{F,\varphi} \times I \rightarrow \underline{\vEc}_G^{F,\varphi}$$
which is an equivariant homotopy from $\eta^*$ to the identity map. \\
There are similar results for $i\eta^*$ and $s\eta^*$.
\end{lemme}

\begin{proof}
For every object $S$ of $\Gamma$, the map
$\begin{cases}
\Ob(\varphi(S) \Mod )=X_S & \longrightarrow \Mor(\varphi(S) \Mod) \\
x & \longmapsto ((\eta_s)_{|x})_{s \in S}
\end{cases}$ provides
a natural transformation from the identity of
$\varphi(S) \Mod$ to the functor
$\eta^*: \varphi(S) \Mod \rightarrow \varphi(S) \Mod$.\end{proof}

Let $\mathcal{H}$ and $\mathcal{H}'$
be two inner product spaces which are either finite-dimensional or isomorphic
to $F^{(\infty)}$,
together with an isometry
$\alpha: \mathcal{H} \hookrightarrow \mathcal{H}'$ and an injective map
$\beta: \mathbb{N} \hookrightarrow \mathbb{N}$.
Then, for every finite subset $A$ of $\mathbb{N}$, $\alpha$ and $\beta$ induce
a linear injection
$\begin{cases}
\mathcal{H}^A & \longrightarrow (\mathcal{H'})^{\beta(A)} \\
(x_a)_{a \in A} & \longmapsto \left(\alpha(x_{\beta^{-1}(b)})\right)_{b\in \beta(A)}
\end{cases}$ and then a strong morphism of Hilbert bundles
$$\begin{CD}
E_A(\mathcal{H}) @>>> E_{\beta(A)}(\mathcal{H}') \\
@VVV @VVV \\
B_A(\mathcal{H}) @>>> B_{\beta(A)}(\mathcal{H}').
\end{CD}$$
From those squares may be derived a natural transformation:
$$(\alpha,\beta) ^*: \text{Fib}^{\mathcal{H}} \longrightarrow \text{Fib}^{\mathcal{H}'}.$$
This yields natural transformations
$$(\alpha,\beta) ^*: \underline{\vEc}_G^{\text{Fib}^{\mathcal{H}}}
\longrightarrow \underline{\vEc}_G^{\text{Fib}^{\mathcal{H}'}} \quad; \quad
B(\alpha,\beta) ^*: B\vEc_G^{\text{Fib}^{\mathcal{H}}}
\longrightarrow B\vEc_G^{\text{Fib}^{\mathcal{H}'}},$$
$$i(\alpha,\beta) ^*: i\underline{\vEc}_G^{\text{Fib}^{\mathcal{H}}}
\longrightarrow i\underline{\vEc}_G^{\text{Fib}^{\mathcal{H}'}} \quad; \quad
Bi(\alpha,\beta) ^*: Bi\vEc_G^{\text{Fib}^{\mathcal{H}}}
\longrightarrow Bi\vEc_G^{\text{Fib}^{\mathcal{H}'}},$$
and
$$s(\alpha,\beta) ^*: s\underline{\vEc}_G^{\text{Fib}^{\mathcal{H}}}
\longrightarrow s\underline{\vEc}_G^{\text{Fib}^{\mathcal{H}'}} \quad; \quad
Bs(\alpha,\beta) ^*: Bs\vEc_G^{\text{Fib}^{\mathcal{H}}}
\longrightarrow Bs\vEc_G^{\text{Fib}^{\mathcal{H}'}}.$$

\begin{prop}\label{injectioninvariance}
Let $\alpha: \mathcal{H} \hookrightarrow \mathcal{H}'$ be an isometry
between two inner product spaces which are either finite-dimensional or
isomorphic to $F^{(\infty)}$, and let
$\beta:\mathbb{N} \hookrightarrow \mathbb{N}$ be an injection.
\begin{enumerate}[(a)]
\item
If $\mathcal{H}=\mathcal{H}'$, then
$B(\alpha,\beta) ^*$ is $G$-homotopic to the identity map, and so are
$Bi(\alpha,\beta)^*$ and $Bs(\alpha,\beta)^*$.
\item
If $\alpha': \mathcal{H} \hookrightarrow \mathcal{H}'$
is another isometry and $\beta': \mathbb{N} \hookrightarrow \mathbb{N}$
is another injection,
then $B(\alpha,\beta)^*$ and $B(\alpha',\beta')^*$ are $G$-homotopic
(and the same result holds for $Bi(\alpha,\beta) ^*$ and $Bi(\alpha',\beta')$ on the one hand, and
$Bs(\alpha,\beta) ^*$ and $Bs(\alpha',\beta')^*$ on the other hand).
\item We have a commutative diagram:
$$\xymatrix{
\bigl[-,\vEc_G^{\text{Fib}^\mathcal{H}(\mathbf{1})}\bigr]_G \ar[rd] \ar[rr]^{(\alpha,\beta)^*} & &
\bigl[-,\vEc_G^{\text{Fib}^\mathcal{H'}(\mathbf{1})}\bigr]_G \ar[ld] \\
& \mathbb{V}\text{ect}_G^F(-). &
}$$
\end{enumerate}
\end{prop}

\begin{proof}
We only prove the result in the case of
$B(\alpha,\beta)^*$, since the other cases may be treated in a strictly identical manner.
\begin{enumerate}[(a)]
\item We apply Lemma \ref{homotopyidentity} to the natural transformation $(\alpha,\beta)^*$. The homotopy from $B(\alpha,\beta)^*$ to the identity map
of $B\vEc_G^\varphi$ is obtained by taking the geometric realization of the natural transformation
$\underline{\vEc}_G^\varphi \times I \rightarrow
\underline{\vEc}_G^\varphi$.
\item Assume, for example, that $\#(\mathbb{N}\setminus \beta'(\mathbb{N})) \geq \#(\mathbb{N}\setminus \beta(\mathbb{N}))$.
We then choose an injection $\beta'':\mathbb{N} \hookrightarrow \mathbb{N}$ such that $\beta'=\beta'' \circ \beta$.
Thus $B(\alpha,\beta')=B(\id_{\mathcal{H}'},\beta'') \circ B(\alpha,\beta)$, and we deduce from (a) that
$B(\alpha,\beta')$ is $G$-homotopic to $B(\alpha,\beta)$. It thus suffices to prove the result when $\beta=\beta'$.

Assume now that $\beta=\beta'$. If $\mathcal{H}$ is isomorphic to $F^{(\infty)}$, then $\mathcal{H'}$ is also isomorphic to $F^{(\infty)}$, and the
result follows from (a). Assume finally that $\mathcal{H}$ is finite-dimensional. Then there exists an isometry
$\alpha'':\mathcal{H}' \overset{\cong}{\rightarrow} \mathcal{H}'$ such that
$\alpha'=\alpha'' \circ \alpha$. Thus $B(\alpha',\beta)=B(\alpha'',\id_{\mathbb{N}}) \circ B(\alpha,\beta)$, and we deduce from
(a) that $B(\alpha',\beta)$ is $G$-homotopic to $B(\alpha,\beta)$.

\item Lemma \ref{strong2strong} shows indeed that $(\alpha,\beta)^*$ induces a strong morphism
of $G$-vector bundles:
$$\begin{CD}
E\vEc_G^{\text{Fib}^\mathcal{H}(\mathbf{1})} @>>> E\vEc_G^{\text{Fib}^\mathcal{H'}(\mathbf{1})} \\
@VVV @VVV \\
\vEc_G^{\text{Fib}^\mathcal{H}(\mathbf{1})} @>>> \vEc_G^{\text{Fib}^\mathcal{H'}(\mathbf{1})}.
\end{CD}$$
\end{enumerate}
\end{proof}

\begin{Rem}
If $m \geq \dim(G)$, then $B(\alpha,\beta)^*$, $Bi(\alpha,\beta)^*$
and $Bs(\alpha,\beta)^*$ are $G$-weak equivalences. This will only be used in the case $\mathcal{H}=\mathcal{H'}$, and
it is then an immediate consequence of Proposition \ref{injectioninvariance} (notice that the assumption $\dim(G) \leq m$ is useless
in this case), so we leave the general case as an easy exercise.
\end{Rem}

\section{Equivariant K-theory}\label{6}

In this chapter, $F=\mathbb{R}$ or $\mathbb{C}$.
Recall that any $\Gamma-G$-space $\underline{A}$ has an underlying structure of
simplicial $G$-space, and thus has a thick geometric realization $BA$, called the classifying space of
$\underline{A}$, which inherits a structure of $G$-space.

For any Lie group $G$, and any $m \in \N^* \cup \{\infty\}$,
we set:
$$KF_G^{[m]}:=\Omega B\vEc_G^{F,m} \quad; \quad  iKF_G^{[m]}:=\Omega Bi\vEc_G^{F,m} \quad; \quad
sKF_G^{[m]}:=\Omega Bs\vEc_G^{F,m}.$$
In Section \ref{6.1}, we define the equivariant K-theory $KF_G^*(-)$ as the good equivariant cohomology theory
in negative degrees classified by the equivariant H-space $KF_G^{[\infty]}$. In Section \ref{6.2}, we construct a natural
transformation $\gamma: \mathbb{K}F_G^*(-) \rightarrow KF_G^*(-)$, and then prove in Section \ref{6.3} that
$\gamma_X$ is an isomorphism whenever $X=(G/H)\times Y$, where $H$ is a compact subgroup of $G$ and
$Y$ is a finite CW-complex on which $G$ acts trivially.
The proof is based upon Segal's theorem on $\Gamma$-spaces. Alternative classifying spaces are discussed in Section \ref{6.4}.
Finally, Sections \ref{7} and \ref{7.6} are devoted to the construction of product structures on $KF_G^*(-)$,
on Bott periodicity and the extension to positive degrees.

\subsection{Equivariant K-theory in negative degrees}\label{6.1}

\begin{Def}
Let $G$ be a Lie group, $F=\mathbb{R}$ or $\mathbb{C}$,
$(X,A)$ a $G$-CW-pair and $n \in \mathbb{N}$.
We set:
$$KF_G^{-n}(X,A):=\bigl[\Sigma ^n(X/A),KF_G^{[\infty]}\bigr]_G^\bullet,$$
and
$$KF_G^{-n}(X):=KF_G^{-n}(X\cup \{*\},\{*\}).$$
In particular, for every $G$-CW-complex $X$,
$$KF_G(X):=KF_G^0(X)=\bigl[X,KF_G^{[\infty]}\bigr]_G.$$
\end{Def}

\begin{prop}
$KF_G^*(-)$ is a good equivariant cohomology theory in negative degrees
on the category of $G$-CW-pairs.
\end{prop}

\begin{proof}
Both the excision property and the invariance under equivariant homotopy are obvious.
The long exact sequence of a $G$-CW-pair is derived from
the equivariant Puppe sequence associated to a $G$-cofibration, since the inclusion $A \subset X$ is a $G$-cofibration for every $G$-CW-pair $(X,A)$
(for a reference on the non-equivariant Puppe sequence, cf.\ p.398 of \cite{Hatcher}).
\end{proof}

\begin{Rem}\label{kernelrem}
It follows from the properties of negatively-graded good equivariant
cohomology theories that, for every $G$-CW-pair $(X,A)$
and every $n \in \mathbb{N}$, we have a natural isomorphism
$$
KF_G^{-n}(X) \cong \Ker\left[ KF_G(S^n \times X)
\longrightarrow KF_G(X)\right],
$$
induced by the projection $S^n \times X \rightarrow (S^n \times X )/(\{*\} \times X)$,
and a natural isomorphism
$$
KF_G^{-n}(X,A) \cong \Ker\left[ KF_G^{-n}(X\cup _A X)
\longrightarrow KF_G^{-n}(X)\right],
$$
induced by the projection $X \cup _A X \rightarrow X/A$.
\end{Rem}

\subsection{The natural transformation $\gamma: \mathbb{K}F_G^*(-) \rightarrow KF_G^*(-)$}\label{6.2}

Let $X$ be a $G$-CW-complex.
Then the canonical map $i:\vEc_G^{F,\infty} = \underline{\vEc}_G^{F,\infty}(\mathbf{1}) \rightarrow KF_G^{[\infty]}$
(which is a homomorphism of equivariant H-spaces)
induces a natural homomorphism $[X,\vEc_G^{F,\infty}]_G \overset{i_*}{\rightarrow} [X,KF_G^{[\infty]}]_G$
of abelian monoids, which, pre-composed with the inverse of the natural isomorphism
$[X,\vEc_G^{F,\infty}]_G \overset{\cong}{\rightarrow} \mathbb{V}\text{ect}_G^{F}(X)$, yields
a natural homomorphism of abelian monoids $$\mathbb{V}\text{ect}_G^{F}(X) \longrightarrow [X,KF_G^{[\infty]}]_G.$$
Since $ [X,KF_G^{[\infty]}]_G$ is an abelian group for any $G$-space $X$, we deduce from the universal property
of the Grothendieck construction that this morphism induces a natural homomorphism of abelian groups
$$\gamma _X: \mathbb{K}F_G(X) \longrightarrow KF_G(X).$$
This clearly defines a natural transformation
$\gamma: \mathbb{K}F_G(-) \longrightarrow KF_G(-)$
on the category of $G$-CW-complexes. For every $G$-CW-complex $X$ and every $n \in \mathbb{N}$, the
inclusion of $X$ into $S^n \times X$ induces a commutative square:
$$\begin{CD}
\mathbb{K}F_G(S^n \times X) @>>> \mathbb{K}F_G(X) \\
@V{\gamma _{S^n \times X}}VV @V{\gamma _{X}}VV \\
KF_G(S^n \times X) @>>> KF_G(X).
\end{CD}
$$
However, we know from Remark \ref{kernelrem} that $KF_G^{-n}(X) \cong \ker\left[ KF_G(S^n \times X)
\longrightarrow KF_G(X)\right]$ whilst, by definition,
$\mathbb{K}F_G^{-n}(X) = \Ker\left[ \mathbb{K}F_G(S^n \times X)
\longrightarrow \mathbb{K}F_G(X)\right]$ (cf.\ \cite{Bob1} definition 3.1).
The preceding square thus induces a functorial homomorphism:
$$\gamma _X^{-n}: \mathbb{K}F_G^{-n}(X) \longrightarrow KF_G^{-n}(X).$$
Finally, given a $G$-CW-pair $(X,A)$ and an integer $n \in
\mathbb{N}$, the inclusion of $X$ into $X \cup_A X$ induces a commutative square:
$$\begin{CD}
\mathbb{K}F_G^{-n}(X\cup _A X) @>>> \mathbb{K}F_G^{-n}(X) \\
@V{\gamma _{X\cup _A X}^{-n}}VV @V{\gamma _X^{-n}}VV \\
KF_G^{-n}(X\cup _A X) @>>> KF_G^{-n}(X).
\end{CD}$$
Again, Remark \ref{kernelrem} shows $KF_G^{-n}(X,A) \cong \Ker\left[ KF_G^{-n}(X\cup _A X)
\longrightarrow KF_G^{-n}(X)\right]$ whilst, by definition,
$\mathbb{K}F_G^{-n}(X,A) = \Ker\left[ \mathbb{K}F_G^{-n}(X\cup _A X)
\longrightarrow \mathbb{K}F_G^{-n}(X)\right]$. The previous square
thus induces a functorial homomorphism:
$$\gamma _{X,A}^{-n}:\mathbb{K}F_G^{-n}(X,A) \longrightarrow KF_G^{-n}(X,A).$$
This completes the definition of $\gamma$ as a natural transformation between negatively-graded
equivariant cohomology theories on the category of $G$-CW-complexes.

\subsection{The coefficients of $KF_G^*(-)$}\label{6.3}

The next results will show that $KF_G^*(-)$ deserves the label ``equivariant K-theory".

\begin{prop}\label{isomorphism}
Let $H$ be a compact subgroup of $G$, and $Y$ be a finite CW-complex
on which $G$ acts trivially.
Then the map
$$\gamma_{(G/H) \times Y}:
\mathbb{K}F_G((G/H) \times Y) \overset{\cong}{\longrightarrow}
\bigl[(G/H) \times Y,KF_G^{[\infty]}\bigr]_G$$
is an isomorphism.
\end{prop}

\begin{cor}\label{isomorphisms}
For every compact subgroup $H$ of $G$, every
$n\in \mathbb{N}$ and every finite CW-complex $Y$ on which $G$ acts
trivially, the map
$$\gamma _{(G/H) \times Y}^{-n}: \mathbb{K}F_G^{-n}((G/H) \times Y)
\overset{\cong}{\longrightarrow} KF_G^{-n}((G/H) \times Y)$$
is an isomorphism.
\end{cor}

\begin{proof}[Proof of Corollary \ref{isomorphisms}:]
The case $n=0$ is precisely the result of Proposition \ref{isomorphism}.
For $n >0$, we notice that if $Y$ is a finite CW-complex on which $G$ acts trivially,
then so is $S^n \times Y$.
We thus have a commutative square
$$\begin{CD}
\mathbb{K}F_G(S^n \times ((G/H) \times Y)) @>>> \mathbb{K}F_G((G/H) \times Y) \\
@V{\cong}V{\gamma_{S^n \times ((G/H)\times Y)}}V @V{\cong}V{\gamma_{(G/H)\times Y}}V \\
KF_G(S^n \times ((G/H) \times Y)) @>>> KF_G((G/H) \times Y).
\end{CD}$$
We deduce that $\gamma_{(G/H) \times Y}^{-n}$ is an isomorphism. \end{proof}

In order to prove Proposition \ref{isomorphism}, we need to establish the following result:

\begin{prop}\label{representability}
Let $H$ be a compact subgroup of a Lie group $G$. Then $\mathbb{K}F_G((G/H) \times -)$
is classified by an {\rm H}-space as a functor
from the category of finite CW-complexes to the category of abelian
groups.
\end{prop}

\begin{proof}
Notice first that for any finite CW-complex $Y$
on which $G$ acts trivially, there is a natural isomorphism
$$\begin{cases}
\mathbb{V}\text{ect}_H^F(Y) & \overset{\cong}{\longrightarrow}
\mathbb{V}\text{ect}_G^F((G/H) \times Y) \\
\left[E\right] & \longmapsto \left[G \times _H E \right].
\end{cases}
$$
This yields a functorial isomorphism
$$\mathbb{K}F_H(-)
\overset{\cong}{\longrightarrow}
\mathbb{K}F_G((G/H) \times -)$$
on the category of finite CW-complexes on which $G$ acts trivially.
It thus all comes down to proving that $\mathbb{K}F_H$
is classified by an H-space over the category of
finite CW-complexes on which $H$ acts trivially.
In the rest of the proof, we let $J$ denote the set of isomorphism classes of irreducible finite-dimensional
linear representation of $H$ with ground field $F$.
Since $H$ is a compact group, we know that $J$ is countable. \\
$\bullet$ \textbf{The case $F=\C$.}
For every $j \in J$, we choose some $V_j$ in the class $j$.
For any finite CW-complex $Y$ on which $H$ acts trivially,
we have a canonical $H$-vector bundle
$\xi _{V_j}(Y): V_j \times Y \rightarrow Y$ over $Y$.
Next, we consider the two morphisms
$$\begin{cases}
\mathbb{V}\text{ect}_H(Y) & \longrightarrow \underset{j \in J}{\bigoplus}
\mathbb{V}\text{ect}(Y) \\
\xi & \longmapsto \underset{j\in J}{\oplus}\Hom (\xi _{V_j}(Y),\xi)
\end{cases} \quad \text{and} \quad \begin{cases}
\underset{j\in J}{\bigoplus}
\mathbb{V}\text{ect}(Y) & \longrightarrow \mathbb{V}\text{ect}_H(Y) \\
(\xi_j)_{j\in J} & \longmapsto \underset{j \in J}{\oplus}(\xi_j \otimes \xi_{V_j}(Y)).
\end{cases}$$
The theory of linear representations of compact groups shows that these homomorphisms
are well-defined and inverse one to the other.
They are also functorial, and so yield a natural isomorphism
$$\mathbb{K}_H(Y) \overset{\cong}{\longrightarrow}
\underset{j \in J}{\oplus}\mathbb{K}(Y).$$
We choose an H-space $B$ which classifies $\mathbb{K}(-)$ on the category of compact spaces, with a strict unit
$e$ (i.e. $\forall x \in B,\, x.e=e.x=x$), e.g. $B=\mathbb{Z} \times \underset{n\in \mathbb{N}}{\Indlim} B\GL_n(\mathbb{C})$.
If $J$ is finite of order $N$, then $B^N$ is the sought H-space. \\
Assume now that $J$ is infinite. Then
 $\mathbb{K}F_H(Y) \overset{\cong}{\rightarrow}
\underset{n \in \mathbb{N}}{\oplus}\mathbb{K}(Y)$
for every finite CW-complex $Y$ with trivial action of $H$.
We define $B^{(\infty)}$
as the homotopy colimit of the injections
$\begin{cases}
B^n & \longrightarrow B^{n+1} \\
(x_1,\dots,x_n) & \longmapsto (x_1,\dots,x_n,e)
\end{cases}$ for $n \in \mathbb{N}$. An element of $B^{(\infty)}$ may be considered
as a family $((x_k)_{k\in \mathbb{N}},t)\in B^\mathbb{N} \times [0,+\infty[$, where
$x_k=e$ when $k> [t]$.
Let $X$ be a compact space. We define:
$$\alpha_X: \begin{cases}
\underset{n \in \mathbb{N}}{\oplus}[X,B] & \longrightarrow [X,B^{(\infty)}] \\
([f_1],\dots,[f_n],0,\dots) & \longmapsto \left[ \begin{cases}
X & \rightarrow B^n \hookrightarrow B^{(\infty)} \\
x & \mapsto \bigl((f_1(x),\dots,f_n(x)),n\bigr)
\end{cases}
\right].
\end{cases}$$
Since $X$ is compact,  any map
$X \rightarrow B^{(\infty)}$ may be factored through $X \rightarrow B^n
\hookrightarrow B^{(\infty)}$ up to homotopy for some $n \in \mathbb{N}$, and
similarly for any homotopy $X \times I \rightarrow B^{(\infty)}$. We deduce that $\alpha_X$ is a bijection.

Finally, we define a multiplicative structure on $B^{(\infty)}$ by:
$$\begin{cases}
B^{(\infty)} \times B^{(\infty)} & \longrightarrow B^{(\infty)} \\
\left[ \left((x_1,\dots,x_m,\dots),t\right),\left((y_1,\dots,y_n,\dots),t'\right) \right] &
\longmapsto \left[(x_1.y_1,\dots,x_i.y_i,\dots),\sup(t,t')\right].
\end{cases}$$
This map is well-defined, continuous, and $(e,0)$ is
a (strict) unit for the H-space $B^{(\infty)}$ (in particular $(e,0).(e,0)=(e,0)$).
We easily see that $\alpha_X$ is a homomorphism of abelian monoids.
Since this isomorphism is functorial, we deduce that
$B^{(\infty)}$ classifies $\mathbb{K}F_H$ on the category of finite CW-complexes with trivial action of $H$. \\
$\bullet$ \textbf{The case $F=\R$.}
For any $j \in J$ and an arbitrary $V_j$ in $j$, we set
$D_j:=\End_G(V_j)$. Then $D_j$ is isomorphic to
$\mathbb{R}$, $\mathbb{C}$ or $\mathbb{H}$ (see Proposition A.1 of \cite{Bob3}).
For $F_1=\mathbb{R}, \mathbb{C}, \mathbb{H}$, we let $J_{F_1}$
denote the subset of those elements $j$ in $J$ such that $D_j \cong F_1$.

Then, by an argument similar to the one in the complex case,
we obtain a natural isomorphism
$$\mathbb{K}O_H(Y) \overset{\cong}{\longrightarrow}
\left(\underset{j \in J_{\mathbb{R}}}{\oplus}\mathbb{K}O(Y)\right)
\oplus
\left(\underset{j \in J_{\mathbb{C}}}{\oplus}\mathbb{K}(Y)\right)
\oplus
\left(\underset{j \in J_{\mathbb{H}}}{\oplus}\mathbb{K}Sp(Y)\right)
$$
on the category of finite CW-complexes with trivial action of $H$.
The proof then runs similarly to the one in the complex case, but this time we consider
a classifying space $B_{F_1}$ of $\mathbb{K}F_1(-)$ for each $F_1=\mathbb{R}, \mathbb{C}$ and $\mathbb{H}$. \end{proof}

We will also need the following theorem: its very technical proof is given in appendix \ref{Proofhomotype}:

\begin{theo}\label{homotopytype}
Let $G$ denote a Lie group, and $H$ be a compact subgroup of $G$. Then
 $(\vEc_G^{F,\infty})^H$, $(i\vEc_G^{F,\infty})^H$ and $(s\vEc_G^{F,\infty})^H$
all have the homotopy type of a CW-complex.
\end{theo}

\begin{proof}[Proof of Proposition \ref{isomorphism}:]
Given a compact subgroup $H$ of $G$,
$(\underline{\vEc}_G^{F,\infty}(-))^H$ is a $\Gamma$-space.
By Theorem \ref{homotopytype}, $(\underline{\vEc}_G^{F,\infty}(\mathbf{1}))^H=(\vEc_G^{F,\infty})^H$
has the homotopy type of a CW-complex.
Proposition \ref{classpace} shows that
$$\pi_0((\vEc_G^{F,\infty})^H)
=[*,(\vEc_G^{F,\infty})^H]=[G/H,\vEc_G^{F,\infty}]_G \cong \VEct_G^F(G/H) \cong \VEct_H^F(*)=\Rep_F(H).$$
hence $\Rep_F(H)\cong \pi_0((\vEc_G^{F,\infty})^H)$ (i.e. the monoid of isomorphism classes of finite-dimensional linear representations of $F$) contains a free and cofinal submonoid: the one
generated by the irreducible representations of $H$.
Finally, $\Omega B(\underline{\vEc}_G^{F,\infty})^H = (KF_G^{[\infty]})^H$.

We deduce from Proposition 4.1 of \cite{Segal-cat}, applied to the $\Gamma$-space $(\underline{\vEc}_G^{F,\infty})^H$,
that,  for every contravariant functor $F$ from the category of finite
CW-complexes to $\mathcal{A}b$
which is classified by an H-space, and
every natural transformation $\eta: [- , (\vEc_G^{F,\infty})^H]
\longrightarrow F$, there is a unique natural transformation
$\eta ': [-, (KF_G^{[\infty]})^H] \longrightarrow F$
which extends $\eta$. In particular, any natural transformation
$[- , (KF_G^{[\infty]})^H] \longrightarrow [- , (KF_G^{[\infty]})^H]$
under $[- , (\vEc_G^{F,\infty})^H]$ is the identity.

For any finite CW-complex $Y$ with trivial action of $G$, we deduce from Proposition \ref{classpace} that
$$[Y,(\vEc_G^{F,\infty})^H]=[(G/H) \times Y, \vEc_G^{F,\infty}]_G \cong
\mathbb{V}\text{ect}_G^F((G/H) \times Y).$$

By Proposition \ref{representability}, $\mathbb{K}F_G((G/H) \times - )$
is classified by an H-space on the category of finite CW-complexes on which $G$ acts trivially.
We deduce that there is a unique natural transformation
$\delta: [-, (KF_G^{[\infty]})^H] \rightarrow \mathbb{K}F_G((G/H) \times - )$
which extends $[-,(\vEc_G^{F,\infty})^H] \longrightarrow \mathbb{K}F_G((G/H) \times - )$.

On the other hand, by the universal property of the Grothendieck construction,
for every contravariant functor $F$ from the category of finite CW-complexes
to $\mathcal{A}b$, and every natural transformation
$\mathbb{V}\text{ect}_G^F((G/H) \times - ) \longrightarrow F$,
there exists a unique natural transformation
$\mathbb{K}F_G((G/H) \times - ) \longrightarrow F$ which is compatible with the preceding one.
In particular, any natural transformation $\mathbb{K}F_G((G/H) \times - )
\longrightarrow \mathbb{K}F_G((G/H) \times - )$ under
$\mathbb{V}\text{ect}_G^F((G/H) \times -) \cong [- ,(\vEc_G^{F,\infty})^H]$ is the identity.
From there, it easily follows that
$\gamma_{(G/H) \times Y}$ is an isomorphism for any finite CW-complex $Y$ on which
$G$ acts trivially.
\end{proof}

\subsection{Alternative classifying spaces}\label{6.4}

The space $KF_G^{[\infty]}$ is not the only relevant classifying
space for the functor $KF_G(-)$ on the category of proper $G$-CW-complexes.
Indeed, $iKF_G^{[\infty]}$ and $sKF_G^{[\infty]}$ (resp.\ $KF_G^{[m]}$, for $m\in \mathbb{N}^*$)
also qualify when $G$ is second countable (resp.\ discrete).

\begin{prop}\label{weaks}
Whenever $\pi_0(G)$ is countable (i.e. $G$ is second-countable),
the sequence of morphisms of equivariant $\Gamma$-spaces $i\underline{\vEc}_G^{F,\infty} \rightarrow
s\underline{\vEc}_G^{F,\infty} \rightarrow \underline{\vEc}_G^{F,\infty}$
induces a sequence of $G$-weak equivalences
$iKF_G^{[\infty]} \rightarrow sKF_G^{[\infty]} \rightarrow KF_G^{[\infty]}$
hence a sequence of isomorphisms
$$[X,iKF_G^{[\infty]}]_G \overset{\cong}{\longrightarrow}
[X,sKF_G^{[\infty]}]_G \overset{\cong}{\longrightarrow}
KF_G(X)$$ for every proper $G$-CW-complex $X$.
It follows that $iKF_G^{[\infty]}$ and $sKF_G^{[\infty]}$ are classifying spaces for $KF_G(-)$ whenever
$G$ is second-countable.
\end{prop}

\begin{proof}
It suffices to check that
the maps $(iKF_G^{[\infty]})^H \rightarrow (sKF_G^{[\infty]})^H$
and $(sKF_G^{[\infty]})^H \rightarrow (KF_G^{[\infty]})^H$
are homotopy equivalences for every compact subgroup $H$ of $G$.
It follows that we only need to make sure that
$B(s\vEc_G^{F,\infty})^H \rightarrow B(\vEc_G^{F,\infty})^H$ and
$B(i\vEc_G^{F,\infty})^H \rightarrow B(s\vEc_G^{F,\infty})^H$
are homotopy equivalences. Using statement (ii) in proposition A.1 of \cite{Segal-cat}
and the definition of a $\Gamma-G$-space, we are thus reduced to proving that the maps
$(i\vEc_G^{F,\infty})^H \rightarrow (s\vEc_G^{F,\infty})^H \rightarrow (\vEc_G^{F,\infty})^H$
are homotopy equivalences. However, by Proposition \ref{3compared}, the maps
$i\vEc_G^{F,\infty} \rightarrow s\vEc_G^{F,\infty} \rightarrow \vEc_G^{F,\infty}$
are $G$-weak equivalences, hence the maps
$(i\vEc_G^{F,\infty})^H \rightarrow (s\vEc_G^{F,\infty})^H
\rightarrow (\vEc_G^{F,\infty})^H$ are weak equivalences.
By Theorem \ref{homotopytype} and Whitehead's theorem (cf.\ Theorem 4.5 of \cite{Hatcher}),
those maps are actually homotopy equivalences.
\end{proof}

\begin{prop}\label{mdependance}
Let $G$ be a discrete group and $m\in \mathbb{N}^*$. Then the morphism
$KF_G^{[m]} \rightarrow KF_G^{[\infty]}$ in $\text{CG}_{G}^{h\bullet}$
(which is well-defined by Lemma \ref{injectioninvariance}) is a $G$-weak equivalence,
and, for every proper $G$-CW-complex, it gives rise to an isomorphism
$[X,KF_G^{[m]}]_G \overset{\cong}{\longrightarrow} KF_G(X)$.
Hence $KF_G^{[m]}$ is a classifying space for $KF_G(-)$.
\end{prop}

\begin{proof}
We choose a linear injection $\alpha: F^m \hookrightarrow F^{(\infty)}$.
It suffices to check that, for every compact subgroup $H$ of $G$,
the map $(\vEc_G^{F,m})^H \rightarrow (\vEc_G^{F,\infty})^H$ induced
by $(\alpha,\id_\mathbb{N})$ is a homotopy equivalence, since $(\underline{\vEc}_G^{F,m})^H$ is a $\Gamma$-space.
By Theorem \ref{homotopytype} and Whitehead's theorem, we are reduced to proving that
$\vEc_G^{F,m} \rightarrow \vEc_G^{F,\infty}$ is a $G$-weak equivalence,
i.e. that
$$[X,\vEc_G^{F,m}]_G \longrightarrow [X,\vEc_G^{F,\infty}]_G$$
is an isomorphism for every proper $G$-CW-complex $X$.
Since $G$ is discrete, we may apply Proposition \ref{classpace} twice to
obtain that $[X,\vEc_G^{F,m}]_G \overset{\cong}{\rightarrow} \mathbb{V}\text{ect}_G^F(X)$
and $[X,\vEc_G^{F,\infty}]_G \overset{\cong}{\rightarrow} \mathbb{V}\text{ect}_G^F(X)$
are isomorphisms for any proper $G$-CW-complex $X$. The result then follows from point (c)
of Proposition \ref{injectioninvariance}. \end{proof}

\subsection{Products in equivariant K-theory}\label{7}

Here, we will follow the ideas from Section 2 of \cite{Bob2}. For every pair $(G,H)$ of Lie groups, we build a pairing
$KF_G^{[\infty]} \wedge KF_H^{[\infty]} \rightarrow KF_{G \times H}^{[\infty]}$, i.e. a morphism in the category
$\text{CG}_{G \times H}^{h\bullet}[W_{G \times H}^{-1}]$. This pairing is obtained by constructing functorial
maps $$\underline{\vEc}_G^{F,\infty}(S) \times \underline{\vEc}_H^{F,\infty}(T) \rightarrow \underline{\vEc}_{G\times H}^{F,\infty}(S \times T)$$
induced by functors
$$\varphi(S) \Mod \times \varphi(T) \Mod \longrightarrow \varphi(S \times T) \Mod,$$
where $\varphi=\text{Fib}^{F^{(\infty)}}$.
A cautious reader will find a difference in the definition of pairings between
ours and that of \cite{Bob2}. In fact, the definition adopted in \cite{Bob2} carries a sign error, and therefore cannot lead to the claimed results. There is another problem in \cite{Bob2}:
there was a misinterpretation of some construction of Segal which lead to the false claim that
there is a natural \emph{homotopy equivalence} $\Omega B A \rightarrow \Omega^2 B^2 A$, when
$\underline{A}$ is a good $\Gamma-G$-space. This is fixed in Section \ref{C} of the appendix.

\subsubsection{Hilbert $\Gamma$-bundles with a product structure}

For any $n \in \N^*$, we define a functor
$\pi^{(n)}_{\Gamma}: \Gamma^n \rightarrow \Gamma$ as follows:
$$\pi^{(n)}_{\Gamma}: \begin{cases}
(S_i)_{1\leq i \leq n} & \longmapsto \underset{i=1}{\overset{n}{\prod}}S_i \\
\left(f_i:S_i \rightarrow T_i\right)_{1 \leq i \leq n} &
\longmapsto \left(\pi^{(n)}_\Gamma((f_i)_{1 \leq i \leq n}): \{(s_i)_{1 \leq i\leq n}\}
\mapsto \underset{1\leq i \leq n}{\prod}f_i(\{s_i\}) \right)
\end{cases}$$
Given an list $(x_1,\dots,x_n)$ of objects of $\Gamma \text{-Fib}_F$, with
$x_i=(S_i,X_i,(p_j)_{j\in S_i})$, set $$\underset{i=1}{\overset{n}{\prod}}x_i:=
\left(\underset{i=1}{\overset{n}{\prod}} S_i,
\underset{i=1}{\overset{n}{\prod}} X_i , (p_{j_1} \otimes \dots \otimes p_{j_n})
_{j_1 \in S_1,\dots,j_n \in S_n}\right)$$
For every $n \in \N^*$, these products yield a functor
$\pi^{(n)}: (\Gamma \text{-Fib}_F)^n \longrightarrow \Gamma \text{-Fib}_F$
defined by
$$\begin{cases}
(\varphi_i)_{1 \leq i \leq n} & \longmapsto \underset{i=1}{\overset{n}{\prod}}\varphi_i \\
\left((\gamma_i,f_i,(f^i_t)_{t \in T_i})\right)_{1 \leq i \leq n} & \longmapsto
\left(\pi^{(n)}_\Gamma((\gamma_i)_{1 \leq i \leq n}),\underset{i=1}{\overset{n}{\prod}}f_i,
\left(\underset{i=1}{\overset{n}{\otimes}}f^i_{t_i}\right)_{(t_i)_{1 \leq i \leq n}
\in \underset{i=1}{\overset{n}{\prod}}T_i}\right)
\end{cases}$$

Let $\varphi:\Gamma \rightarrow \Gamma \text{-Fib}_F$ be a Hilbert
$\Gamma$-bundle, and $n \geq 2$ be an integer.
There are two ``natural" functors $\varphi \circ \pi^{(n)}_{\Gamma}$ and $\pi^{(n)} \circ \varphi^n$
from $\Gamma^{(n)}$ to $\Gamma \text{-Fib}_F$. The
diagram
$$\begin{CD}
\Gamma ^n @>{\varphi^n}>> (\Gamma \text{-Fib}_F)^n \\
@V{\pi_\Gamma^{(n)}}VV @VV{\pi^{(n)}}V \\
\Gamma @>{\varphi}>> \Gamma \text{-Fib}_F
\end{CD}$$
might not be commutative, which motivates the next definition.

\begin{Def}
If $n$ is an integer greater or equal to $2$, and $\varphi$ a Hilbert $\Gamma$-bundle,
an \defterm{$n$-fold product structure} on $\varphi$ is a natural transformation
$$m: \pi^{(n)} \circ \varphi^n \rightarrow \varphi \circ \pi_\Gamma^{(n)}$$
such that $\mathcal{O}_\Gamma^F(m)=\id_\Gamma^{(n)}$.
A $2$-fold product structure is simply called a \defterm{product structure}.
\end{Def}

A product structure on a Hilbert $\Gamma$-bundle $\varphi$ is simply the data, for every pair $(S,T)$
of finite sets, of a morphism $\varphi(S)\times \varphi(T) \rightarrow \varphi(S \times T)$
in the category $\Gamma \text{-Fib}_F$, with some additional compatibility conditions.

\begin{Def}
Let $n$ be an integer greater or equal to $2$, $\varphi$ be a Hilbert $\Gamma$-bundle, and
$m$ and $m'$ be two $n$-fold product structures on $\varphi$. \\
We say that $m$ \defterm{maps to} $m'$ if and only if there exists a natural transformation
$\eta: \varphi \circ \pi_\Gamma^{(n)} \rightarrow \varphi \circ \pi_\Gamma^{(n)}$
such that $\eta \circ m=m'$. \\
We say that $m$ and $m'$ are \defterm{equivalent} if and only if they belong to the same class
for the equivalence relation (on the set of $n$-fold product structures on $\varphi$) generated by the binary relation
``map to''.
\end{Def}

Given two product structures $m$ and $m'$ on a Hilbert $\Gamma$-bundle $\varphi$, $m$ maps to $m'$
 iff we have the data, for every pair of finite sets $(S,T)$, of a morphism
$\eta_{(S,T)}:\varphi(S \times T) \rightarrow  \varphi(S \times T)$ such that the diagram
$$\xymatrix{
& \varphi(S\times T) \ar[dd]^{\eta_{(S,T)}} \\
\varphi(S)\times \varphi(T) \ar[ur]^{m_{(S,T)}} \ar[dr]_{m'_{(S,T)}} \\
& \varphi(S\times T)
}$$
commutes, with some additional compatibility conditions between the $\eta_{(S,T)}$'s.

\begin{Def}
We define the transposition (or ``twist") functor $$T_\Gamma:
\begin{cases}
\Gamma \times \Gamma & \longrightarrow \Gamma \times \Gamma \\
(S,T) & \longmapsto (T,S) \\
(\gamma_1,\gamma_2) & \longmapsto (\gamma_2,\gamma_1).
\end{cases}
$$
\end{Def}

\begin{Def}
Let $\Phi$ and $\Psi$ be two functors from $\Gamma^2$ to $\Gamma \text{-Fib}_F$, and
$f: \Phi \rightarrow \Psi$ be a natural transformation such that
$\mathcal{O}_\Gamma^F(f_{(S,T)})=\id_{T \times S}$ for every pair $(S,T)$ of finite sets. \\
We define $Tf$ as a natural transformation between functors from $\Gamma^2$ to
$\Gamma \text{-Fib}_F$ as follows:
for every pair $(S,T)$ of finite sets, if $f_{(S,T)}=(\id_{T \times S},g,(g_{t,s})_{(t,s) \in T \times S})$, then \\
$$(Tf)_{(S,T)}:=(\id_{S \times T},g,(g_{t,s})_{(s,t) \in S \times T}).$$
\end{Def}

If $\varphi$ is a Hilbert $\Gamma$-bundle, and $m$ is a product structure on $\varphi$, one can easily check that
$T(m \circ T_\Gamma)$ is another product structure on $\varphi$. Moreover, both $m \circ (m \times \id)$ and
$m \circ (\id \times m)$ are $3$-fold product structures on $\varphi$.

\begin{Def}
Let $\varphi$ be a Hilbert $\Gamma$-bundle, and $m$ be a product structure on $\varphi$.
Then $m$ is called a \defterm{good} product structure when the product structures $m$ and $T(m \circ T_\Gamma)$
are equivalent and the $3$-fold product structures $m \circ (m \times \id)$ and $m \circ (\id \times m)$ are equivalent.
\end{Def}

\begin{Rem}A product structure is good when it is
commutative ``up to natural transformations'', and
associative ``up to natural transformations''.
\end{Rem}

\subsubsection{Products on $KF_G^\varphi$}\label{7.2}

Here $\varphi$ will denote a Hilbert $\Gamma$-bundle, $n$ an integer greater or equal to $2$,
and $m$ an $n$-fold product structure on $\varphi$. For any Lie group $G$, we set
$$KF_G^\varphi:=\Omega B\vEc_G^\varphi.$$
For every $n$-tuple of finite sets $(S_i)_{1\leq i\leq n}$,
$m$ induces a functor \\
$m((S_i)_{1 \leq i\leq n})\Mod: \left(\underset{i=1}{\overset{n}{\prod}}\varphi(S_i)\right) \Mod
\longrightarrow \varphi\left(\underset{i=1}{\overset{n}{\prod}}S_i\right) \Mod$
defined by
$$m((S_i)_{1\leq i\leq n})\Mod:
\begin{cases}
(x_i)_{1\leq i\leq n} & \longmapsto m_{(S_i)_{1\leq i \leq n}}((x_i)_{1\leq i\leq n}) \\
(x_i \overset{f_{s_i}}{\rightarrow} x'_i)_{(s_i)_{1\leq i \leq n}\in \underset{i=1}{\overset{n}{\prod}}
S_i}
& \longmapsto \left(m_{(S_i)_{1\leq i \leq n}} \circ \left(\underset{i=1}{\overset{n}{\otimes}}f_{s_i}\right)_{(s_i)_{1\leq i \leq n}
\in \underset{i=1}{\overset{n}{\prod}}S_i} \circ m_{(S_i)_{1\leq i \leq n}}^{-1} \right)
\end{cases}.$$
Let $(G_i)_{1 \leq i \leq n}$ be an $n$-tuple of Lie groups. We set $G:=\underset{i=1}{\overset{n}{\prod}}G_i$.
The functors $m((S_i)_{1\leq i \leq n}) \Mod$ induce a continuous map
$$B^n m: \underset{i=1}{\overset{n}{\prod}}B\vEc_{G_i}^\varphi
\longrightarrow B^n\vEc_G^\varphi,$$
where $B^n\vEc_G^\varphi$ denotes the $n$-fold realization of the $\Gamma$-space
$\underline{\vEc}_G^\varphi$ (cf.\ \cite{Segal-cat} §1).
This yields a pointed $G$-map $\underset{i=1}{\overset{n}{\bigwedge}}B\vEc_{G_i}^\varphi \longrightarrow
B^n\vEc_G^\varphi/
B^n_0\vEc_G^\varphi$
and, furthermore, a morphism $\underset{i=1}{\overset{n}{\bigwedge}}B\vEc_{G_i}^\varphi
\longrightarrow B^n\vEc_G^\varphi$ in the category $CG_G^{h\bullet}$.
Our first result is the following key lemma:
\begin{lemme}\label{isomorphismstability}
Let $m$ and $m'$ be two product structures of order $n$ on a Hilbert $\Gamma$-bundle $\varphi$.
If $m$ and $m'$ are equivalent, then they induce the same morphism
$\underset{i=1}{\overset{n}{\bigwedge}}B\vEc_{G_i}^\varphi
\longrightarrow B^n\vEc_G^\varphi$
in the category $CG_G^{h\bullet}$.
\end{lemme}

\begin{proof}
It suffices to prove the result when $m$ maps to $m'$.
However, any natural transformation $\eta$ from $m$ to $m'$
induces a pointed equivariant homotopy from
$B^n m$ to $B^n m'$, by a result that is essentially similar to Lemma \ref{homotopyidentity}. \end{proof}
The previous morphism yields a morphism
$$m^*: \Omega ^n\left(\underset{i=1}{\overset{n}{\bigwedge}}B\vEc_{G_i}^\varphi\right)
\longrightarrow \Omega ^n B^n \vEc_G^\varphi$$ in the category $CG_G^{h\bullet}$.
We also have a pointed $G$-map:
$$\underset{i=1}{\overset{n}{\prod}}\Omega B\vEc_{G_i}^\varphi \longrightarrow
\Omega ^n\left(\underset{i=1}{\overset{n}{\bigwedge}}B\vEc_{G_i}^\varphi\right)$$
which yields a morphism in the category $CG^{h \bullet}_G$:
$$\underset{i=1}{\overset{n}{\bigwedge}}KF_{G_i}^\varphi \longrightarrow
\Omega ^n\left(\underset{i=1}{\overset{n}{\bigwedge}}B\vEc_{G_i}^\varphi\right).$$
Composing this last morphism with the above one yields a morphism in $CG_G^{h\bullet}$:
$$\underset{i=1}{\overset{n}{\bigwedge}}KF_{G_i}^\varphi \longrightarrow
\Omega ^n B^n\vEc_G^\varphi.$$
Composing it with the inverse of $-(i_{n-1}^0\circ \dots \circ i_{1}^0):
\Omega B \vEc_G^\varphi \longrightarrow \Omega^n B^n \vEc_G^\varphi$ in the category
$CG_G^{h\bullet}[W_G^{-1}]$ (cf.\ the definition of the $i_j^k$'s in Section \ref{C} of the appendix),
finally yields a morphism
$$\underset{i=1}{\overset{n}{\bigwedge}}KF_{G_i}^\varphi \longrightarrow KF_G^\varphi$$
in the category $CG^{h \bullet}_G[W_G^{-1}]$.
By Lemma \ref{isomorphismstability}, this morphism only depends on
the isomorphism class of the chosen $n$-fold product structure $m$.
We may rewrite the previous morphism as the composite morphism:
$$\underset{i=1}{\overset{n}{\bigwedge}}KF_{G_i}^\varphi \longrightarrow
\Omega ^n\left(\underset{i=1}{\overset{n}{\bigwedge}}B\vEc_{G_i}^\varphi\right) \overset{m^*}{\longrightarrow}
\Omega ^n B^n\vEc_G^\varphi \Left9{-(i_{n-1}^0\circ \dots \circ i_{1}^0)} KF_G^\varphi$$

\begin{Rems}
\begin{enumerate}[(i)]
\item By Corollary \ref{independanceonk} in the appendix, we would have
obtained the same morphism by using the inverse of $-(i_{n-1}^{k_{n-1}}\circ \dots \circ i_{1}^{k_1})$
in $CG^{h \bullet}_G[W_G^{-1}]$ for any $(k_1,\dots,k_{n-1})\in \underset{j=0}{\overset{n-2}{\prod}}[j]$
, instead of using the inverse of $-(i_{n-1}^0\circ \dots \circ i_{1}^0)$.
\item This construction is different from that in \cite{Bob2},
where instead of the inverse of $-(i_{n-1}^0\circ \dots \circ i_{1}^0)$,
the ``opposite map" is used (in the sense of inversion of
the order of looping). The construction of \cite{Bob2} however leads to false claims on the ring structure of $KF_G^*(-)$.
\item The morphism $\underset{i=1}{\overset{n}{\wedge}}KF_{G_i}^\varphi
\longrightarrow KF_G^\varphi$ we have just constructed is clearly natural with respect to the Lie groups $G_1,\dots,G_n$.
\end{enumerate}
\end{Rems}

Let now $m$ be a product structure on $\varphi$, and $G$ and $H$ be two Lie groups.
Composing with the above morphism, yields a map
$$[X,KF_G^\varphi]_G^\bullet \times [Y,KF_H^\varphi]_H^\bullet \longrightarrow
[X \wedge Y,KF_G^\varphi \wedge KF_H^\varphi]_{G \times H}^\bullet \longrightarrow
[X \wedge Y,KF_{G \times H}^\varphi]_{G\times H}^\bullet$$
for every pointed proper $G$-space $X$ and every pointed proper $H$-space $Y$.

\subsubsection{Properties of the product given by a good product structure}\label{7.3}

Here $\varphi$ will denote a Hilbert $\Gamma$-bundle, and $m$ a good product structure on $\varphi$.
Let $G$ and $H$ be two Lie groups. The following results are easily shown to follow from the definition of
a good product structure (see \cite{Bob2} for details):

\begin{prop}\label{functorprod}
The product map $[X,KF_G^\varphi]_G^{\bullet} \times [Y,KF_H^\varphi]_H^\bullet
\longrightarrow [X \wedge Y,KF_{G \times H}^\varphi]_{G \times H}^\bullet$
induced by $m$ is bilinear, functorial with respect to $X$ and $Y$ on the one hand, and with respect to $G$ and $H$
on the other hand.
\end{prop}

\begin{prop}\label{commutativity}
The product induced by $m$ is commutative, i.e. for any pointed proper
$G$-CW-complex $X$ and any pointed proper $H$-CW-complex $Y$,
the following square is commutative
$$\begin{CD}
[X,KF_G^\varphi]_G^\bullet \times [Y,KF_H^\varphi]_H^\bullet
@>>> [X \wedge Y,KF_{G\times H}^\varphi]_{G \times H}^\bullet \\
@V{i}VV @V{j^*}VV \\
[Y,KF_H^\varphi]_H^\bullet \times [X,KF_G^\varphi]_G^\bullet
@>>> [Y \wedge X,KF_{H \times G}^\varphi]_{H \times G}^\bullet
\end{CD}$$
where $i$ is the transposition of factors, and $j$ is induced
by the transposition map
$Y \wedge X \overset{\cong}{\longrightarrow} X \wedge Y$ and by the map
$KF_{G\times H}^\varphi \rightarrow KF_{H\times G}^\varphi$,
which is itself induced by the functor
$\mathcal{E}(H \times G) \rightarrow \mathcal{E}(G\times H)$
associated to the transposition map
$H\times G \overset{\cong}{\longrightarrow} G \times H$.
\end{prop}

\begin{prop}\label{associativity}
The product associated to $m$ is associative, i.e. for any
triple $(G_1,G_2,G_3)$ of Lie groups, for any triple $(X_1,X_2,X_3)$ such that
$X_i$ is a pointed proper $G_i$-CW-complex for every $i \in \{1,2,3\}$, the square
$$\begin{CD}
[X_1,KF_{G_1}^\varphi]_{G_1}^\bullet
\times [X_2,KF_{G_2}^\varphi]_{G_2}^\bullet \times [X_3,KF_{G_3}^\varphi]_{G_3}^\bullet
@>{a}>>
[X_1 \wedge X_2, KF_{G_1 \times G_2}^\varphi]_{G_1 \times G_2}^\bullet
\times [X_3,KF_{G_3}^\varphi]_{G_3}^\bullet \\
@V{b}VV @V{c}VV \\
[X_1,KF_{G_1}^\varphi]_{G_1}^\bullet
\times [X_2 \wedge X_3,KF_{G_2 \times G_3}^\varphi]_{G_2 \times G_3}^\bullet
@>{d}>> [X_1 \wedge X_2 \wedge X_3,
KF_{G_1 \times G_2 \times G_3}^\varphi]_{G_1\times G_2 \times G_3}^\bullet
\end{CD}$$
induced by the good product structure $m$ is commutative.
\end{prop}

\subsubsection{Typical product structures on  $\text{Fib}^{F^{(\infty)}}$}\label{7.4}

Let $n$ be an integer greater than or equal to $2$, and $\alpha: \mathbb{N}^n \overset{\cong}{\rightarrow} \mathbb{N}$
be a bijection. Identifying the canonical basis of $F^{(\infty)}$ with $\mathbb{N}$,
we use $\alpha$ to obtain an isomorphism $\alpha_F: \underset{i=1}{\overset{n}{\otimes}} F^{(\infty)}
\overset{\cong}{\rightarrow} F^{(\infty)}$ which is bicontinuous with respect to the limit topology for the inclusion of
finite-dimensional subspaces. It follows that we obtain a bicontinuous isomorphism
$$\alpha^F: \begin{cases}
\underset{i=1}{\overset{n}{\bigotimes}}\left[\underset{k\in \mathbb{N}}{\oplus}(F^{(\infty)})^{\{k\}}\right] & \longrightarrow
\underset{k\in \mathbb{N}}{\oplus}(F^{(\infty)})^{\{k\}} \\
x_1\otimes \dots \otimes x_n \in (F^{(\infty)})^{\{i_1\}}\otimes \dots \otimes
(F^{(\infty)})^{\{i_n\}} &  \longmapsto \alpha_F(x_1\otimes \dots\otimes x_n)\in
(F^{(\infty)})^{\{\alpha(i_1,\dots,i_n)\}},
\end{cases}$$
where, for every $k \in \mathbb{N}$, $(F^{(\infty)})^{\{k\}}$ denotes
the space of maps $\{k\} \rightarrow F^{(\infty)}$ (seen as a subspace of the space of maps
from $\mathbb{N}$ to $F^{(\infty)}$).

For every $n$-tuple $(A_1,\dots,A_n)$ of finite subsets of $\mathbb{N}$, the map
$\alpha$ thus induces a cartesian square
$$\begin{CD}
E_{A_1}(F^{(\infty)}) \otimes \dots \otimes  E_{A_n}(F^{(\infty)}) @>>> E_{\alpha(A_1 \times \dots \times A_n)}(F^{(\infty)}) \\
@VVV @VVV \\
G_{A_1}(F^{(\infty)}) \times \dots \times G_{A_n}(F^{(\infty)}) @>>> G_{\alpha(A_1 \times \dots \times A_n)}(F^{(\infty)})
\end{CD}$$
in which the top horizontal morphism
is defined by
$$\begin{cases}
E_{A_1}(F^{(\infty)}) \otimes \dots \otimes E_{A_n}(F^{(\infty)}) & \longrightarrow
E_{\alpha(A_1 \times \dots \times A_n)} (F^{(\infty)}) \\
(x_1,V_{x_1}) \otimes \dots\otimes (x_n,V_{x_n}) & \longmapsto (\alpha^F(x_1 \otimes\dots \otimes x_n),
\alpha^F(V_{x_1} \otimes \dots \otimes V_{x_n})).
\end{cases}$$
Those cartesian squares define strong morphisms of Hilbert bundles, and thus induce a natural transformation:
$$m_\alpha: \pi^{(n)} \circ \left(\text{Fib}^{F^{(\infty)}}\right)^n \longrightarrow \text{Fib}^{F^{(\infty)}} \circ \pi^{(n)}_\Gamma,$$
hence $m_\alpha$ is an $n$-fold product structure on $\text{Fib}^{F^{(\infty)}}$.

\begin{prop}\label{bijprodinvariance}
The equivalence class of the $n$-fold product structure $m_\alpha$ on  $\text{Fib}^{F^{(\infty)}}$
does not depend on the choice of the bijection $\alpha: \mathbb{N}^n \overset{\cong}{\rightarrow} \mathbb{N}$.
\end{prop}

\begin{proof}
Let $\alpha$ and $\beta$ be two isomorphisms $\mathbb{N}^n
\overset{\cong}{\rightarrow} \mathbb{N}$.
Then $\gamma:=\beta \circ \alpha^{-1}$ is a permutation of
$\mathbb{N}$, and it induces an automorphism $\gamma':F^{(\infty)}
\overset{\cong}{\longrightarrow} F^{(\infty)}$.
We thus obtain a natural transformation
$(\gamma',\gamma)^*:
\text{Fib}^{F^{(\infty)}} \rightarrow \text{Fib}^{F^{(\infty)}}$
(cf.\ Section \ref{4.7}) such that $m_\beta=(\gamma',\gamma)^*_{\pi_\Gamma^{(n)}} \circ m_\alpha$.
This proves that $m_\beta$ is equivalent to $m_\alpha$.
\end{proof}

\begin{prop}
For any bijection $\alpha: \mathbb{N} \times
\mathbb{N} \overset{\cong} {\longrightarrow} \mathbb{N}$, the product
structure $m_\alpha$ on $\text{Fib}^{F^{(\infty)}}$ is good.
\end{prop}

\begin{proof}
Let $T_\mathbb{N}$ be the transposition of factors of $\mathbb{N}^2$.
Then $T(m_\alpha \circ T_\Gamma)=m_{\alpha \circ T_\mathbb{N}}$.
We deduce from Proposition \ref{bijprodinvariance} that $m_\alpha$ and $T(m_\alpha \circ T_\Gamma)$ are equivalent.
Moreover
$m_\alpha \circ (m_\alpha \times \id)=m_{\alpha \circ (\alpha \times \id)}$
and $m_\alpha \circ (\id \times m_\alpha)=m_{\alpha \circ (\id \times \alpha)}$
We deduce from Proposition \ref{bijprodinvariance} that $m_\alpha \circ (m_\alpha \times \id)$ and
$m_\alpha \circ (\id \times m_\alpha)$ are equivalent. \end{proof}

\subsubsection{The product maps in equivariant K-theory}\label{7.5}

Let $G$ and $H$ be two Lie groups. If we choose a bijection $\alpha: \mathbb{N}\times \mathbb{N} \overset{\cong}{\rightarrow}
\mathbb{N}$, we obtain a good product structure $m_\alpha$ on $\text{Fib}^{F^{(\infty)}}$ which yields a morphism
$KF_G^{[\infty]} \wedge KF_H^{[\infty]} \rightarrow KF_{G\times H}^{[\infty]}$
in the category $CG_{G \times H}^{h\bullet}[W_{G\times H}^{-1}]$.
Proposition \ref{bijprodinvariance} and Lemma \ref{isomorphismstability} show this is independent from the choice of $\alpha$.

Given a proper $G$-CW-complex $X$ and a proper $H$-CW-complex
$Y$, we then obtain a map
\begin{multline*}
m_\alpha^*: KF_G(X) \times KF_H(Y) \longrightarrow
[X^+,KF_G^{[\infty]}]_G^\bullet \times [Y^+,KF_G^{[\infty]}]_H^\bullet
\longrightarrow
[X^+ \wedge Y^+,KF_{G \times H}^{[\infty]}]_{G \times H}^\bullet \\
= KF_{G \times H}(X \times Y).
\end{multline*}
For any proper $G$-CW-pair $(X,A)$ and any proper $H$-CW-pair
$(Y,B)$, a similar procedure yields, for every pair $(m,n) \in \mathbb{N}^2$, a map
$$KF_G^{-m}(X,A) \times KF_G^{-n}(Y,B) \longrightarrow
KF_{G\times H}(\Sigma^m(X/A) \wedge
\Sigma^n(Y/B),*).$$
These are our product maps in equivariant K-theory. Since $m_\alpha$ is a good product structure,
Proposition \ref{commutativity} and \ref{associativity} show that they are ``commutative'' and ``associative''.
Just like in \cite{Bob2}, one then proves:

\begin{prop}\label{productcompatible}
Let $G$ and $H$ be two Lie groups, $X$ be a proper $G$-CW-complex and $Y$ be a proper
$H$-CW-complex.
Then the square
$$\begin{CD}
\mathbb{K}F_{G}(X) \times \mathbb{K}F_{H}(Y)
@>{\gamma _{X} \times \gamma _{Y}}>> KF_{G}(X) \times KF_{H}(Y) \\
@V{\otimes}VV @V{m_\alpha^*}VV \\
\mathbb{K}F_{G\times H}(X\times Y) @>{\gamma _{X\times Y}}>> KF_{G\times H}(X\times Y)
\end{CD}$$
is commutative.
\end{prop}

\subsection{Bott periodicity and the extension to positive degrees}\label{7.6}

Assume $F=\C$.
By the same method as in \cite{Bob2}, the Bott homomorphism
$$\beta _{X,A}^{-n}:KF_G^{-n}(X,A) \longrightarrow KF_G^{-(n+2)}(X,A)$$
is defined for every proper $G$-CW pair and shown to be an isomorphism.
In our case, this uses Proposition \ref{isomorphism} and the fact that Bott periodicity holds
for equivariant $K$-theory with compact groups (cf. \cite{SegalKtheory}).

For any proper $G$-CW-complex $X$, a structure of graded ring may also be defined on
$$KF_G^*(X):=\underset{n=0}{\overset{+ \infty}{\oplus}}KF_G^{-n}(X)$$
using the good product structure derived from $m_\alpha$.
Following the arguments from \cite{Bob2}, this is easily shown to satisfy the following properties:

\begin{theo}\label{ringstructure}
Let $G$ be a Lie group, and $X$ be a proper $G$-CW-complex. Then:
\begin{enumerate}[(i)]
\item The Bott homomorphisms induce an isomorphism
$$\beta: KF_G^*(X) \rightarrow KF_G^*(X)$$ which is linear for the
product of the graded ring $KF_G^*(X)$ (on both sides).
\item The map $\gamma _X: \mathbb{K}_G(X) \rightarrow KF_G(X)$ is a ring homomorphism.
\end{enumerate}
\end{theo}

\begin{cor}
The cohomology theory $KF_G^*(-)$ may
be extended to positive degrees on the category of $G$-CW-pairs, so that we obtain a good equivariant cohomology theory
with a graded ring structure and Bott periodicity.
\end{cor}

\begin{Rems}
\begin{itemize}
\item The construction of the ring structure and of Bott homomorphisms may also carried out in the case $F=\mathbb{R}$,
with the usual modifications.
\item All the previous constructions may also be carried out with $iKF_G^{[\infty]}$ and $sKF_G^{[\infty]}$,
and this easily yields the same product structure and the same Bott homomorphisms when $\pi_0(G)$ is countable.
\end{itemize}
\end{Rems}

\section{Connection with other equivariant K-theories}\label{comparesec}

\subsection{A connection with Segal's equivariant K-theory}\label{8}

Let $G$ be a Lie group. We have constructed a natural transformation
$$\gamma: \mathbb{K}F_G^*(-) \longrightarrow KF_G^*(-)$$
on the category of $G$-CW-pairs.
Recall that, for every non-negative integer $n$, every compact subgroup
$H$ of $G$, and every finite complex $Y$ on which $G$ acts trivially,
$$\gamma^{-n}_{(G/H) \times Y}: \mathbb{K}F_G^{-n}((G/H) \times Y) \overset{\cong}{\longrightarrow} KF_G^{-n}((G/H) \times Y)$$
is an isomorphism. Also that $\gamma_X: \mathbb{K}F_G(X) \longrightarrow KF_G(X)$ is a ring homomorphism
for every $G$-CW-complex $X$.

We now assume that $\mathbb{K}F_G^*(-)$ is a good equivariant cohomology theory on the category of finite proper
$G$-CW-pairs: we know this is true in the case $G$ is a compact Lie group or a discrete group (cf.\ \cite{Bob1}).
In this case, it is easy to check that the natural transformation $\gamma$ is compatible with the boundary maps
in the respective long exact sequences of a pair for the equivariant cohomology theories $\mathbb{K}F_G^*(-)$ and $KF_G^*(-)$.
It follows that the natural transformation $\gamma$ is compatible both with the long exact sequences of a pair and the long exact
sequences of Mayer-Vietoris. As a consequence, we have:

\begin{prop}
If $\mathbb{K}F_G^*(-)$ is a good equivariant cohomology theory on the category of finite proper
$G$-CW-pairs, then, for any finite proper $G$-CW-pair $(X,A)$ and every $n \in \N$, $\gamma _{X,A}^{-n}$ is an isomorphism.
\end{prop}

\begin{proof}
By compatibility of $\gamma$ with the long exact sequence of a pair, and by the five lemma,
it suffices to prove that $\gamma_X^{-n}$ is an isomorphism for every proper finite $G$-CW-complex $X$ and every non-negative integer $n$.
Let $X$ be such a proper finite $G$-CW-complex, and $n\in \mathbb{N}$. We prove that $\gamma_X^{-n}$ is an isomorphism by a
double induction process on the dimension of $X$ and the number of cells in a given dimension: this is done by considering
the long exact sequences of Mayer-Vietoris associated to push-out squares of the form
$$\begin{CD}
G/H \times S^{m-1} @>>> G/H \times D^m \\
@V{\psi}VV @VVV \\
Y @>>> (G/H \times D^m) \cup_\psi Y,
\end{CD}$$ and then by using the compatibility of $\gamma$ with those sequences, together with the result of Corollary \ref{isomorphisms}
and the five lemma. \end{proof}

With the same arguments as in Section 2 of \cite{Bob2}, one also obtains:

\begin{prop}
Assume $\mathbb{K}F_G^*(-)$ is a good equivariant cohomology theory on the category of finite proper
$G$-CW-pairs. Then, for any finite proper $G$-CW-complex $X$, the homomorphism
$$\gamma_X^*: \underset{n=0}{\overset{+\infty}{\oplus}}\mathbb{K}F^{-n}_G(X) \longrightarrow \underset{n=0}{\overset{+\infty}{\oplus}} KF^{-n}_G(X)$$
is a ring isomorphism which commutes with the Bott homomorphisms.
\end{prop}

\subsection{A connection with the K-theory of Lück and Oliver}\label{9}

\subsubsection{A natural transformation $KF_G^*(-) \rightarrow K'F_G^*(-)$}

Here, $G$ will be a discrete group.
For any $G$-CW-pair $(X,A)$ and any integer $n$, we denote by $K'F_G^n(X,A)$ the Lück-Oliver equivariant K-theory of the pair $(X,A)$
in degree $n$ as defined in \cite{Bob2}.

We define the category $\Gamma \text{-Fib}_F^{*}$ as follows:
\begin{itemize}
\item An object of $\Gamma \text{-Fib}_F^*$ consists of a finite
set $S$, of a locally-countable CW-complex $X$,
and, for every $s \in S$, of a finite-dimensional vector bundle
(with underlying field $F$): $p_s:E_s \rightarrow X$.
\item A morphism $f:(S,X,(p_s)_{s\in S}) \longrightarrow
(T,Y,(q_t)_{t\in T})$ consists of a morphism
$\gamma: T \rightarrow S$
in the category $\Gamma$, of a continuous map
$\bar{f}:X \rightarrow Y$, and for every $t \in T$, of a strong morphism
of vector bundles:
$$\begin{CD}
\underset{s \in \gamma(t)}{\oplus}E_t @>{f_t}>> E'_t \\
@V{\underset{s \in \gamma(t)}{\oplus}p_s}VV @VV{q_t}V \\
X @>>{\bar{f}}> Y.
\end{CD}$$
\end{itemize}
The composition of morphisms is defined as in $\Gamma \text{-Fib}_F$.

A functor $\mathcal{O}_\Gamma^{F*}:\Gamma \text{-Fib}_F^* \rightarrow \Gamma$ is
defined as follows:
$$\begin{cases}
(S,X,(p_s)_{s\in S}) & \longmapsto S \\
f:(S,X,(p_s)_{s\in S}) \rightarrow
(T,Y,(q_t)_{t\in T}) & \longmapsto (\gamma: T \rightarrow S).
\end{cases}$$

As in the case of Hilbert $\Gamma$-bundles, we may define the sum in
$\Gamma \text{-Fib}_F^*$, and the functors
$\Mod: \Gamma \text{-Fib}_F^* \rightarrow
\text{kCat}$ and $\Bdl: \Gamma \text{-Fib}_F^* \rightarrow
\text{kCat}$. Statement (i) of Proposition \ref{modproduit} is then true in the case of $\Gamma \text{-Fib}_F^*$, and so
are Proposition \ref{naturaltransformations} and Corollary \ref{naturalcor}.

A \defterm{$\Gamma$-vector bundle} is a contravariant functor $\varphi: \Gamma \rightarrow \Gamma \text{-Fib}_F^*$ satisfying:
\begin{enumerate}[(i)]
\item $\mathcal{O}_\Gamma^{F*}(\varphi)=\id_{\Gamma}$;
\item $\varphi(\mathbf{0})=(\mathbf{0},*,\emptyset)$;
\item $\forall n \in \mathbb{N}^*, \exists f_n:n.\varphi(\mathbf{1}) \rightarrow \varphi(\mathbf{n})$
such that $\mathcal{O}_\Gamma^{F*}(f_n)=\id_{\mathbf{n}}$.
\end{enumerate}

Moreover, if $F=\mathbb{R}$ or $\mathbb{C}$, we have a natural functor
$\Gamma \text{-Fib}_F \rightarrow \Gamma \text{-Fib}_F^*$
obtained by forgetting the Hilbert structure on the fibers.
Any Hilbert $\Gamma$-bundle may thus be seen as
a $\Gamma$-vector bundle. The two $\Mod$ constructions are identical.

We define a functor $\psi: \Gamma \longrightarrow \Gamma \text{-Fib}_F^*$ in the following way:
for any finite set $S$, we let $X_S$ denote the subset of
$\left(\underset{k=0}{\overset{+\infty}{\cup}}\sub_k(F^{(\infty)})\right)^S$ consisting of those families
indexed over $S$ whose factors are in direct sum, and we equip $X_S$ with the discrete topology.
For any $s \in S$, we let $p^S_s:E^S_s \rightarrow X_S$ denote the canonical vector bundle over $X_S$
whose fiber over any $(x_t)_{t \in S}$ is $x_s$.
We finally set $$\psi(S):=(S,X_S,(p^S_s)_{s\in S}).$$
For every morphism $f:S \rightarrow T$ in $\Gamma$, the direct sum of subspaces gives rise to a morphism $\psi(f): \psi(T) \rightarrow \psi(S)$
in $\Gamma \text{-Fib}_F^*$. It is then easy to check that $\psi$ is a $\Gamma$-vector bundle.

We observe that the space $KF_G$ defined
by Lück and Oliver in \cite{Bob2} is simply $KF_G^\psi$ as defined in the beginning of Section \ref{7.2}.
We now wish to relate $KF_G^\psi$ to $KF_G^{[\infty]}$.

For any finite set $S$, the underlying space of $\text{Fib}^F(S)$
is the set of $S$-tuples of subspaces of $\underset{n=0}{\overset{+\infty}{\cup}}F^{[n]}$
that are in direct sum and have a basis consisting of vectors of the canonical basis of
$\underset{n=0}{\overset{+\infty}{\cup}}F^{[n]}$.
Indeed, for any finite subset $A$ of $\mathbb{N}$, $F^A$ is the only $(\# A)$-dimensional subspace of itself.
The isomorphism
$$\begin{cases}
\underset{n=0}{\overset{+\infty}{\cup}}F^{[n]} & \overset{\cong}{\longrightarrow} F^{(\infty)} \\
e_n & \longmapsto e_{n+1}
\end{cases}$$
yields a canonical injection of $\text{Fib}^F(S)$ into $X_S$.
These injections yield a natural transformation $\epsilon: \text{Fib}^F \longrightarrow \psi$ between
$\Gamma$-vector bundles. Then $\epsilon$ gives rise to a morphism of equivariant $\Gamma$-spaces:
$$\epsilon^*:\underline{\vEc}_G^{F,1} \longrightarrow \underline{\vEc}_G^\psi.$$
We claim that $\epsilon^*_\mathbf{1}$ is an equivariant homotopy equivalence.
To see this, we choose, for every non-negative integer $n$ and every $n$-dimensional subspace
$E$ of $F^{(\infty)}$, an isomorphism $E \overset{\cong}{\longrightarrow} F^{[n-1]}$. Those choices yield
a strong morphism of vector bundles
$$\begin{CD}
E^\mathbf{1}_1 @>>> \underset{f \in \Gamma(\mathbf{1})}{\coprod}E_{f(1)}(F)\\
@V{p_1^\mathbf{1}}VV @VVV \\
X_\mathbf{1} @>>> \underset{f \in \Gamma(\mathbf{1})}{\coprod}B_{f(1)}(F).
\end{CD}$$

Using results that are essentially similar to the ones of Proposition \ref{naturaltransformations} and Lemma \ref{homotopyidentity} (adapted to the case of $\Gamma$-vector bundles),
we conclude that the natural transformation $\epsilon: \text{Fib}^F \rightarrow \psi$
induces an equivariant homotopy equivalence
$$\epsilon^*_\mathbf{1}: \vEc_G^{F,1} \overset{\cong}{\longrightarrow} \vEc_G^\psi.$$

Since $\underline{\vEc}_G^{F,1}$ and $\underline{\vEc}_G^\psi$ are $\Gamma-G$-spaces, we deduce
that $\epsilon^*_S: \underline{\vEc}_G^{F,1}(S) \overset{\cong}{\longrightarrow}
\underline{\vEc}_G^\psi(S)$ is an equivariant homotopy equivalence for every finite set $S$.

We finally deduce from the previous discussion that $\epsilon^*$ induces a $G$-weak equivalence
$$\epsilon^*: KF_G^{[1]} \longrightarrow KF_G.$$
Since $G$ is discrete, Proposition \ref{mdependance} applied to $m=1$ shows
that the canonical map $KF_G^{[1]} \longrightarrow KF_G^{[\infty]}$ is a $G$-weak equivalence.
By composing its inverse with $\epsilon^*: KF_G^{[1]} \longrightarrow KF_G$, we recover an isomorphism
$KF_G^{[\infty]} \longrightarrow KF_G$
in the category $CG_G^{h\bullet}[W_G^{-1}]$.

For every pointed proper $G$-CW-complex $X$,
this morphism induces a group isomorphism
$$[X,KF_G^{[\infty]}]_G^\bullet \overset{\cong}{\longrightarrow} [X,KF_G]_G^\bullet.$$

For every $n \in \mathbb{N}$ and every proper $G$-CW-pair $(X,A)$, we thus
obtain a group isomorphism
$$\epsilon_{X,A}^{-n}: KF_G^{-n}(X,A) \overset{\cong}{\longrightarrow} K'F_G^{-n}(X,A).$$

This defines a bijective natural transformation from our equivariant K-theory to the one of Lück and Oliver.

\subsubsection{The compatibility of $\epsilon$ with products}
We now check that the natural transformation $\epsilon$ is compatible with products and Bott homomorphisms.
It actually suffices to prove that it is compatible with exterior products.

We let $\alpha: \mathbb{N} \times \mathbb{N} \overset{\cong}{\longrightarrow} \mathbb{N}$
be a bijection such that $\alpha(0,0)=0$. We also let $\beta: \begin{cases}
F & \longrightarrow F^{(\infty)} \\
x & \longmapsto (x,0,0,\dots)
\end{cases}$ denote the canonical injection.

The embedding $\beta$ helps us see $\text{Fib}^F$ as a sub-Hilbert $\Gamma$-bundle of $\text{Fib}^{F^{(\infty)}}$.
Since $\alpha(0,0)=0$, we have $\alpha_F(e_1,e_1)=e_1$, and it follows that the good product
structure $m_\alpha$ on $\text{Fib}^{F^{(\infty)}}$
induces a good product structure $m_\alpha'$ on $\text{Fib}^F$.
Moreover, for any pair of Lie groups $(G,H)$, the square

$$\begin{CD}
KF_G^{[1]} \wedge KF_H^{[1]} @>>> KF_G^{[\infty]} \wedge KF_H^{[\infty]} \\
@VV{(m_\alpha')^*}V @VV{m_\alpha^*}V \\
KF_{G \times H}^{[1]} @>>>  KF_{G \times H}^{[\infty]}
\end{CD}$$
is commutative in $CG_{G\times H}^{h\bullet}[W_{G\times H}^{-1}]$. \\
If we assume that $G$ and $H$ are both discrete, then the square
$$\begin{CD}
KF_G^{[1]} \wedge KF_H^{[1]} @>>> KF_G \wedge KF_H \\
@VV{(m'_\alpha)^*}V @VV{m_{(\alpha_F)}^*}V \\
KF_{G\times H}^{[1]} @>>> KF_{G \times H}
\end{CD}$$
is also commutative in $CG_G^{h\bullet}[W_{G\times H}^{-1}]$, and this follows from our definition of products
and the one in Section 2 of \cite{Bob2}.

We deduce from the previous two commutative squares that
$\epsilon^*: KF_G^*(-) \longrightarrow K'F_G^*(-)$ is compatible with exterior products. In particular,
it is compatible with Bott homomorphisms and it is therefore possible to extend $\epsilon$ to positive degrees.
We finally obtain
a ring isomorphism
$$\epsilon_X^*: KF_G^*(X) \overset{\cong}{\longrightarrow} K'F_G^*(X)$$
for any discrete group $G$ and any proper $G$-CW-complex $X$.

We conclude that our equivariant K-theory coincides
with the one of Lück and Oliver in the case of discrete groups.

\begin{Rem}
The reason our construction is much more complicated than the one of Lück and Oliver
is that the latter fails in the general case of a Lie group.
For example, $\vEc_{S^1}^{F,1}$ is not a classifying
space for $\mathbb{V}\text{ect}_{S^1}^\mathbb{R}(-)$ on the category
of $S^1$-CW-complexes, because the fixed point set
$(\widetilde{\vEc}_G^{F,1})^{\{(-1,-1),(1,1)\}}$ for the subgroup
$\{(-1,-1),(1,1)\} \subset S^1 \times \GL_1(\mathbb{R})$ is empty.
The additional complexity we introduced was there so that
Proposition \ref{classifyingspace} would hold.
\end{Rem}

\appendix

\section{On some fiber bundles}\label{Bundleproof}

Our aim here is to give a complete proof of Theorem \ref{universal}, which we now restate:

\begin{theo}
Let $X$ be a locally-countable CW-complex,
$\varphi: E \rightarrow X$ be an $n$-dimensional vector bundle over $X$, and
$G$ be a Lie group.
Then: \begin{enumerate}[(i)]
\item $\widetilde{\vEc}_G^\varphi \longrightarrow \vEc_G^\varphi$ is a $(G,\GL_n(F))$-principal bundle;
\item $E\vEc_G^\varphi \longrightarrow \vEc_G^\varphi$ is an $n$-dimensional $G$-vector bundle;
\item The canonical map
$$\widetilde{\vEc}_G^\varphi \times _{\GL_n(F)} F^n \longrightarrow
E\vEc_G^\varphi$$
is an isomorphism of $G$-vector bundles over
$\vEc_G^\varphi$.
\end{enumerate}
\end{theo}

There is a free right-action of $\GL_n(F)$ on $\varphi \Frame$,
and this induces a free action of $\GL_n(F)$  on
$|\Func(\mathcal{E}G,\varphi \Frame)|$.
On the other hand, $G$ acts on $\mathcal{E}G$ on the right (by left-multiplication of the inverse),
and, by precomposition, this induces a left-action of $G$ on $|\Func(\mathcal{E}G,\varphi
\Frame)|$. The respective actions of $G$ and $\GL_n(F)$ on $|\Func(\mathcal{E}G,\varphi
\Frame)|$ are clearly compatible. It follows that, in order to prove that $|\Func(\mathcal{E}G,\varphi
\Frame)| \longrightarrow |\Func(\mathcal{E}G,\varphi
\Mod)|$ is a $(G,\GL_n(F))$-principal bundle, it suffices
to produce local trivializations of it.

The proof is split into four parts. In the first one, we construct ``structural" maps that are
linked to the topological categories of Section \ref{3}, and we prove that they are continuous.
This will be used to prove that the local trivialization maps that will be constructed later are actually continuous.
In Step 2, we construct an open covering of the space $|\Func(\mathcal{E}G, \varphi \Mod)|$,
and use it in Step 3 to produce local sections, and then local trivializations.
All the results from Theorem \ref{universal} are finally deduced in Step 4. Before moving on to Step 1, we will need
to recall some notations on simplicial spaces and prove a basic lemma on fibre bundles.

\subsection{Basic notation}

For $n \in \mathbb{N}$, we let
$\Delta^n:=\bigl\{(t_0,\dots,t_n)\in \mathbb{R}^{n+1}: \; t_0+\dots+t_n=1\bigr\}$
denote the standard $n$-simplex, and $\partial \Delta^n$ its boundary.

We let $\Delta$ denote the simplicial category. For $N \in \mathbb{N}$ and $i\in [N+1]$,
$\delta_i^N: [N] \hookrightarrow [N+1]$ will denote the face morphism whose image does not contain $i$, whilst,
for $N \in \mathbb{N}$ and $i\in [N]$, $\sigma_i^{N+1}: [N+1] \twoheadrightarrow [N]$ the degeneracy morphism such that
$\sigma_i^N(i)=\sigma_i^N(i+1)$.
If $A=(A_n)_{n\in \mathbb{N}}$ is a simplicial space, $N \in \mathbb{N}$ and
$i\in [N+1]$, the face map associated to $\delta_i^N$ is denoted by $d_i^N:A_{N+1} \rightarrow A_N$
and, for $N \in \mathbb{N}^*$ and $i\in [N-1]$, the degeneracy map associated to $\sigma_i^N$ is denoted by $s_i^N:A_{N-1} \rightarrow A_N$.
When $A=(A_n)_{n\in \mathbb{N}}$ is a simplicial space, $B=(B_n)_{n\in \mathbb{N}}$ is a cosimplicial space, and
$f\in \Hom(\mathbf{\Delta})$, we let $f_*$ (respectively $f^*$) denote the map associated to $f$ in $A$
(resp. in $B$).
When $A$ is a simplicial space, recall that its thin geometric realization is the quotient space
$$|A|:=\biggl(\underset{n \in \mathbb{N}}{\coprod} (A_n \times \Delta^n)\biggr)/\sim$$
where $\sim$ is defined by: $(x,f^*(y)) \sim (f_*(x),y)$ for all $f: [n] \rightarrow [m]$, all
$x \in A_m$ and all $y \in \Delta^n$.

\subsection{A fundamental lemma}

\begin{lemme}\label{separationfonction}
Let $X$ be a locally-countable CW-complex, and $\tilde{\varphi}:\tilde{E}
\rightarrow X$ be a $\GL_n(F)$-principal bundle (with corresponding vector bundle $\varphi$).
Then:
\begin{enumerate}[(i)]
\item There is a continuous map $\delta: \Mor(\varphi \Mod) = (\tilde{E} \times \tilde{E})/\GL_n(F) \rightarrow \mathbb{R}_+$, called
a \defterm{separation map} such that
\begin{equation}\label{separationproperty}
\forall (x,y) \in \tilde{E}\times \tilde{E}, \delta (\left[x,y\right])=0 \Leftrightarrow x=y.
\end{equation}
\item For every $x \in X$, there is a continuous section
$s_x: X \longrightarrow E^{\oplus n}$
such that $s_x(x) \in \tilde{E}$.
\end{enumerate}
\end{lemme}

\begin{Rem}
Property \eqref{separationproperty} means that $\delta$ vanishes precisely on the set
of identity morphisms of the category $\varphi \Mod$.
\end{Rem}

\begin{proof}
Notice first that given a relative CW-complex $(Y,B)$, any continuous map $\alpha:B \rightarrow \R_+$
may be extended to a continuous map
$\tilde{\alpha}:Y \rightarrow \R_+$ so that
$\tilde{\alpha}^{-1}\{0\}=\alpha^{-1}\{0\}$: this is easily done by working cell by cell.
We will use this remark to construct a continuous map $\alpha : X \times X \rightarrow \R_+$ such that
$$\alpha^{-1}\{0\}=\Delta_X:=\bigl\{(x,x) \mid x \in X\bigr\}.$$
Notice that $X \times X$ is a CW-complex for the cartesian product topology (since $X$ is locally-countable).
Let $n \in \N$ and assume that $\alpha$ has been defined on the $(n-1)$-th
skeleton of $X \times X$. Let $\Delta_X^{(n)}$ denote the intersection of $\Delta_X$ with the $n$-th skeleton $\Sk_n(X\times X)$ of $X \times X$.
It is however easily shown that $\bigl(\Sk_n(X\times X),\Delta_X ^{(n)} \cup \Sk_{n-1}(X\times X)\bigr)$
is a relative CW-complex, using the somewhat obvious fact that
$(\Delta^n \times \Delta^n,\Delta_{\Delta^n} \cup \partial(\Delta^n \times \Delta^n))$ is a finite relative CW-complex
for every $n \in \N$.

We then extend $\alpha$, first on $\Delta_X ^{(n)} \cup \Sk_{n-1}(X\times X)$ by mapping any
$x\in \Delta_X^{(n)}$ to $0$, and then on $\Sk_n(X \times X)$ using the initial remark.
Then $\forall x \in \Sk_n(X \times X), \; \alpha(x)=0 \Leftrightarrow x \in \Delta_X$.
This induction process then defines $\alpha$ on the whole $X \times X$
with the claimed property.

We now choose a continuous map $\alpha : X \times X \rightarrow \R_+$ such that $\alpha^{-1}\{0\}=\Delta_X$.
We also choose a system $(U_i,\varphi_i)_{i\in I}$ of local trivializations for the
principal bundle $\tilde{\varphi}$
(where $\varphi_i: \tilde{E}_{|U_i} \overset{\cong}{\rightarrow} U_i \times \GL_n(F)$, by convention)
together with a partition of unity $(\alpha_i)_{i \in I}$ for the covering $(U_i)_{i \in I}$ of $X$. For any $i \in I$, we will let
$\pi_2: U_i \times \GL_n(F) \rightarrow \GL_n(F)$ denote the projection onto the second factor.
Finally, we choose a norm $N$ on the linear space $\text{M}_n(F)$ of square matrices.

For $(x,y) \in (\GL_n(F))^2$, we set
$$\delta_1(x,y):=N(xy^{-1} -I_n)$$
and notice that $\delta_1$ is continuous and invariant by the right-action of
$\GL_n(F)$ on itself. Notice also that
$$\forall (x,y)\in \GL_n(F)^2, \; \delta_1(x,y) \geq 0 \quad \text{and} \quad \delta _1(x,y)=0 \Leftrightarrow x=y.$$
For $(x,y) \in \tilde{E} \times \tilde{E}$, we then set:
$$\delta(x,y):=\alpha\bigl(\tilde{\varphi}(x),\tilde{\varphi}(y)\bigr)
+ \underset{i \in I}{\sum}\alpha_i\left(\tilde{\varphi}(x)\right)\alpha_i\left(\tilde{\varphi}(y)\right)
\inf\Bigl[1,\delta_1\bigl((\pi_2 \circ \varphi_i)(x),
(\pi_2 \circ \varphi_i)(y)\bigr)\Bigr].$$
From there, it is straightforward to prove that $\delta : \tilde{E} \times \tilde{E} \rightarrow \R_+$
is continuous, that $\delta^{-1}\{0\}=\Delta_{\tilde{E}}$ and that
$\forall (M,x,y) \in \GL_n(F) \times \tilde{E}^2, \; \delta(x.M,y.M)=\delta(x,y)$, which
shows that $\delta$ induces a continuous map
$(\tilde{E} \times \tilde{E})/\GL_n(F) \rightarrow \mathbb{R}_+$ which satisfies condition (\ref{separationproperty}).
This proves statement (i).

We now choose a system of local trivializations
$(U_i,\psi _i)_{i \in I}$ for $\varphi:E \rightarrow X$,
together with a partition of unity $(\alpha _i)_{i \in I}$ for covering $(U_i)_{i \in I}$ of $X$.
We choose an $i_0\in I$ such that $\alpha _{i_0}(x) >0$. If
$\theta_{i_0}: U_{i_0} \rightarrow E^{\oplus n}$ is the map which assigns
$(\varphi_{i_0}(y,e_1),\dots,\varphi_{i_0}(y,e_n))$ to $y$ (where $(e_1,\dots,e_n)$ denotes
the canonical basis of $F^n$), then $s_x:=\alpha_{i_0}.\theta_{i_0}$ has the required property for statement (ii).
\end{proof}

\subsection{Proof of Theorem \ref{universal}}

We fix a system $(U_i,\varphi _i)_{i \in I}$ of local trivializations of $\varphi:E \rightarrow X$.

\subsubsection{Step 1: Structural maps}

To make things easier, we rename
$$\mathcal{E}:=\Func (\mathcal{E}G,\varphi \Mod), \quad \mathcal{F}:=\Func (\mathcal{E}G,\varphi \Frame) \quad \text{and} \quad
\mathcal{G}:=\Func (\mathcal{E}G,\varphi \Bdl).$$
For every $m \in \mathbb{N}$ and $g \in G$, there is a canonical map
$$\alpha_{g,m}:\begin{cases}
\mathcal{N}(\mathcal{E})_m & \longrightarrow X \\
F_0 \rightarrow \dots \rightarrow F_m & \longmapsto F_0(g),
\end{cases}$$ and we let
$\mathcal{N}(\mathcal{E})_m \underset{g,0,X}{\times} \tilde{E}$ denote
the limit of the diagram $\mathcal{N}(\mathcal{E})_m \overset{\alpha_{g,m}}{\longrightarrow}
X \overset{\tilde{\varphi}}{\longleftarrow} \tilde{E}$.

We then let
$$\psi _{g,m}: \mathcal{N(E)}_m \underset{g,0,X}{\times}
\tilde{E} \longrightarrow \mathcal{N(F)}_m$$
denote the map which assigns $F'_0 \rightarrow \dots \rightarrow F'_m$ to the compatible pair $(F_0 \overset{\eta_0}{\rightarrow} \dots
\overset{\eta_{m-1}}{\rightarrow} F_m,\mathbf{B})$, where
$F'_k(h)=\left[\eta_{k-1} \circ\dots \circ \eta_0\right] \circ F_0((g,h))[\mathbf{B}]$
for $k \in \{0,\dots,m\}$ and $h \in G$.

One should think of an element
$F_0 \overset{\eta_0}{\rightarrow} \dots \overset{\eta_{m-1}}{\rightarrow} F_m$ of $\mathcal{N}(\mathcal{E})$
as an array of fibers of $\varphi$, with an isomorphism between any pair of fibers in the array, such that
the whole diagram is commutative. On the other hand, an element
$F_0 \overset{\eta_0}{\rightarrow} \dots \overset{\eta_{m-1}}{\rightarrow} F_m$ of $\mathcal{N}(\mathcal{F})$
should simply be thought of as an array of basis of fibers of $\varphi$.
Let $F_0 \overset{\eta_0}{\rightarrow} \dots \overset{\eta_{m-1}}{\rightarrow} F_m$ in $\mathcal{N}(\mathcal{E})$, a basis
$\mathbf{B}$ of the fiber $E_{g,0}$, and say we wish to obtain an element of $\mathcal{N}(\mathcal{F})$.
If we plug the basis $\mathbf{B}$ at the position $(g,0)$ in the diagram associated to
$F_0 \overset{\eta_0}{\rightarrow} \dots \overset{\eta_{m-1}}{\rightarrow} F_m$,
then we can use the isomorphisms of the diagram to recover a basis in every position $(h,i)$ of the diagram. The remaining
diagram of basis corresponds to an element of $\mathcal{N}(\mathcal{F})$, and this is precisely the image of
the pair $(F_0 \overset{\eta_0}{\rightarrow} \dots \overset{\eta_{m-1}}{\rightarrow} F_m,\mathbf{B})$ by $\psi _{g,m}$,
judging from the previous definition.

\begin{lemme}\label{psicomp}
Let $f: [N] \rightarrow [N']$ be a morphism in $\Delta$, let $g \in G$,
$x=F_0 \overset{\eta_0}{\rightarrow} \dots \overset{\eta_{N'-1}}{\rightarrow} F_{N'}
\in \mathcal{N}(\mathcal{E})_{N'}$, and $y \in \tilde{E}$
such that $\tilde{\varphi}(y)=\alpha_{g,N'}(x)$. Then
$$f_*\bigl(\psi_{g,N'}(x,y)\bigr)=\psi_{g,N}\bigl(f_*(x),(\eta_{f(0)-1}(g) \circ \dots \circ \eta_0(g))(y)\bigr)$$
with the convention that $\eta_{f(0)-1} \circ \dots \circ \eta_0=\id_{E_{F_0(g)}}$ whenever
$f(0)=0$.
\end{lemme}

\begin{proof}
By definition, $\psi_{g,N}(f_*(x),y)$
is the only $N$-simplex of $\mathcal{N}(\mathcal{F})$ which is sent to $f_*(x)$ by $\mathcal{N}(\mathcal{F})
\rightarrow \mathcal{N}(\mathcal{E})$ and such that the basis in position
$(g,0)$ is $(\eta_{f(0)-1}(g) \circ \dots \circ \eta_{0}(g))(y))$.
On the other hand, $\psi_{g,N'}(x,y)$ is the only $N'$-simplex of $\mathcal{N}(\mathcal{F})$
which is sent to $x$ by $\mathcal{N}(\mathcal{F})
\rightarrow \mathcal{N}(\mathcal{E})$ and such that the basis in position
$(g,0)$ is $y$. Hence $f_*(\psi_{g,N'}(x,y))$
is the only $N$-simplex of $\mathcal{N}(\mathcal{F})$ which is sent to $f_*(x)$
by $\mathcal{N}(\mathcal{F}) \rightarrow \mathcal{N}(\mathcal{E})$ and such that the basis in position $(g,0)$ is
$(\eta_{f(0)-1}(g) \circ \dots \circ \eta_{0}(g))(y)$.
\end{proof}

We have another map
$$\chi_m: \mathcal{N(F)}_m \underset{\mathcal{N(E)}_m}{\times}
\mathcal{N(F)}_m \longrightarrow \GL_n(F),$$
which sends a (compatible) pair
$\bigl((F_0\rightarrow \dots \rightarrow F_m),(F'_0\rightarrow \dots \rightarrow F'_m)\bigr)$
to the unique $M \in \GL_n(F)$ such that $F_k(g)=F'_k(g).M$ for every $k \in \{0,\dots,m\}$
and $g \in G$. We also define
$$v_m: \mathcal{N(F)}_m \times F^n \longrightarrow
\mathcal{N(G)}_{m,}$$ which maps a
pair $(F_0 \rightarrow \dots \rightarrow
F_m,\,x)$ to $F'_0 \overset{\eta_0}{\rightarrow} \dots \overset{\eta_{m-1}}{\rightarrow}F'_m$, where:
\begin{enumerate}[(i)]
\item $F'_k(g)=\theta(F_k(g),x)$ for all $(k,g)\in \{0,\dots,m\} \times G$;
\item $F'_k(g,g')=(F'_k(g),F'_k(g'),F_k(g) \mapsto F_k(g'))$, for all $(k,g,g')\in \{0,\dots,m\}
\times G^2$;
\item $\eta_k(g)=(F'_k(g),F'_{k+1}(g),F_k(g) \mapsto F_{k+1}(g))$
for all $(k,g) \in \{0,\dots,m-1\} \times G$.
\end{enumerate}
Finally, we define
$$\epsilon_m: \mathcal{N(G)}_m \underset{\mathcal{N(E)}_m}{\times} \mathcal{N(F)}_m
\longrightarrow F^n$$
as the map which sends every (compatible) pair
$((F_0 \rightarrow \dots \rightarrow F_n),(F'_0 \rightarrow \dots \rightarrow F'_m))$
to the unique $x\in F^n$ such that $F'_0(1_G)=\theta(F_0(1_G),x)$, i.e.\ the unique $x\in F^n$ such that
$\nu_m((F'_0 \rightarrow\dots \rightarrow F'_m),x)=F_0 \rightarrow \dots \rightarrow F_{m.}$

\begin{prop}\label{structuralcontinuity}
For every $m \in \mathbb{N}$ and $g \in G$, the maps $\psi_{g,m}$,
$\chi_m$, $v_m$ and $\epsilon_m$ are continuous.
\end{prop}

\begin{proof}
All products and spaces of maps will be formed in the category of k-spaces. \\
$\bullet$ \textbf{Continuity of $\chi_m$.} The map $\chi_m$ may be decomposed as
$$\mathcal{N(F)}_m \underset{\mathcal{N(E)}_m}{\times}
\mathcal{N(F)}_m \Right9{(1_G,0) \times (1_G,0)}
\tilde{E} \underset{X}{\times} \tilde{E} \longrightarrow \GL_n(F),$$
where the map $(1_G,0)$ assigns $F_0(1_G)$ to $F_0 \rightarrow \dots \rightarrow F_m$.
It follows that the first map is continuous, and the second map also is since its restrictions over
the open sets $U_i$ are continuous. \\
$\bullet$ \textbf{Continuity of $v_m$.} It obviously suffices to prove that $v_0$ is continuous.
However, the canonical map $\tilde{E} \times F^n \rightarrow E$ is continuous, hence the continuity of
$$\begin{cases}
(\tilde{E} \times \tilde{E})^{G \times G} \times F^n
& \longrightarrow (\Mor(\varphi \Bdl))^{G \times G} \\
(f_1,f_2,x) & \longmapsto \bigl[(g_1,g_2) \mapsto  (f_1(g_1,g_2).x, f_2(g_1,g_2).x
,[f_1(g_1,g_2) \mapsto f_2(g_1,g_2)])
\bigr].
\end{cases}$$
Therefore, $v_0$, being the restriction of this map to $\mathcal{N(F)}_0 \times F^n$, is continuous. \\
$\bullet$ \textbf{Continuity of $\epsilon_m$.} The map $\epsilon_m$ may be seen as the composite:
$$\mathcal{N(G)}_m \underset{\mathcal{N(E)}_m}{\times}
\mathcal{N}(\mathcal{F})_m \Right9{(1_G,0) \times (1_G,0)}
E \underset{X}{\times} \tilde{E} \longrightarrow F^n,$$
where the second map sends any compatible pair $(x,\mathbf{B})$ to the unique
$y\in F^n$ such that $\theta(\mathbf{B},y)=x$. The continuity of the first map is proven in the same way as
the continuity of $\chi_m$. The second map is continuous because its restrictions over
the open subsets $U_i$ of $X$ clearly are. Therefore $\epsilon_m$ is continuous. \\
$\bullet$ \textbf{Continuity of $\psi _{g,m}$.} Let $\Mor(\varphi \Mod) \underset{X}{\times} \tilde{E}$
denote the subspace of the product space
$\Mor(\varphi \Mod) \times \tilde{E}$ consisting of the pairs $(f,y)$ such that $\In_{\varphi \Mod}(f)=\tilde{\varphi}(y)$,
and let
$(\tilde{E} \times \tilde{E})\underset{X}{\times} \tilde{E}$ denote the subspace
of $\tilde{E}\times \tilde{E} \times \tilde{E}$ consisting of the triples
$(x,y,z)$ such that $\tilde{\varphi}(x)=\tilde{\varphi}(z)$.

We start by proving that the following map is continuous:
$$\begin{cases}
\Mor(\varphi \Mod) \underset{X}{\times} \tilde{E} & \longrightarrow \tilde{E}\\
(f,\mathbf{B}) & \longmapsto f(\mathbf{B}).
\end{cases}$$
Notice first that the projection
$(\tilde{E} \times \tilde{E}) \underset{X}{\times} \tilde{E}
\longrightarrow \Mor(\varphi \Mod) \underset{X}{\times} \tilde{E}$
is an open identification map. It then suffices to prove that the map
$$\begin{cases}
(\tilde{E} \times \tilde{E})\underset{X}{\times} \tilde{E} & \longrightarrow
\tilde{E} \\
(\mathbf{B}.M,\mathbf{B'},\mathbf{B}), M \in \GL_n(F) & \longmapsto
\mathbf{B'}.M^{-1}
\end{cases}$$
is continuous, which is a straightforward task using local trivializations of $\tilde{\varphi}$.

We may now prove that $\psi_{g,0}$ is continuous. We denote by $(\Mor(\varphi \Mod))^G\underset{G,X}{\times}\tilde{E}$
(resp.\ by $(\Mor(\varphi \Mod))^{G \times G} \underset{G,X}{\times} \tilde{E}$)
the subset of $(\Mor(\varphi \Mod))^G \times \tilde{E}$ consisting of the pairs
$(f,y)$ such that $\forall g' \in G, \; \In(f(g'))=\tilde{\varphi}(y)$
(resp.\ $\forall g' \in G,\; \In(f(g,g'))=\tilde{\varphi}(y)$).
We consider then the composite map
$$(\Mor(\varphi \Mod))^{G \times G}
\underset{G,X}{\times} \tilde{E} \longrightarrow
(\Mor(\varphi \Mod ))^G \underset{G,X}{\times}
\tilde{E} \longrightarrow \tilde{E}^G \longrightarrow
\tilde{E}^G \times \tilde{E}^G \longrightarrow (\tilde{E} \times
\tilde{E})^{G \times G},$$
where the first map is obtained by precomposition with the injection
$\begin{cases}
G & \rightarrow G \times G \\
g_1 & \mapsto (g,g_1),
\end{cases}$ and is thus continuous.
The last map is continuous due to classic results on k-spaces, whilst
the third one is simply the diagonal map.
It remains to prove that the map
$$\bigl(\Mor(\varphi \Mod)\bigr)^G \underset{G,X}{\times}
\tilde{E} \longrightarrow \tilde{E}^G$$
is continuous. However, it may be regarded upon as a composite
$$(\Mor(\varphi \Mod))^G \underset{G,X}{\times}
\tilde{E} \longrightarrow (\Mor(\varphi \Mod))^G \underset{G,X}{\times}
\tilde{E}^G \longrightarrow \bigl(\Mor(\varphi \Mod)
\underset{X}{\times} \tilde{E}\bigr)^G \longrightarrow
\tilde{E}^G.$$
The first map is continuous, since it comes from
$\begin{cases}
\tilde{E} & \rightarrow \tilde{E}^G \\
y & \mapsto \left[g \mapsto y\right],
\end{cases}$
the second one is also continuous by the usual properties
of k-spaces, and we have already proven that the third one is continuous.
Hence $\psi_{g,0}$ is continuous.

Finally, in order to prove the continuity of $\psi_{g,m}$ for an arbitrary $m\in \N$,
it suffices to prove the continuity of the map obtained by composing $\psi_{g,m}$ with
the inclusion of $\mathcal{N}(\mathcal{F})_m$ into
$(\tilde{E})^{\{0,\dots,m\} \times G}$. This is done is the same way as in the case $m=0$.
\end{proof}

The maps $\chi_m$, for $m\in \mathbb{N}$, induce a morphism of simplicial spaces
$\mathcal{N}(\mathcal{F}) \underset{\mathcal{N}(\mathcal{E})}{\times} \mathcal{N}(\mathcal{F})
\longrightarrow \GL_n(F)$, hence a continuous map:
$$\chi: |\mathcal{F}| \underset{|\mathcal{E}|}{\times} |\mathcal{F}| \longrightarrow \GL_n(F).$$
In the same manner, the maps $\epsilon_m$, for $m\in \N$, induce a morphism of simplicial spaces
$\mathcal{N}(\mathcal{G}) \underset{\mathcal{N}(\mathcal{E})}{\times} \mathcal{N}(\mathcal{F})
\longrightarrow F^n$, hence a continuous map:
$$\epsilon: |\mathcal{G}| \underset{|\mathcal{E}|}{\times} |\mathcal{F}| \longrightarrow F^n.$$
Finally, the maps $v_m$, for $m\in \mathbb{N}$, induce a morphism of simplicial spaces
$\mathcal{N}(\mathcal{F}) \times F^n \longrightarrow \mathcal{N}(\mathcal{G})$, hence a continuous map
$$v: |\mathcal{F}| \times F^n \longrightarrow |\mathcal{G}|.$$

\subsubsection{Step 2: An open cover of $|\Func(\mathcal{E}G, \varphi \Mod)|$}

By Lemma \ref{separationfonction}, we may choose
$\delta: \Mor(\varphi \Mod) \rightarrow \mathbb{R}_+$ satisfying \eqref{separationproperty}
and, for every $x_1 \in X$, a section $s_{x_1}$ which satisfies the requirements of statement (ii).

Let $m\in \N$ and $(g_0,g_1,\dots,g_{m-1})\in G^m$.  Then, for every $0\leq i_0<\dots<i_m\leq N$, we define the map
$$\alpha_{N,x_1,(g_0,\dots,g_{m-1})}^{(i_0,\dots,i_m)}:
\mathcal{N(E)}_N \times \Delta ^N \longrightarrow E^{\oplus n}$$
as the one which maps any pair $\bigl((F_0 \overset{\eta _0}{\rightarrow} \dots \overset{\eta _{N-1}}{\rightarrow}
F_N),(t_0,\dots,t_N)\bigr)$
to
$$t_{i_{0}}
\left[\underset{j=1}{\overset{m}{\prod}}t_{i_{j}}\delta
\left[\eta _{i_{(j-1)}}(g_{j-1})\circ \dots \circ
\eta _{i_{j}-1}(g_{j-1})\right]\right](\eta _{i_{0}-1}(g_{0})
\circ\dots \circ \eta _{0}(g_{0}))^{-1}(s_{x_1}(F_{i_{0}}(g_{0}))).$$
We then set:
$$
\alpha_{N,x_1,(g_0,\dots,g_{m-1})}:
\begin{cases}
\mathcal{N(E)}_N \times \Delta ^N & \longrightarrow E^{\oplus n} \\
(x,t) & \longmapsto \underset{0\leq i_0<\dots<i_m\leq N}{\sum}
\alpha_{N,x_1,(g_0,\dots,g_{m-1})}^{(i_0,\dots,i_m)}(x,t).
\end{cases}$$
Notice that $\alpha_{N,x_1,(g_0,\dots,g_{m-1})}$ is continuous
and, for any $x=F_0 \rightarrow \dots \rightarrow F_N$ and any
$t\in \Delta^N$, one has $\alpha_{N,x_1,(g_0,\dots,g_{m-1})}(x,t) \in (E_{F_0(g_0)})^n$.

Set now
$$U^{(N)}_{x_1,(g_0,\dots,g_{m-1})}:=(\alpha_{N,x_1,(g_0,\dots,g_{m-1})})^{-1}(\tilde{E})$$
and notice that this is an open subset of $\mathcal{N(E)}_N \times \Delta^N$ since $\tilde{E}$ is an open subset of $E^{\oplus n}$.
Denote by $\pi _N: \mathcal{N(E)}_N \times \Delta ^N \rightarrow
|\mathcal{E}|$ the canonical projection.

\begin{prop} ${}$
\begin{enumerate}[(i)]
\item $U_{x_1,(g_0,\dots,g_{m-1})}:=\underset{N=0}{\overset{\infty}{\bigcup}}
\pi _N(U^{(N)}_{x_1,(g_0,\dots,g_{m-1})})$ is an open subset of
$|\mathcal{E}|$.
\item $\forall N \in \N$,
$\pi _N^{-1}(U_{x_1,(g_0,\dots,g_{m-1})})=U^{(N)}_{x_1,(g_0,\dots,g_{m-1}).}$
\item $\forall (N,N') \in \N^2, \forall f \in \Hom_\Delta([N],[N']),
\forall (x=F_0\overset{\eta_0}{\rightarrow}\dots\overset{\eta_{N'-1}}{\rightarrow}F_{N'} ,t)
\in \mathcal{N(E)}_{N'} \times \Delta ^N$,
$$\alpha_{N,x_1,(g_0,\dots,g_{m-1})}(f_*(x),t)=
(\eta_{f(0)-1}(g_0) \circ \dots \circ \eta_{0}(g_0))\left( \alpha_{N',x_1,(g_0,\dots,g_{m-1})} (x,f^*(t))\right)$$
with the convention that $\eta_{f(0)-1}(g_0) \circ \dots \circ \eta_{0}(g_0)=\id_{E_{F_0(g_0)}}$
when $f(0)=0$.
\item $(U_{x_1,(g_0,\dots,g_{m-1})})_{x_1\in X, m\in \mathbb{N},(g_0,\dots,g_{m-1})\in G^m}$
is a cover of $|\mathcal{E}|$.
\end{enumerate}
\end{prop}

\begin{proof}
We first prove (ii).
It suffices to prove the following facts:

\begin{multline}\label{degcomp1}
\forall N \in \mathbb{N}, \forall (x,t) \in \mathcal{N(E)}_N \times \Delta ^{N+1},
\forall i \in [N], \\
(s_i^{N+1}(x),t) \in U^{(N+1)}_{x_1,(g_0,\dots,g_{m-1})}
\Leftrightarrow
(x,\sigma_i^{N+1}(t)) \in U^{(N)}_{x_1,(g_0,\dots,g_{m-1})}
\end{multline}

and
\begin{multline}\label{facecomp1}
\forall N \in \mathbb{N}, \forall (x,t) \in \mathcal{N(E)}_{N+1} \times \Delta ^{N},
\forall i \in [N+1], \\
(x,\delta_i^N(t))\in U^{(N+1)}_{x_1,(g_0,\dots,g_{m-1})} \Leftrightarrow
(d_i^N(x),t)\in  U^{(N)}_{x_1,(g_0,\dots,g_{m-1}).}
\end{multline}
Let $0 \leq i_0<\dots<i_m \leq N+1$, let $i \in [N]$, $x\in \mathcal{N(E)}_N$ and $t \in \Delta ^{N+1}$.
If there exists an index $j$ such that $i_j=i$ and $i_{j+1}=i+1$, then
$\delta(\eta_{j})=\delta(\id_{E_{F_{i_j}(g_j)}})=0$, and so
$$\alpha_{N+1,x_1,(g_0,\dots,g_{m-1})}^{(i_0,\dots,i_m)}(s_i^{N+1}(x),t)=0.$$
If there exists an index $j$ such that $i_j=i$ and $i_{j+1}>i+1$, we set $i'_l=i_l$
when $l\neq j$, and $i'_j=i+1$.
Also, for $l \in [m]$, we set $i''_l=\sigma_i^{N+1}(i_l)$. We may then check that
$$\alpha_{N+1,x_1,(g_0,\dots,g_{m-1})}^{(i_0,\dots,i_m)}(s_i^{N+1}(x),t)+
\alpha_{N+1,x_1,(g_0,\dots,g_{m-1})}^{(i'_0,\dots,i'_m)}(s_i^{N+1}(x),t)=
\alpha_{N,x_1,(g_0,\dots,g_{m-1})}^{(i''_0,\dots,i''_m)}
(x,\sigma_i^{N+1}(t)).$$
If $\{i,i+1\} \cap \{i_j,0\leq j \leq m\}=\emptyset$, we still set $i''_l=\sigma_i^{N+1}(i_l)$ for any $l\in [m]$.
It follows that
$$\alpha_{N+1,x_1,(g_0,\dots,g_{m-1})}^{(i_0,\dots,i_m)}(s_i^{N+1}(x),t)=\alpha_{N,x_1,(g_0,\dots,g_{m-1})}^{(i''_0,\dots,i''_m)}
(x,\sigma_i^{N+1}(t)).$$
Summing the previous equalities yields
$$\alpha_{N+1,x_1,(g_0,\dots,g_{m-1})}(s_i^{N+1}(x),t)=\alpha_{N,x_1,(g_0,\dots,g_{m-1})}
(x,\sigma_i^{N+1}(t))$$
and \eqref{degcomp1} follows right away. \\
Let $0 \leq i_0<\dots<i_m \leq N+1$, $i \in [N+1]$, $x\in \mathcal{N(E)}_{N+1}$
and $t \in \Delta ^N$. We write $x=F_0 \overset{\eta_0}{\rightarrow} \dots \overset{\eta_N}{\rightarrow} F_N$.
If there exists an index $j$ such that $i_j=i$, then
$$\alpha_{N+1,x_1,(g_0,\dots,g_{m-1})}^{(i_0,\dots,i_m)}(x,\delta_i^N(t))=0.$$
Otherwise, we set $i'_j=\sigma_i^{N+1}$ for all $j \in [m]$. \\
When $i>0$, $$\alpha_{N+1,x_1,(g_0,\dots,g_{m-1})}^{(i_0,\dots,i_m)}(x,\delta_i^N(t))
=\alpha_{N,x_1,(g_0,\dots,g_{m-1})}^{(i'_0,\dots,i'_m)}(d_i^N(x),t).$$
Also $$\eta_0(g_0)\left(\alpha_{N+1,x_1,(g_0,\dots,g_{m-1})}^{(i_0,\dots,i_m)}(x,\delta_0^N(t))\right)
=\alpha_{N,x_1,(g_0,\dots,g_{m-1})}^{(i'_0,\dots,i'_m)}(d_0^N(x),t).$$
Summing these equalities when yields
$$\begin{cases}
i>0 \; \Rightarrow \; \alpha_{N,x_1,(g_0,\dots,g_{m-1})}(d_i^N(x),t)=\alpha_{N+1,x_1,(g_0,\dots,g_{m-1})}(x,\delta_i^N(t)) \\
\alpha_{N,x_1,(g_0,\dots,g_{m-1})}(d_0^N(x),t)=
\eta_0(g_0)\left(\alpha_{N+1,x_1,(g_0,\dots,g_{m-1})}(x,\delta_0^N(t)\right)).
\end{cases}$$
and \eqref{facecomp1} follows right away. This proves both statements (ii) and (iii), and
(i) then follows from (ii) and the fact that $U^{(N)}_{x_1,(g_0,\dots,g_{m-1})}$
is an open subset of $\mathcal{N}(\mathcal{E})_N \times \Delta^N$.

Let $y \in |\mathcal{E}|$. Then there exists an integer $m \in \mathbb{N}$,
a non-degenerate $m$-simplex $x \in \mathcal{N(E)}_m$ and a point $\alpha \in \Delta ^m
\setminus \partial \Delta ^m$ such that $y=\pi _m(x,\alpha)$.
Since $x=F_0 \overset{\eta _0}{\rightarrow} \dots \overset{\eta _{m-1}}{\rightarrow} F_m$ is non-degenerate,
there exists, for all $i \in [m-1]$, an element $g_i \in G$ such that $\delta (\eta _i (g_i)) \neq 0$.
We simply remark that $(x,\alpha) \in U_{F_0(g_0),(g_0,\dots,g_{m-1})}$, which proves statement (iv).
\end{proof}

The open subsets $U_{x_1,(g_0,\dots,g_{m-1})}$ are defined in such a way that,
for every $N\in \mathbb{N}$, and every $(x,t)\in \mathcal{N}(\mathcal{E})_N \times \Delta^N$ such that
$[x,t]\in U_{x_1,(g_0,\dots,g_{m-1})}$: if we write $x=F_0 \rightarrow \dots \rightarrow F_N$, then
the map $\alpha_{N,x_1,(g_0,\dots,g_{m-1})}(x,t)$ provides a continuous way (with respect to $(x,t)$)
to choose a basis in the fiber of $F_0(g_0)$. In the next step, this is used, in conjunction with the maps $\psi_{g,m}$,
to construct local sections of the bundle $|\mathcal{F}| \rightarrow |\mathcal{E}|$.

\subsubsection{Step 3: Local trivializations}

We let $\pi: |\mathcal{F}| \rightarrow |\mathcal{E}|$ denote the map induced by the functor
$\Func(\mathcal{E}G,\varphi \Frame) \longrightarrow \Func(\mathcal{E}G,\varphi \Mod)$ discussed earlier.
We will now write
$$\alpha_{N,x_1,(g_0,\dots,g_{m-1})}: U^{(N)}_{x_1,(g_0,\dots,g_{m-1})}
\longrightarrow \tilde{E}$$
when we actually mean the restriction of $\alpha_{N,x_1,(g_0,\dots,g_{m-1})}$
to $U^{(N)}_{x_1,(g_0,\dots,g_{m-1})}$.

We then consider, for every $N \in \mathbb{N}$, the composite map
$$\beta _{x_1,(g_0,\dots,g_{m-1})}^{(N)}: U^{(N)}_{x_1,(g_0,\dots,g_{m-1})}
\longrightarrow
\left[\mathcal{N(E)}_N \underset{g_0,0,X}{\times} \tilde{E}
\right] \times \Delta ^N \Right9{\psi _{g_0,N} \times \id_{\Delta ^N}}
\mathcal{N(F)}_N \times \Delta ^N,$$
where the first map assigns $(x, \alpha_{N,x_1,(g_0,\dots,g_{m-1})}(x,t),t)$
to $(x,t)$. Proposition \ref{structuralcontinuity} shows $\beta _{x_1,(g_0,\dots,g_{m-1})}^{(N)}$
is continuous.

\begin{prop}Let $x_1\in X$, $m\in \mathbb{N}$, and $(g_0,\dots,g_{m-1})\in G^m$.
\begin{enumerate}[(i)]
\item The maps $\beta _{x_1,(g_0,\dots,g_{m-1})}^{(N)}$, for $N \in \N$, induce
a continuous map  $$\beta_{x_1,(g_0,\dots,g_{m-1})}:
U_{x_1,(g_0,\dots,g_{m-1})} \longrightarrow |\mathcal{F}|.$$
\item The map $\beta_{x_1,(g_0,\dots,g_{m-1})}$ is a local section of $\pi$.
\item The map $$\varphi _{x_1,(g_0,\dots,g_{m-1})}:
\begin{cases}
U_{x_1,(g_0,\dots,g_{m-1})} \times \GL_n(F) & \longrightarrow
\pi ^{-1}(U_{x_1,(g_0,\dots,g_{m-1})}) \\
(x,M) & \longmapsto \beta_{x_1,(g_0,\dots,g_{m-1})}(x).M
\end{cases}$$ is a homeomorphism over $U_{x_1,(g_0,\dots,g_{m-1})}$.
\end{enumerate}
\end{prop}

\begin{proof}
\begin{enumerate}[(i)]
\item Let $(N,N')\in \N^2$, $f\in \Hom_{\Delta}([N],[N'])$,
and $(x,t)\in \mathcal{N}(\mathcal{E})_{N'} \times \Delta^N$ such that
$(x,f^*(t))\in U^{(N')}_{x_1,(g_0,\dots,g_{m-1})}$.
If $\beta _{x_1,(g_0,\dots,g_{m-1})}^{(N')}(x,f^*(t))=(y,s)$
and $\beta _{x_1,(g_0,\dots,g_{m-1})}^{(N)}(f_*(x),t)=(y',s')$, we have to
prove that $s=f^*(s')$ and $y'=f_*(y)$. The first identity is obvious from the definition of
$\beta _{x_1,(g_0,\dots,g_{m-1})}^{(N)}$. The second one may be restated as:
\begin{equation}\label{facecomp}
f_*\left(\psi_{g_0,N'}(x,\alpha_{N',x_1,(g_0,\dots,g_{m-1})}(x,f^*(t)))\right)
=\psi_{g_0,N}(f_*(x),\alpha_{N,x_1,(g_0,\dots,g_{m-1})}(f_*(x),t))
\end{equation}
However $\alpha_{N,x_1,(g_0,\dots,g_{m-1})}(f_*(x),t)=
(\eta_{f(0)-1}(g_0) \circ \dots \circ \eta_{0}(g_0))\left( \alpha_{N',x_1,(g_0,\dots,g_{m-1})} (x,f^*(t))\right)$.
hence \eqref{facecomp} follows from Lemma \ref{psicomp}. This proves (i) since the
maps $\beta _{x_1,(g_0,\dots,g_{m-1})}^{(N)}$ are continuous.
\item is an obvious consequence of the definition of $\beta _{x_1,(g_0,\dots,g_{m-1}).}$
\item The map $\varphi _{x_1,(g_0,\dots,g_{m-1})}$ is clearly continuous, and clearly bijective since
$\beta _{x_1,(g_0,\dots,g_{m-1})}$ is a local section of $\pi$.
It remains to prove the continuity of its inverse map:
the composite of it with the projection on the first factor
is the continuous map $\pi _{|\pi ^{-1}\left(U_{x_1,(g_0,\dots,g_{m-1})}\right).}$
The composite with the projection on the second factor is the map
$$\begin{cases}
\pi ^{-1}(U_{x_1,(g_0,\dots,g_{m-1})}) & \longrightarrow
\GL_n(F) \\
x & \longmapsto M \quad \text{such that} \quad x=\beta_{x_1,(g_0,\dots,g_{m-1})}(x).M
\end{cases}$$
Setting $V:=U_{x_1,(g_0,\dots,g_{m-1})}$,
it may be seen as the composite map
$$\pi ^{-1}(V)
\Right{{15}}{(\id,\beta_{x_1,(g_0,\dots,g_{m-1})} \circ
\pi)} \pi ^{-1}(V) \underset{V}{\times}\pi ^{-1}(V) \overset{\chi}{\longrightarrow} \GL_n(F),$$
and is thus continuous. Therefore $\varphi_{x_1,(g_0,\dots,g_{m-1})}$ is a homeomorphism.
\end{enumerate}
\end{proof}

We conclude that $\widetilde{\vEc}_G^\varphi \rightarrow
\vEc_G^\varphi$ is a $(G,\GL_n(F))$-principal bundle.

\subsubsection{Step 4: $E\vEc_G^\varphi \rightarrow \vEc_G^\varphi$ as a $G$-vector bundle}

The fibers of $\pi':E\vEc_G^\varphi \rightarrow \vEc_G^\varphi$
have natural structures of vector spaces which are inherited
from those of the fibers of $E \rightarrow X$. The group $G$ acts on the left on
$E\vEc_G^\varphi$ and the projection $E\vEc_G^\varphi \rightarrow
\vEc_G^\varphi$ is a $G$-map. Also, the action of $G$
on $E\vEc_G^\varphi$ restricts to linear isomorphisms on the fibers.

Take arbitrary $x_1 \in X$, $m \in \mathbb{N}$ and
$(g_0,\dots,g_{m-1})\in G^m$, and set $V:=U_{x_1,(g_0,\dots,g_{m-1})}$.
We may then consider the composite map
$$\varphi '_{x_1,(g_0,\dots,g_{m-1})}: V \times F^n
\Right{{16}}{\beta_{x_1,(g_0,\dots,g_{m-1})} \times \id_{F^n}}
\pi ^{-1}(V) \times F^n \overset{v}{\longrightarrow} (\pi ')^{-1}(V).$$

\begin{prop} The map
$\varphi '_{x_1,(g_0,\dots,g_{m-1})}$ is a homeomorphism
over $V$.
\end{prop}

\begin{proof}
The continuity, the bijectivity, and the fiberwise linearity of $\varphi '_{x_1,(g_0,\dots,g_{m-1})}$ are clear.
It remains to prove that the projection
of the inverse map on the second factor is continuous.
However this projection is no other than the (obviously continuous) composite map
$$(\pi ') ^{-1}(V) \Right{{15}}{(\id,\beta_{x_1,(g_0,\dots,g_{m-1})})}
(\pi ') ^{-1}(V)\underset{V}{\times} \pi ^{-1}(V)
\overset{\epsilon}{\longrightarrow} F^n.$$
\end{proof}

Obviously, the transitions between the trivialization maps
constructed for $\pi$
are identical to the transitions between the trivialization maps
constructed for $\pi '$. Therefore, $E\vEc_G^\varphi \rightarrow \vEc_G^\varphi$
is an $n$-dimensional vector bundle, hence a $G$-vector bundle.

It remains to check that
$\widetilde{\vEc}_G^\varphi \times _{\GL_n(F)} F^n
\longrightarrow E\vEc_G^\varphi$ is an isomorphism
of vector bundles over $\vEc_G^\varphi$
(since we already know that it is continuous and equivariant).
This comes from the surjectivity of the map
$\widetilde{\vEc}_G^\varphi \times F^n
\longrightarrow E\vEc_G^\varphi$ and from the fact that the two vector bundles
involved share the same dimension.
This completes the proof of Theorem \ref{universal}.

\begin{Rem}
We have claimed in Section \ref{3.5} that theorems similar to Theorem \ref{universal} hold for $G$-Hilbert bundles
and $G$-simi-Hilbert bundles. Their proofs are almost identical,
the only noticeable difference being in the construction of the maps $\beta^{(N)}_{x_1,(g_0,\dots,g_{m-1})}$: here, after
using the maps $\alpha_{N,x_1,(g_0,\dots,g_{m-1})}$, we need to use the orthonormalization process (resp.\ the simi-orthonormalization process)
to obtain orthonormal bases (resp.\ simi-orthonormal bases) of the fibers of $\varphi$. This works because the process is continuous
and compatible with the action of $U_n(F)$ (resp.\ of $\Sim_n(F)$).
\end{Rem}

\section{On the homotopy type of $(\vEc_G^{F,\infty})^H$}\label{Proofhomotype}

Here, we fix a Lie group $G$, a compact subgroup $H$ of $G$, and define
$\Rep_F(H)$ as the monoid of isomorphism classes of finite-dimensional linear representations of $H$ (with ground field $F$).
Our aim is to prove Theorem \ref{homotopytype} which we restate for convenience:

\begin{theo}
Let $G$ denote a Lie group, and $H$ be a compact subgroup of $G$. Then each one of the spaces $(\vEc_G^{F,\infty})^H$, $(i\vEc_G^{F,\infty})^H$
and $(s\vEc_G^{F,\infty})^H$ has the homotopy type of a CW-complex.
\end{theo}

We will only give the details in the case of $\vEc_G^{F,\infty}$. The strategy is as follows:
recall from Proposition \ref{naturaltransformations}
that $\vEc_G^{F,\infty}$ and $\vEc_G^{\gamma(F)}$ have the same equivariant homotopy type.
In Section \ref{5.1.1}, we will consider the restriction functor
$$\res_H: \Func(\mathcal{E}G,\gamma(F) \Mod)^H \longrightarrow \Func(\mathcal{B}H,\gamma(F) \Mod),$$
and, in Section \ref{5.1.3}, we construct a section of it
$$\Prol_H: \Func(\mathcal{B}H,\gamma(F) \Mod) \longrightarrow \Func(\mathcal{E}G,\gamma(F) \Mod)^H$$
with a continuous equivalence of functors between
$\Prol_H \circ \res_H$ and $\id_{\Func(\mathcal{E}G,\gamma(F) \Mod)^H}$.
We will deduce that $|\res_H|: (\vEc_G^{F,\infty})^H
\longrightarrow |\Func(\mathcal{B}H,\gamma(F) \Mod)|$ is a homotopy equivalence.
In Section \ref{5.2.2}, we will prove that $\mathcal{N}(\Func(\mathcal{B}H,\gamma(F) \Mod))_m$
has the homotopy type of a CW-complex for every $m \in \mathbb{N}$, and it will follow that its thick realization
$\|\Func(\mathcal{B}H,\gamma(F) \Mod)\|$ also does (cf.\ Appendix A of \cite{Segal-cat}). We will also show that
$\Func(\mathcal{B}H,\gamma(F) \Mod)$ is a good simplicial space in the sense of Segal (cf.\ again \cite{Segal-cat}),
deduce that $\|\Func(\mathcal{B}H,\gamma(F) \Mod)\|$ is homotopy equivalent to
$|\Func(\mathcal{B}H,\gamma(F) \Mod)|$, and conclude that $\left(\vEc_G^{F,\infty}\right)^H$
has the homotopy type of a CW-complex.

\subsection{A homotopy equivalence from
$(\vEc_G^{F,\infty})^H$ to $\Func(\mathcal{B}H,\gamma(F) \Mod)$}\label{5.1}

Proposition \ref{naturaltransformations} shows that
the map $\vEc_G^{\gamma^{(m)}(F)} \rightarrow \vEc_G^{F,m}$
is an equivariant homotopy equivalence for any $m \in \N^* \cup \{\infty\}$.
It will thus suffice to show that $(\vEc_G^{\gamma (F)})^H$ has the homotopy type of a CW-complex.

\subsubsection{The functor $\res_H: \Func(\mathcal{E}G,\gamma(F) \Mod)^H
\rightarrow \Func(\mathcal{B}H,\gamma (F) \Mod)$}\label{5.1.1}

Let $f: \mathcal{E}G \rightarrow \gamma (F) \Mod$ be a continuous functor
which is invariant for the action of $H$ on $\mathcal{E}G$ by right-multiplication.
Then, for all $h \in H$, $f(h)=f(1_G)$,
and $\forall (h,h') \in H^2, \;f(1_G,hh')=f(h',hh') \circ f(1_G,h')=f(1_G,h)\circ f(1_G,h')$.
The functor $f$ thus induces a covariant functor
$$f_{|H}:
\begin{cases}
\mathcal{B}H & \longrightarrow \gamma (F) \Mod \\
* & \longrightarrow f(1_G) \\
h & \longmapsto f(1_G,h).
\end{cases}$$

Given a natural transformation
$\alpha: f \rightarrow f'$
between two functors that are invariant by the action of $H$,
we may consider the restriction
$\alpha_{|H}:f_{|H} \rightarrow f'_{|H}$ defined by
$\alpha_{|H}(*): f(1_G) \overset{\alpha(1_G)}{\rightarrow} f'(1_G)$.
This yields a continuous functor
$$\res_H:\Func(\mathcal{E}G,\gamma(F) \Mod)^H \rightarrow \Func(\mathcal{B}H,\gamma(F) \Mod).$$
Taking geometric realizations, yields a continuous map
$$|\res_H|: \Bigl(\vEc_G^{\gamma(F)}\Bigr)^H \rightarrow |\Func(\mathcal{B}H,\gamma(F) \Mod)|.$$
We will prove that it is a homotopy equivalence.

\subsubsection{A decomposition of $\Func(\mathcal{B}H,\gamma(F) \Mod)$}\label{5.1.2}

From now on, we set $J=\Rep_F(H)$.
To every object $f$ of $\Func(\mathcal{B}H,\gamma(F) \Mod)$ is assigned the linear representation
$\begin{cases}
H & \rightarrow f(*) \\
h & \mapsto f(h).
\end{cases}$ For any $j \in J$, we let $\Func_j(\mathcal{B}H,\gamma(F) \Mod)$
denote the full subcategory of $\Func(\mathcal{B}H,\gamma(F) \Mod)$
whose objects are the functors whose corresponding linear representation of $H$ has $j$ as its isomorphism class.

\begin{prop}\label{decomprop}
We have the following decomposition of topological categories:
$$\Func(\mathcal{B}H,\gamma(F) \Mod) = \underset{j \in J}{\coprod}
\Func_j(\mathcal{B}H,\gamma(F) \Mod).$$
\end{prop}

\begin{proof}
If we consider a morphism $f \overset{\alpha}{\rightarrow} f'$ in
$\Func(\mathcal{B}H,\gamma(F) \Mod)$, then $\alpha$ induces
an isomorphism between the representations respectively associated to $f$ and $f'$.
Therefore $\Func(\mathcal{B}H,\gamma(F) \Mod)$ and $\underset{j \in J}{\coprod}
\Func_j(\mathcal{B}H,\gamma(F) \Mod)$ are isomorphic as categories.
We now need to prove that, for all $j \in J$,
the category $\Func_j(\mathcal{B}H,\gamma(F) \Mod)$ is open in
$\Func(\mathcal{B}H,\gamma(F) \Mod)$, and it suffices to show it for the spaces of objects. \\
We start by pointing out that
$$\Func(\mathcal{B}H,\gamma(F) \Mod)
=\Func\left(\mathcal{B}H,\underset{n \in \mathbb{N}}{\coprod}(\gamma_n(F) \Mod) \right)
=\underset{n \in \mathbb{N}}{\coprod}
\Func(\mathcal{B}H,\gamma_n(F) \Mod).$$
Let $f:\mathcal{B}H \rightarrow \gamma_n (F) \Mod$ be a continuous functor and define
$$\chi_f: \begin{cases}
H & \longrightarrow F \\
h & \longmapsto \tr(f(h))
\end{cases}$$
as the character associated to $f$. \\
We then check that the following map is continuous:
$$\chi^{(n)}: \begin{cases}
\Ob(\Func(\mathcal{B}H,\gamma_n(F) \Mod)) & \longrightarrow L^2(H) \\
f & \longmapsto \chi_f
\end{cases}$$
By construction, $\gamma_n(F) \Mod =\underset{\underset{m \in \mathbb{N}}{\longrightarrow}}
{\lim}(\gamma^m_n(F) \Mod )$.
It thus suffices to show that the restriction of
$\chi^{(n)}$ to $\Ob(\Func(\mathcal{B}H,\gamma^m_n(F) \Mod))$
is continuous for any $m \in \mathbb{N}$.
Let $f$ be an object of $\Func(\mathcal{B}H,\gamma^m_n(F) \Mod)$ and let
$U_f$ denote the set of objects $f'$ of $\Func(\mathcal{B}H,\gamma^m_n(F) \Mod)$ for which
$f'(*) \cap f(*)^\bot=\{0\}$.
We choose an isomorphism $\varphi: f(*) \overset{\cong}{\rightarrow} F^n$ and notice
that $U_{f(*)}$ is an open neighborhood of $f$ in $\Ob(\Func(\mathcal{B}H,\gamma^m_n(F) \Mod))$. \\
We then remark that the restriction of $\chi^{(n)}$ to $U_{f(*)}$ is the composite of the map
$$\begin{cases}
U_{f(*)} & \longrightarrow \Hom(H,\GL_n(F)) \\
f' & \longmapsto \left[h \mapsto \left(\varphi \circ \pi_{f(*)}^{f'(*)}\right) \circ
f'(h) \circ \left(\varphi \circ \pi_{f(*)}^{f'(*)}\right)^{-1} \right],
\end{cases}$$
and of the continuous map $\Hom(H,\GL_n(F)) \longrightarrow L^2(H)$ induced by composing with
the trace map $\tr: M_n(F) \rightarrow F$, both of which are easily shown to be continuous.
This shows that $\chi^{(n)}$ is continuous.

By the theory of linear representations of compact groups (cf.\ \cite{Brocker}),
two objects of $\Func(\mathcal{B}H,\gamma_n(F) \Mod)$ have the same
image by $\chi^{(n)}$ if and only if their associated representations are isomorphic.
Also, the image of $\chi^{(n)}$ is a discrete subset of $L^2(H)$. This proves that
$\Func(\mathcal{B}H,\gamma_n (F)\Mod)=\underset{j \in J}
{\coprod}\Func_j(\mathcal{B}H,\gamma_n (F) \Mod)$. \end{proof}

\subsubsection{The construction of $\Prol_H:\Func(\mathcal{B}H,\gamma(F) \Mod) \rightarrow \Func(\mathcal{E}G,\gamma(F) \Mod)^H$}\label{5.1.3}

Given a continuous functor $f: \mathcal{B}H \rightarrow \gamma(F) \Mod$, we wish to find
an $H$-invariant continuous functor $\tilde{f}: \mathcal{E}G \rightarrow \gamma(F) \Mod$
whose restriction is $f$ (by the preceding construction), and to do this in a continuous way with respect to $f$.
Here is the strategy:
we let $V: H \rightarrow \GL_n(F)$ denote a homomorphism which is isomorphic to
the linear representation of $H$ associated to $f$, and we consider only functors
$f': \mathcal{B}H \rightarrow \gamma(F) \Mod$ whose associated linear representation of $H$ is isomorphic to $V$.
In Step 1, every such $f'$ is considered as being associated to $\begin{cases}
H & \longrightarrow B_n(F^{(\infty)}) \\
h & \longmapsto f'(h)[\mathbf{B}']=\mathbf{B}'.V(h)
\end{cases}$ for some basis $\mathbf{B'}$ of $f'(*)$. Then
we extend the map
$\begin{cases}
H \times B_n(F^{(\infty)}) & \longrightarrow B_n(F^{(\infty)}) \\
(h,\mathbf{B}) & \longmapsto \mathbf{B}.V(h)
\end{cases}$ to a continuous map $G \times B_n(F^{(\infty)}) \longrightarrow B_n(F^{(\infty)})$ which is
equivariant for a certain group action. We use this last map to assign
an $H$-invariant continuous functor $\mathcal{E}G \rightarrow \gamma(F) \Mod$ to every $f'$.
In Step 2, we prove that this construction is continuous with respect to $f'$, by proving that
the choice of $\mathbf{B}'$ may locally  be made continuous with respect to $f$.
This construction is extended to morphisms in Step 3, where the claimed properties are checked.

\paragraph{Step 1: Constructing $\Prol_H^V:\Ob(\Func_j(\mathcal{B}H,\gamma(F) \Mod)
\rightarrow \Ob(\Func(\mathcal{E}G,\gamma(F) \Mod))$}

Let $j \in J$ and $V:H \rightarrow \GL_n(F)$ be a linear representation
of $H$ with isomorphism class $j$.
We let $\GL_V(F)$ denote the centralizer of $\im V$ in $\GL_n(F)$: this is a closed subgroup of $\GL_n(F)$.
We set $X_j:=\Ob(\Func_j(\mathcal{B}H,\gamma(F) \Mod))$ we let
$\Hom_j(H,\GL_n(F))$ denote the space consisting of the linear representations of $H$ that
are isomorphic to $V$. Finally, recall the decomposition $\Hom(H,\GL_n(F))=\underset{j_0\in J}{\coprod}\Hom_{j_0}(H,\GL_n(F))$.

For the canonical right-action of $H$ on $G$, the pair
$(G,H)$ is a relative $H$-CW-complex. The action of $\GL_n(F)$ on $B_n(F^{(\infty)})$ induces a right-action of
$\GL_V(F)$ on $B_n(F^{(\infty)})$.
Moreover, for every $k \in \mathbb{N}^*$, $B_n(F^k)$ has a structure
of smooth manifold such that $B_n(F^{k-1})$ is a closed smooth submanifold;  the action of $\GL_V(F)$ on $B_n(F^k)$
is free (and therefore proper), smooth, and stabilizes $B_n(F^{k-1})$.
We deduce from Theorems I and II of \cite{Illman}
that $(B_n(F^k),B_n(F^{k-1}))$ is a relative $\GL_V(F)$-CW-complex.
Therefore, for every $k \in \mathbb{N}^*$, the pair \\
$\left(G \times B_n(F^k),(H \times B_n(F^k)) \cup
(G \times B_n(F^{k-1}))\right)$ is a relative $(H \times \GL_V(F))$-CW-complex.

The action of $H \times \GL_V(F)$ on $G \times B_n(F^j)$ is free.
Finally, we may consider the right-action of
$H$ on $B_n(F^{(\infty)})$
defined by $\mathbf{B}.h=\mathbf{B}.V(h)$ for every $h\in H$ and every $\mathbf{B} \in B_n(F^{(\infty)})$.
This action is compatible with the right-action of $\GL_V(F)$.
This yields a right-action of $H \times \GL_V(F)$ on $B_n(F^{(\infty)})$.

\begin{lemme}
The map $(1_G,x) \mapsto x$ on $\{1_G\} \times B_n(F^{(\infty)})$ may be extended to an $(H \times \GL_V(F))$-map:
$$\gamma _V: G \times B_n(F^{(\infty)}) \longrightarrow
B_n(F^{(\infty)}).$$
\end{lemme}

\begin{proof}
We use an induction process.
Assume that, for some $k \in \mathbb{N}\setminus\{0,1\}$,
we have an equivariant map
$\gamma^{k-1}_V: G \times B_n(F^{k-1}) \longrightarrow
B_n(F^{(\infty)})$ such that $\forall x \in B_n(F^{k-1}), \, \gamma_V^{k-1}(1_G,x)=x$.
The action of $H$ on $B_n(F^{(\infty)})$, together with the inclusion of $B_n(F^k)$ into
$B_n(F^{(\infty)})$, define an $(H\times \GL_V(F))$-map:
$$\beta_V^k: \begin{cases}
H \times B_n(F^k) & \longrightarrow B_n(F^{(\infty)}) \\
(h,\mathbf{B}) & \longmapsto \mathbf{B}.V(h).
\end{cases}$$
This yields an $(H \times \GL_V(F))$-map:
$$(H \times B_n(F^k)) \cup (G \times B_n(F^{k-1}))
\Right5{\beta_V^k \cup \gamma_V^{k-1}} B_n(F^{(\infty)}),
$$
which we want to extend to $G \times B_n(F^k)$.
However, $H \times \GL_V(F)$ acts freely on $G \times B_n(F^k)$,
and the projection $B_n(F^{(\infty)}) \rightarrow *$ is a homotopy equivalence.
We deduce that the map $\beta_V^k \cup \gamma_V^{k-1}$
extends to an $(H\times \GL_V(F))$-map $\gamma_V^k:G \times B_n(F^k) \longrightarrow B_n(F^{(\infty)})$
which has the required property.
Therefore, the $\gamma_V^j$'s yield an $(H\times \GL_V(F))$-map
$\gamma_V: G \times B_n(F^{(\infty)}) \rightarrow B_n(F^{(\infty)})$. \end{proof}

Let $f \in X_j$ and $\mathbf{B}_f$ a basis such that $\varphi_V(\mathbf{B}_f)=f$. We set
$$\Prol_H^V(f)(1_G,g):\mathbf{B}_f \longmapsto \gamma_V(g,\mathbf{B}_f).$$
This definition does not depend on the choice of $\mathbf{B}_f$.
Indeed, if $\mathbf{B}'_f$ is another possible basis, there exists an
$M \in \GL_V(F)$ such that $\mathbf{B}'_f=\mathbf{B}_f.M$, and
so $\gamma_V(g,\mathbf{B}'_f)=\gamma_V(g,\mathbf{B}_f).M$ (since $\gamma_V$ is an $(H\times \GL_V(F))$-map).
Therefore, the same linear isomorphism maps $\mathbf{B}_f$ to $\gamma_V(g,\mathbf{B}_f)$, and
$\mathbf{B}'_f$ to $\gamma_V(g,\mathbf{B}'_f)$. The map $\Prol_H^V(f)$ can then easily be extended to a continuous functor
$\mathcal{E}G \rightarrow \gamma(F) \Mod$.

For any $h \in H$, $\gamma_V(h,\mathbf{B}_f)=\mathbf{B}_f.h=\mathbf{B}_f.V(h)=f(h)[\mathbf{B}]$,
and so $\Prol_H^V(f)(1_G,h)=f(h)$. This proves that $\res_H(\Prol_H^V(f))=f$.

For any $h \in H$ and $g \in G$, $\gamma_V(gh,\mathbf{B}_f)=
\gamma_V(g,\mathbf{B}_f).h$, and
$$\Prol_H^V(f)(1_G,g) \circ \Prol_H^V(f)(1_G,h) :\mathbf{B}_f
\longmapsto \mathbf{B}_f.h \longmapsto \gamma_V(g,\mathbf{B}_f).h=\gamma_V(gh,\mathbf{B}_f).$$
Therefore $\Prol_H^V(f)(1_G,gh)=\Prol_H^V(f)(1_G,g) \circ \Prol_H^V(f)(1_G,h)$.
This proves that $\Prol_H^V(f):\mathcal{E}G \longrightarrow
\gamma(F) \Mod$ is invariant by the right-action of $H$.

\paragraph{Step 2: Continuity of $\Prol_H^V:X_j \rightarrow \Ob(\Func(\mathcal{E}G,\gamma(F) \Mod))$} ${}$ \\
The action of $\GL_n(F)$ on $\Hom(H,\GL_n(F))$ by conjugation is continuous; the orbit of
$V$ is $\Hom_j(H,\GL_n(F))$ and its isotropy subgroup $\GL_V(F)$. This yields a continuous bijection
$$\alpha _V: \begin{cases}
\GL_n(F)/\GL_V(F) & \longrightarrow \Hom_j(H,\GL_n(F)) \\
\left[\varphi \right] & \longmapsto \varphi\circ V \circ \varphi^{-1}.
\end{cases}.$$
By Theorem 2.3.2 of \cite{Mneimne}, $\alpha_V$ is a homeomorphism since
$\Hom_i(H,\GL_n(F))$ is locally compact (cf.\ \cite{Goto-Kimura}) and
$\GL_n(F)$ is a $\sigma$-compact Lie group.

Recall that  $B_n(F^{(\infty)})$ denotes the limit of the sequence $B_n(F^j) \rightarrow B_n(F^{j+1}) \rightarrow \dots$, and
that the canonical projection $\tilde{\gamma_n}: B_n(F^{(\infty)}) \rightarrow G_n(F^{(\infty)})$
defines a $\GL_n(F)$-principal bundle. Let $\mathbf{B} \in B_n(F^{(\infty)})$. To
$\mathbf{B}$ may be assigned a continuous map:
$$\varphi_V(\mathbf{B}): \begin{cases}
H & \rightarrow \GL(\text{Vect}_F(\mathbf{B})) \\
h & \longmapsto \left[\mathbf{B} \mapsto \mathbf{B}.V(h)\right]
\end{cases} $$
which is simply the conjugate of $V$ by
the unique isomorphism from $F^n$ to  $\text{Vect}_F(\mathbf{B})$
which maps the canonical basis of $F^n$ to $\mathbf{B}$.
In particular, $\varphi_V(\mathbf{B})$ is a linear representation of $H$ and is isomorphic to $V$.
Also, for every $\mathbf{B'} \in B_n(F^{(\infty)})$, we have
$\varphi_V(\psi(\mathbf{B}))=\psi \circ \varphi_V(\mathbf{B}) \circ \psi^{-1}$,
where $\psi: \mathbf{B} \mapsto \mathbf{B}'$.

\begin{lemme}The map
$\varphi_V: \begin{cases}
B_n(F^{(\infty)}) & \longmapsto X_j \\
\mathbf{B} & \longmapsto \varphi_V(\mathbf{B})
\end{cases}$ is a $\GL_V(F)$-principal bundle.
\end{lemme}

\begin{proof}
Notice first that $\varphi_V$ is continuous since it is the composite of \\
$\begin{cases}
B_n(F^{(\infty)}) & \longrightarrow \Mor(\gamma(F) \Frame)^H \\
\mathbf{B} & \longmapsto [h \mapsto (\mathbf{B},\mathbf{B}.V(h))]
\end{cases}$ with $\Mor(\gamma(F) \Frame)^H \rightarrow  \Mor(\gamma(F) \Mod)^H$
induced by composition of the canonical functor $\gamma(F) \Frame \rightarrow \gamma(F) \Mod$.

Given $f\in X_j$, there is a linear isomorphism
$\psi: F^n \mapsto f(*)$ such that the representation associated to $f$ is the conjugate of $V$
by $\psi$. If we let $\mathbf{B}$ denote the image of the canonical basis of $F^n$ by $\psi$, then
$f=\varphi_V(\mathbf{B})$. This shows $\varphi_V$ is onto.

Let $\mathbf{B} \in B_n(F^{(\infty)})$ and $M \in \GL_V(F)$.
Then $M$ commutes with $V(h)$ for every $h \in H$ and it follows that
$\varphi_V(\mathbf{B}.M)=\varphi_V(\mathbf{B})$. If $\mathbf{B}$ and $\mathbf{B}'$
have the same image by $\varphi$, then the unique $M \in \GL_n(F)$ such that
$\mathbf{B}.M=\mathbf{B}'$ commutes with $V(h)$ for every $h \in H$,
and therefore belongs to $\GL_V(F)$. We deduce that the fiber of $\varphi_V$ over any $f\in X_j$ is isomorphic
to $\GL_V(F)$. We now simply need to find local trivializations of $\varphi_V$, and
in order to do so, it suffices to construct local sections since $B_n(F^{(\infty)}) \rightarrow G_n(F^{(\infty)})$
is a $\GL_n(F)$-principal bundle.

Let $f \in X_j$ and $\mathbf{B} \in B_n(F^{(\infty)})$ be a basis such that $\varphi_V(\mathbf{B})=f$.
Set $X^f_j:=\bigl\{g \in X_j : \; g(*)=f(*)\bigr\}$.
The projection:
$\begin{cases}
\GL_n(F) & \longrightarrow X^f_j  \\
M & \longmapsto \varphi_V(\mathbf{B}.M)
\end{cases}$ then defines a $\GL_V(F)$-principal bundle.
Indeed, if $\varphi_V(\mathbf{B})=f$, the linear isomorphism from $F^n$ to $f(*)$ which maps the canonical basis of
$F^n$ to $\mathbf{B}$ induces an isomorphism from $X_j^f$ to $\Hom_j(H,\GL_n(F))$, and we may
then use the fact that $\alpha_V$ is a homeomorphism and that $\GL_n(F) \rightarrow \GL_n(F)/\GL_V(F)$
is a $\GL_V(F)$-principal bundle since $\GL_V(F)$ is a Lie subgroup of $\GL_n(F)$ (cf.\ Theorem 4.3 of \cite{Brocker}).

Let $\delta: U_f \rightarrow \GL_n(F)$ denote a local section of the preceding $\GL_V(F)$-principal bundle,
where $U_f$ is an open neighborhood of $f \in X_i^f$, such that $\delta(f)=\mathbf{B}$.
Also, let $\beta: V_f \rightarrow B_n(F^{(\infty)})$
be a local section of the $\GL_n(F)$-principal bundle  $\tilde{\gamma_n}: B_n(F^{(\infty)}) \rightarrow
G_n(F^{(\infty)})$, where $V_f$ is an open neighborhood of $f(*)$ in $G_n(F^{(\infty)})$, such that
$\beta (f)=\mathbf{B}$.
For any $g\in X_j$ such that $g(*) \in \tilde{\gamma_n}^{-1}(V_f)$, we define
$\psi_g: f(*) \rightarrow g(*)$ as the linear isomorphism which maps $\mathbf{B}$ to $\beta (g(*))$.
Then $g \mapsto \psi_g$ is continuous. Set then $V_f':=\bigl\{g \in X_j : \; \psi_g^{-1} \circ g \circ \psi_g \in U_f\bigr\}$.
Obviously, $V_f'$ is an open neighborhood of $f$ in $X_V$, and we have a continuous map:
$$\begin{cases}
V_f' & \longrightarrow B_n(F^{(\infty)}) \\
g & \longmapsto \psi_g\left[\mathbf{B}.\delta(\psi_g^{-1} \circ g \circ \psi_g)\right].
\end{cases}$$
This is a local section of $\varphi_V$. Indeed, let $g\in V_{f'}$ and $g_1=\psi_g^{-1} \circ g \circ \psi_g$; then
$$\varphi_V[\psi_g(\mathbf{B}.\delta(g_1))]=\psi_g\circ \varphi_V[\mathbf{B}.\delta(g_1)] \circ \psi_g^{-1}
=\psi_g\circ g_1 \circ \psi_g^{-1}=g.$$
We conclude that $\varphi_V$ is a $\GL_V(F)$-principal bundle. \end{proof}

In order to prove that the map
$\Prol_H^V: X_j \rightarrow \Func(\mathcal{E}G,\mathcal{F})^H$ is continuous, it suffices
to prove that the continuity of the map:
$$\begin{cases}
G \times X_j & \longrightarrow \Mor(\gamma(F) \Mod) \\
(g,f) & \longmapsto \Prol_H^V(f)(1_G,g).
\end{cases}$$
Let $f \in X_j$. We choose an open neighborhood $U$ of $f$ in $X_i$
together with a local section $s:U \rightarrow B_n(F^{(\infty)})$ of $\varphi_V$.
Composing of the maps
$G \times U \Right4{\id_G \times s} G \times B_n(F^{(\infty)})
$,
$G \times B_n(F^{(\infty)}) \Right4{(\pi_2,\gamma_V)}
 B_n(F^{(\infty)}) \times  B_n(F^{(\infty)})$
 (where $\pi_2$ denotes the projection on the second factor), and
$$\begin{cases}
B_n(F^{(\infty)}) \times  B_n(F^{(\infty)}) & \longrightarrow \Mor(\gamma(F) \Mod) \\
(\mathbf{B},\mathbf{B}') & \longmapsto \left[ \mathbf{B} \rightarrow
\mathbf{B}'\right]
\end{cases}$$
yields $$\begin{cases}
G \times U & \longrightarrow \Mor(\gamma(F) \Mod)\\
(g,f) & \longmapsto \Prol_H^V(f)(1_G,g).
\end{cases}$$
Since all three maps are continuous, their composite is continuous, hence so is $\Prol_H^V$.

\paragraph{Step 3: The functor $\Prol_H$} ${}$ \\
For every $j \in J$, we choose a representative $V$, and, since
$\Func(\mathcal{B}H,\gamma(F) \Mod)=\underset{j \in J}{\coprod}
\Func_j(\mathcal{B}H,\gamma(F) \Mod)$, we finally obtain a continuous map
$$\Prol_H: \Ob(\Func(\mathcal{B}H,\gamma(F) \Mod)) \longrightarrow
\Ob(\Func(\mathcal{E}G,\gamma(F) \Mod)^H)$$
such that $\res_H(\Prol_H(f))=f$ for all $f\in \Ob(\Func(\mathcal{B}H,\gamma(F) \Mod))$.
We wish to extend this map to a functor
$\Func(\mathcal{B}H,\gamma(F) \Mod) \longrightarrow \Func(\mathcal{E}G,\gamma(F) \Mod)^H$. This is based on the following lemma.

\begin{lemme}\label{extensionmorphism}
Let $(f,f')\in \Ob(\Func(\mathcal{E}G,\gamma(F) \Mod)^H)^2$ together with a morphism $\alpha: \res_H(f) \rightarrow
\res_H(f')$. Then there is a unique morphism $\tilde{\alpha}:f \rightarrow f'$ in $\Func(\mathcal{E}G,\gamma(F) \Mod)^H$
such that $\alpha=\res_H(\tilde{\alpha})$.
\end{lemme}

\begin{proof}
For all $g \in G$, we define $\tilde{\alpha}(g):f(g) \rightarrow f'(g)$ as the unique linear isomorphism
which makes the square

$$\begin{CD}
f(g) @>{\tilde{\alpha}(g)}>> f'(g) \\
@V{f(g,1_G)}VV @VV{f'(g,1_G)}V \\
f(1_G) @>>{\alpha(*)}> f'(1_G)
\end{CD}$$
commute. By definition, $\tilde{\alpha}$ is the unique morphism from $f$ to $f'$ in
$\Func(\mathcal{E}G,\gamma(F) \Mod)$. We only need to prove that it is invariant under the action of $H$. \\
Let then $(g,h)\in G\times H$, and consider the commutative diagram:
$$\xymatrix@1{
f(gh) \ar[dd]_{f(gh,g)} \ar[rr]^{\tilde{\alpha}(gh)} \ar[rd]_{f(gh,h)} & & f'(gh) \ar'[d][dd]_(0.3){f'(gh,g)}
 \ar[rd]_{f'(gh,h)} \\
& f(h) \ar[dd]_(0.35){f(h,1_G)} \ar[rr]^(.3){\tilde{\alpha}(h)} & & f'(h) \ar[dd]_{f'(h,1_G)} \\
f(g) \ar'[r][rr]^(0.3){\tilde{\alpha}(g)} \ar[rd]_{f(g,1_G)} & & f'(g)  \ar[rd]_{f'(g,1_G)}\\
& f(1_G) \ar[rr]^(0.3){\tilde{\alpha}(1_G)} & & f'(1_G)
}$$
Since $\alpha$ is a morphism from $\res_H(f)$ to $\res_H(f')$, we deduce that $\tilde{\alpha}(h)=\tilde{\alpha}(1_G)
=\alpha(*)$. Since $f$ and $f'$ are invariant under the action of $H$, we have
$f(gh,h)=f(g,1_G)$, and $f'(gh,h)=f'(g,1_G)$. We deduce that $\tilde{\alpha}(gh)=\tilde{\alpha}(g)$.  \end{proof}

Given a morphism $\alpha: f \rightarrow f'$ in $\Func(\mathcal{B}H,\gamma(F) \Mod)$, we now define
$\Prol_H(\alpha)$ as the morphism in $\Func(\mathcal{E}G,\gamma(F) \Mod)^H$ which is associated to
$\Prol_H(f)$, $\Prol_H(f')$ and $\alpha$ as in Lemma \ref{extensionmorphism}. That $\Prol_H$ defines a functor is then obvious
from the uniqueness in Lemma \ref{extensionmorphism}. That $\Prol_H$ is continuous on morphisms is also obvious,
and we obtain a continuous functor
$$\Prol_H:  \Func(\mathcal{B}H,\gamma(F) \Mod) \longrightarrow
\Func(\mathcal{E}G,\gamma(F) \Mod)^H.$$
Obviously $$\res_H \circ \Prol_H=\id_{\Func(\mathcal{B}H,\gamma(F) \Mod)}.$$
It remains to prove that $|\Prol_H| \circ |\res_H|$ is homotopic to the identity map of
$|\Func(\mathcal{E}G,\gamma(F) \Mod)^H|$.
Let $f$ be an object of $\Func(\mathcal{E}G,\gamma(F) \Mod)^H$.
Then $\res_H(f)=\res_H((\Prol_H \circ \res_H)(f))$, and we can thus consider the morphism
$f \rightarrow \Prol_H(\res_H(f))$ of $\Func(\mathcal{E}G,\gamma(F) \Mod)^H$ which is associated to
$f$, $\Prol_H(\res_H(f))$ and $\id_{f(1_G)}$ by Lemma \ref{extensionmorphism}. This defines
a continuous natural transformation
$$\eta: \id_{\Func(\mathcal{E}G,\gamma(F) \Mod)^H}
\longrightarrow (\Prol_H \circ\res_H)(f).$$
Therefore $|\Prol_H|$ is a homotopy inverse of $|\res_H|$.

\paragraph{Conclusion:}
The space $(\vEc_G^{F,\infty})^H$ has the homotopy type of
$|\Func(\mathcal{B}H,\gamma(F) \Mod)|$.

\begin{Rems}
\begin{enumerate}[(i)]
\item
The previous constructions can also be achieved in the cases of $i\vEc_G^{F,\infty}$
and $s\vEc_G^{F,\infty}$. The only difference is that we need to consider Hilbert representations of $H$.
Proposition \ref{decomprop} obviously implies a similar
result in these cases, since the topological categories $\Func(\mathcal{B}H,\gamma(F) \imod)$ and
$\Func(\mathcal{B}H,\gamma(F) \smod)$ may be seen as embedded in $\Func(\mathcal{B}H,\gamma(F) \Mod)$
(and since two Hilbert representations of $H$ are isomorphic as linear representations of $H$ if and only if they are
isomorphic as Hilbert representations of $H$).
The construction of $\Prol_H$ is then achieved by replacing $\GL_n(F)$ by $U_n(F)$ (resp.\
by $\Sim_n(F)$), and free $n$-tuples by orthonormal $n$-tuples (resp.\
simi-orthonormal $n$-tuples). We deduce that $(i\vEc_G^{F,\infty})^H$ is homotopy equivalent
to $\left|\Func(\mathcal{B}H,\gamma(F) \imod)\right|$, and that $(s\vEc_G^{F,\infty})^H$ is homotopy equivalent
to $\left|\Func(\mathcal{B}H,\gamma(F) \smod)\right|$. Interestingly, we may prove directly
from there that $(i\vEc_G^{F,\infty})^H$ is homotopy equivalent to
$(s\vEc_G^{F,\infty})^H$. Indeed, any continuous functor from $\mathcal{B}H$ to $\gamma(F) \smod$
induces a functor from $\mathcal{B}H$ to $\gamma(F) \imod$, since $H$ is compact.
We thus define a functor
$$\begin{cases}
\Func(\mathcal{B}H,\gamma(F) \smod) & \longrightarrow \Func(\mathcal{B}H,\gamma(F) \imod) \\
f & \longmapsto f \\
\alpha:( f \rightarrow f') & \longmapsto \frac{\alpha}{\|\alpha \|}:(f \rightarrow f'),
\end{cases}$$
and it is easy to check that it induces a homotopy inverse of the inclusion
$\left|\Func(\mathcal{B}H,\gamma(F) \imod)\right| \subset \left|\Func(\mathcal{B}H,\gamma(F) \smod)\right|$.
\item When $G$ is discrete, the previous construction may adapted to any of the spaces
$\vEc_G^{F,m}$, for $m\in \mathbb{N}^*$. By Lemma \ref{extensionmorphism}, it suffices to construct
a section of the restriction map on objects. For every class $z$ in $G/H$, we choose an element
$g_z \in G$ such that $[g_z]=z$. Let then $f:\mathcal{B}H \rightarrow \gamma^{(m)}(F) \Mod$.
We define a functor $\tilde{f}: \mathcal{E}G \rightarrow \gamma^{(m)}(F) \Mod$
by $\tilde{f}(1_G,g_z.h):=f(h)$ for all $z\in G/H$ and $h \in H$.
We can easily check that this defines an $H$-invariant functor, that the map $f \mapsto \tilde{f}$
is continuous, and that $\res_H(\tilde{f})=f$. This proves that
 $(\vEc_G^{F,m})^H$ is homotopy equivalent to $|\Func(\mathcal{B}H,\gamma^{(m)}(F) \Mod)|$ for all $m \in \mathbb{N}$,
whenever $G$ is discrete. The same line of reasoning also applies to $i\vEc_G^{F,m}$ and
$s\vEc_G^{F,m}$.
\end{enumerate}
\end{Rems}

\subsection{On the homotopy type of $|\Func(\mathcal{B}H,\gamma(F) \Mod)|$}\label{5.2}

In this section, we will prove the following result.

\begin{prop}\label{homotopyCW}
Let $j \in \Rep_F(H)$ be an $n$-dimensional class.
Then $|\Func_j(\mathcal{B}H,\gamma(F) \Mod)|$
has the homotopy type of a CW-complex.
\end{prop}

For any $k \in \mathbb{N}$, we set
$\mathcal{C}_{j,k}:=\Func_j(\mathcal{B}H,\gamma_n^k(F) \Mod)$
and $\mathcal{C}_j:=\Func_j(\mathcal{B}H,\gamma_n(F) \Mod)$.
Clearly, $\mathcal{C}_j =\underset{\underset{k \in \mathbb{N}}{\rightarrow}}
{\lim} (\mathcal{C}_{j,k})$.

We will show that the nerve
$(\mathcal{N}(\mathcal{C}_j)_m)_{m\in \mathbb{N}}$ of $\mathcal{C}_j$ is a good simplicial space (cf.\ Appendix A of \cite{Segal-cat})
and that each of its components $\mathcal{N}(\mathcal{C}_j)_m$ has the homotopy type of a CW-complex.
To do so, we will equip the space $\mathcal{N}(\mathcal{C}_{j,k})_m$ with a structure of smooth manifold, for
every pair $(k,m)\in \mathbb{N}^* \times \mathbb{N}$, and prove that these structures are compatible with standard inclusions.

\subsubsection{A structure of smooth manifold on
$\mathcal{N}(\gamma_n^k(F) \Mod)_m$}\label{5.2.1}

We start with a short reminder of the canonical
manifold structure on $G_n(F^k)$. For every $x \in G_n(F^k)$, let
$U_x$ denote the subspace of $G_n(F^k)$ consisting of those elements $y$
such that $y \cap x^\bot=\{0\}$ and notice that $U_x$ is an open neighborhood of $x$ in $G_n(F^k)$.
For any pair $(y,z)$ of subspaces of $F^k$, let $\pi_z^y$ denote the restriction to $y$ of the orthogonal projection on $z$. In the case $y=F^k$, we set $\pi_z:=\pi_z^{F^k}$. We then obtain a chart
$\psi_x: \begin{cases}
U_x & \overset{\cong}{\longrightarrow} L(x,x^\bot) \\
y & \longmapsto \pi_{x^\bot} \circ (\pi_x^y)^{-1}.
\end{cases}$ The chart transitions are smooth.

For every $x \in G_n(F^k)$, we choose an isomorphism
$\alpha_x: F^n \overset{\cong}{\rightarrow} x$ and set
$$\varphi_x:\begin{cases}
U_x \times \GL_n(F) & \longrightarrow \tilde{\gamma}_n^k(F)^{-1}(U_x) \\
(y,\mathbf{B}) & \longmapsto \pi_y(\alpha_x(\mathbf{B})).
\end{cases}$$
Clearly, $\varphi_x$ is a homeomorphism over $U_x$, whilst
$(U_x,\psi_x)_{x\in G_n(F^k)}$ is a system of local trivializations of $\gamma_n^k(F)$,
and the trivialization transitions are smooth. We have just defined a structure
of smooth $\GL_n(F)$-principal bundle on $\tilde{\gamma}_n^k(F)$.

For $(x_0,\dots,x_m) \in G_n(F^k)^{m+1}$, let $U^{(k)}_{x_0,\dots,x_m}$ denote the set consisting of those $m$-simplices
$y_0 \rightarrow \dots \rightarrow y_m$
in $\mathcal{N}(\gamma_n^k \Mod)_m$ such that $y_i \in U_{x_i}$ for all $i \in \{0,\dots,m\}$.
The family $(U^{(k)}_{x_0,\dots,x_m})_{(x_0,\dots,x_m)\in G_n(F^k)^{m+1}}$
is an open cover of
$\mathcal{N}(\gamma_n^k(F) \Mod)_k$.
We then obtain charts
$$\psi^{(k)}_{x_0,\dots,x_m}:\begin{cases}
U^{(k)}_{x_0,\dots,x_m} & \longrightarrow \left[ \underset{j=0}{\overset{m}{\prod}}
L(x_j,x_j^\bot)\right] \times (\GL_n(F))^m \\
(y_0 \overset{\alpha_0}{\rightarrow} \cdots \overset{\alpha_{m-1}}{\rightarrow}
y_m) & \longmapsto \left((\psi_{x_j}(y_j))_{0 \leq j \leq m},
(\alpha_{x_{j+1}}^{-1} \circ \pi_{x_{j+1}} \circ \alpha_j \circ (\pi^{y_j}_{x_j})^{-1}
\circ \alpha_{x_j})_{0\leq j \leq m-1}
\right)
\end{cases}$$
and, again, the chart transitions are smooth. We therefore end up with
a structure of smooth manifold on $\mathcal{N}(\gamma_n^k(F) \Mod)_m$.

\subsubsection{A structure of smooth manifold on
$\mathcal{N}(\mathcal{C}_{j,k})_m$}\label{5.2.2}

For $m \in \mathbb{N}$, $\sigma \subset \{0,\dots,m-1\}$, and a simplicial space $X$, we set
$X_{m,\sigma}:=\underset{j \in \sigma}{\bigcap} s_j^m(X_{m-1})$.

Let $V:H \rightarrow \GL_n(F)$ be a linear representation with isomorphism class $j$.
We recall the homeomorphism $\alpha_{V}:\begin{cases}
\GL_n(F)/\GL_V(F) & \longrightarrow \Hom_j(H,\GL_n(F)) \\
\left[M\right] & \longmapsto M.V.M^{-1}.
\end{cases}$
There is a unique structure of smooth manifold on
$\Hom_j(H,\GL_n(F))$ such that $\alpha_V$ is a diffeomorphism.
We can check that this structure does not depend on the choice of $V$.
To see this, we remind the reader of the following classical lemma.

\begin{lemme}
Let $G$ be a Lie group, $H$ a closed subgroup of $G$, and $g \in G$.
Then the unique $G$-map
$\beta_{G,H,g}:\begin{cases}
G/H & \longrightarrow G/gHg^{-1} \\
[1_G] & \longmapsto [g]
\end{cases}$
which assigns $[g]$ to $[1_G]$ is a diffeomorphism.
\end{lemme}

Let $V'$ be another representative of $j$. If $M \in \GL_n(F)$ is such that $V'=M.V.M^{-1}$,
then $\alpha_{V'} =\alpha_{V} \circ \beta_{\GL_n(F),\GL_V(F),M}$, and so $\alpha_{V'}$ is a diffeomorphism
if and only if $\alpha_V$ is a diffeomorphism.

The group $\GL_n(F)$ acts on the left on $\Hom_j(H,\GL_n(F))$ (by the conjugation action).
This action is smooth, because it is induced by $\alpha_V$ and the smooth map \\
$\begin{cases}
\GL_n(F) \times (\GL_n(F)/\GL_V(F)) & \longrightarrow \GL_n(F)/\GL_V(F) \\
(M,\left[M'\right]) & \longmapsto \left[M.M' \right].
\end{cases}$

Consider the projection
$$\pi_{j,k,m}: \begin{cases}
\mathcal{N}(\mathcal{C}_{j,k})_m & \longrightarrow
\mathcal{N}(\gamma_{n}^k(F) \Mod)_m \\
\left(f_0 \overset{\alpha_0}{\rightarrow}
\cdots \overset{\alpha_{m-1}}{\rightarrow} f_m \right) & \longmapsto
\left(f_0(*) \overset{\alpha_0}{\rightarrow}
\cdots \overset{\alpha_{m-1}}{\rightarrow} f_m(*) \right).
\end{cases}$$
The fiber of $\pi_{j,k,m}$ over any point of $\mathcal{N}(\gamma_n^k \Mod)$ is clearly isomorphic to
$\Hom_j(H,\GL_n(F))$. It is now easy to check that the following result is true.

\begin{prop}\label{manifoldstructure}
Let $k \in \mathbb{N}^*$. \\
For every $(x_0,\dots,x_m) \in G_n(F^k)^{m+1}$, we set
$$\Phi^{(k)}_{x_0,\dots,x_m}:\begin{cases}
U^{(k)}_{x_0,\dots,x_m} \times \Hom_j(H,\GL_n(F))
& \longrightarrow \pi_{j,k,m}^{-1}(U_{x_0,\dots,x_m}^{(k)}) \\
(y_0 \overset{\alpha_0}{\rightarrow} \cdots \overset{\alpha_{m-1}}{\rightarrow}
y_m,f) & \longmapsto
\left[(\pi_{x_0}^{y_0})^{-1} \circ \alpha_{x_0} \circ f \circ \alpha_{x_0}^{-1}
\circ \pi_{x_0}^{y_0}\right] \overset{\alpha_0}{\rightarrow}
\cdots
\end{cases}$$
Then $\Phi^{(k)}_{x_0,\dots,x_m}$ is a homeomorphism over $U_{x_0,\dots,x_m}^{(k)}$.
\item Let $f_0 \overset{\alpha_0}{\rightarrow} \dots \overset{\alpha_{m-1}}{\rightarrow} f_m$
be an $m$-simplex in $\mathcal{N}(\mathcal{C}_{j,k})_m$.
For all $\ell \in \{0,\dots,m\}$, we set $x_\ell:=f_\ell(*)$, and let $x_\ell^\bot$ and
$y_\ell$ respectively denote the orthogonal complement of $x_\ell$ in $F^k$ and in $F^{k+1}$.

Then:
\begin{enumerate}[(i)]
\item
$\left(\psi_{x_0,\dots,x_m}^{(k)} \times \id_{\Hom_j(H,\GL_n(F))}\right)
\circ (\Phi_{x_0,\dots,x_m}^{(k)})^{-1}$ is a homeomorphism onto

$$\left[\underset{\ell=0}{\overset{m}{\prod}}L(x_\ell,x_\ell^\bot)\right]
\times \GL_n(F)^m \times \Hom_j(H,\GL_n(F)).$$
\item The restriction of
$\left(\psi_{x_0,\dots,x_m}^{(k+1)} \times \id_{\Hom_j(H,\GL_n(F))}\right)
\circ (\Phi_{x_0,\dots,x_m}^{(k+1)})^{-1}$ to \\
$\pi_{j,k+1,m}^{-1}(U_{x_0,\dots,x_m}^{(k+1)}) \cap
\mathcal{N}(\mathcal{C}_{j,k})_m$
is precisely $\left(\psi_{x_0,\dots,x_m}^{(k)} \times \id_{\Hom_j(H,\GL_n(F))}\right)
\circ (\Phi_{x_0,\dots,x_m}^{(k)})^{-1}$.
\item If $f_0 \overset{\alpha_0}{\rightarrow} \cdots
\overset{\alpha_{m-1}}{\rightarrow} f_m \in
\mathcal{N}(\mathcal{C}_{j,k})_{m,\sigma}$ for some $\sigma \subset \{0,\dots,m\}$, then \\
$\left(\psi_{x_0,\dots,x_m}^{(k)} \times \id_{\Hom_j(H,\GL_n(F))}\right)
\circ (\Phi_{x_0,\dots,x_m}^{(k)})^{-1}$ maps $\pi_{j,k,m}^{-1}(U_{x_0,\dots,x_m}^{(m)}) \cap
\mathcal{N}(\mathcal{C}_{j,k})_{m,\sigma}$ to
\begin{multline*}
\left\{((\varphi_\ell)_{0\leq \ell \leq m},(\psi_\ell)_{0\leq \ell \leq m-1})
\in \left[\underset{\ell=0}{\overset{m}{\prod}}L(x_\ell,x_\ell^\bot)\right]
\times \GL_n(F)^m: \forall \ell \in \sigma,
\begin{cases}
\varphi_\ell=\varphi_{\ell+1} \\
\psi_\ell=I_n
\end{cases}
\right\}
\\
\times \Hom_j(H,\GL_n(F))
\end{multline*}
\end{enumerate}
\end{prop}

\begin{cor}
For every $m \in \mathbb{N}$,
$\mathcal{N}(\mathcal{C}_{j,k})_m$ has a structure of smooth manifold.
\end{cor}

\begin{proof}
The maps $\Phi_{x_0,\dots,x_m}^{(k)}$ define a system of local trivializations.
The transitions are smooth hence $\pi_{j,k,m}$
is equipped with a structure of smooth fibre bundle with fiber $\Hom_j(H,\GL_n(F))$
and structural group $\GL_n(F)$.
\end{proof}

\begin{cor}\label{submanifold}
Let $k \in \mathbb{N}$,
$m \in \mathbb{N}^*$, $\sigma \subset \sigma ' \subset \{0,\dots,m-1\}$.
\begin{itemize}
\item $\mathcal{N}(\mathcal{C}_{j,k})_{m,\sigma}$
is a closed smooth submanifold of $\mathcal{N}(\mathcal{C}_{j,k})_m$.
\item $\mathcal{N}(\mathcal{C}_{j,k})_{m,\sigma'}$
is a closed smooth submanifold of
$\mathcal{N}(\mathcal{C}_{j,k})_{m,\sigma}$.
\item $\mathcal{N}(\mathcal{C}_{j,k})_{m,\sigma}$
is a closed smooth submanifold of
$\mathcal{N}(\mathcal{C}_{j,k+1})_{m,\sigma}$.
\end{itemize}
\end{cor}

\begin{proof}
By Proposition \ref{manifoldstructure}, it suffices to proof the closeness.

Let $\ell \in \{0,\dots,m-1\}$, and $f_0 \overset{\alpha_0}{\rightarrow} \dots
\overset{\alpha_{m-1}}{\rightarrow} f_m$ be an $m$-simplex in
$\mathcal{N}(\mathcal{C}_{j,k})_{m} \setminus \mathcal{N}(\mathcal{C}_{j,k})_{m,\{\ell\}}$. \\
If $f_\ell(*) \neq f_{\ell+1}(*)$, then we choose a pair $(U,U')$ of disjoint open subsets
of $G_n(F^k)$ such that $f_\ell(*) \in U$ and $f_{\ell+1}(*) \in U'$;
then
$\left\{(g_0 \rightarrow \cdots \rightarrow g_m) \in \mathcal{N}(\mathcal{C}_{j,k})_{m}:
g_\ell(*) \in U, g_{\ell+1}(*) \in U'\right\}$ is an open neighborhood of
$f_0 \overset{\alpha_0}{\rightarrow} \cdots \overset{\alpha_{m-1}}{\rightarrow} f_m$
in $\mathcal{N}(\mathcal{C}_{j,k})_{m}$ which intersects $\mathcal{N}(\mathcal{C}_{j,k})_{m,\{\ell\}}$ trivially. \\
If $f_\ell(*)=f_{\ell+1}(*)$, then
$$(\psi^{(k)}_{x_0,\dots,x_m} \circ \pi_{j,k,m})^{-1}\left(\left[\underset{l=0}{\overset{m}{\prod}}
L(x_l,x_l^\bot) \right] \times \GL_n(F)^\ell \times (\GL_n(F)\setminus \{I_n\})
\times \GL_n(F)^{m-\ell-1}\right)$$
is an open neighborhood of $f_0 \overset{\alpha_0}{\rightarrow} \cdots \overset{\alpha_{m-1}}{\rightarrow} f_m$
in $\mathcal{N}(\mathcal{C}_{j,k})_{m}$ which intersects $\mathcal{N}(\mathcal{C}_{j,k})_{m,\{\ell\}}$
trivially. This proves that $\mathcal{N}(\mathcal{C}_{j,k})_{m,\{\ell\}}$ is a closed subspace
of $\mathcal{N}(\mathcal{C}_{j,k})_{m}$. We deduce that $\mathcal{N}(\mathcal{C}_{j,k})_{m,\sigma'}$
is a closed subspace of $\mathcal{N}(\mathcal{C}_{j,k})_{m,\sigma}$ whenever
$\sigma \subset \sigma ' \subset \{0,\dots,m-1\}$.

Let $\sigma \subset \{0,\dots,m-1\}$, and $f_0 \overset{\alpha_0}{\rightarrow} \cdots
\overset{\alpha_{m-1}}{\rightarrow} f_m$ be an $m$-simplex in
$\mathcal{N}(\mathcal{C}_{j,k+1})_{m,\sigma} \setminus \mathcal{N}(\mathcal{C}_{j,k})_{m,\sigma}$.
Then there exists an index $\ell \in \{0,\dots,m\}$ such that $f_\ell(*) \in
G_n(F^{k+1})\setminus G_n(F^k)$. The subset of $\mathcal{N}(\gamma_n^{k+1}(F) \Mod)_m$ consisting
of those $m$-simplices $y_0 \rightarrow \cdots \rightarrow y_m$ such that
$y_\ell \in G_n(F^{k+1})\setminus G_n(F^k)$,
is open in $\mathcal{N}(\gamma_n^{k+1}(F)\Mod)_m$, and its inverse image by $\pi_{j,k+1,m}$
is an open neighborhood of $f_0 \overset{\alpha_0}{\rightarrow} \cdots
\overset{\alpha_{m-1}}{\rightarrow} f_m$ in $\mathcal{N}(\mathcal{C}_{j,k+1})_m$
which intersects $\mathcal{N}(\mathcal{C}_{j,k})_{m,\sigma}$ trivially. It follows that
$\mathcal{N}(\mathcal{C}_{j,k})_{m,\sigma}$ is a closed subspace of
$\mathcal{N}(\mathcal{C}_{j,k+1})_{m,\sigma}$.
\end{proof}

\begin{cor}
For every $m \in \mathbb{N}$,
$\mathcal{N}(\mathcal{C}_j)_m$ has the homotopy type of a CW-complex.
\end{cor}

\begin{proof}
Let $m \in \mathbb{N}$.
By Proposition \ref{submanifold}, $\mathcal{N}(\mathcal{C}_{j,k-1})_m$
is a closed submanifold of $\mathcal{N}(\mathcal{C}_{j,k})_m$ for every $k \in \mathbb{N}\setminus \{0,1\}$.
In particular, the inclusion $\mathcal{N}(\mathcal{C}_{j,k-1})_m \subset \mathcal{N}(\mathcal{C}_{j,k})_m$ is
a closed cofibration, and it follows that $\mathcal{N}(\mathcal{C}_j)_m$, which is the
colimit of this sequence of inclusions, has the homotopy type of the homotopy colimit of the sequence
$$\mathcal{N}(\mathcal{C}_{j,1})_m \subset \cdots \subset
\mathcal{N}(\mathcal{C}_{j,k-1})_m \subset \mathcal{N}(\mathcal{C}_{j,k})_m \subset \cdots .$$

For every $k \in \mathbb{N}^*$, we equip $\mathcal{N}(\mathcal{C}_{j,k})_m$ with a CW-complex
structure (this is possible since $\mathcal{N}(\mathcal{C}_{j,k})_m$ is a smooth manifold).
Thus every inclusion $\mathcal{N}(\mathcal{C}_{j,k})_m \subset
\mathcal{N}(\mathcal{C}_{j,k+1})_m$ is homotopic to a cellular map $f_k$.
The homotopy colimit of the sequence
$$\mathcal{N}(\mathcal{C}_{j,1})_m \overset{f_1}{\rightarrow} \cdots \overset{f_{k-1}}{\rightarrow}
\mathcal{N}(\mathcal{C}_{j,k})_m \overset{f_k}{\rightarrow} \mathcal{N}(\mathcal{C}_{j,k+1})_m \subset \cdots$$
is then a CW-complex and has the homotopy type of the previous homotopy colimit. It follows
that $\mathcal{N}(\mathcal{C}_j)_m$ has the homotopy type of a CW-complex. \end{proof}

\begin{cor}\label{goodspace}
Let $m \in \mathbb{N}$ and $\sigma \subset \sigma ' \subset \{0,\dots,m\}$. \\
Then $\mathcal{N}(\mathcal{C}_j)_{m,\sigma'} \hookrightarrow
\mathcal{N}(\mathcal{C}_j)_{m,\sigma}$ is a closed cofibration.
\end{cor}

\begin{proof}  We notice that
$\mathcal{N}(\mathcal{C}_j)_{m,\sigma'}=
\underset{\underset{k \in \mathbb{N}^*}{\longrightarrow}}{\lim}
\mathcal{N}(\mathcal{C}_{j,k})_{m,\sigma'}$
and that $\mathcal{N}(\mathcal{C}_j)_{m,\sigma}=
\underset{\underset{k \in \mathbb{N}^*}{\longrightarrow}}{\lim}
\mathcal{N}(\mathcal{C}_{j,k})_{m,\sigma}$.
Furthermore, for every $k \in \mathbb{N}^*$,
$\mathcal{N}(\mathcal{C}_{j,k})_{m,\sigma}
\cap \mathcal{N}(\mathcal{C}_{j,k+1})_{m,\sigma'}=
\mathcal{N}(\mathcal{C}_{j,k})_{m,\sigma'}$.
The corollary is thus derived from Proposition \ref{submanifold},
from the fact that the inclusion of a closed smooth submanifold is a closed
cofibration, and from the following lemma (the proof of which is straightforward).
\end{proof}

\begin{lemme}
Let $$\begin{CD}
A_0 @>{\alpha_0}>> A_1 @>>> \cdots @>{\alpha_{n-1}}>> A_n @>{\alpha_n}>>
  A_{n+1} @>{\alpha_{n+1}}>> \cdots \\
@V{f_0}VV @V{f_1}VV  & & @V{f_n}VV  @V{f_{n+1}}VV \\
B_0 @>{\beta_0}>> B_1 @>>> \cdots @>{\beta_{n-1}}>> B_n @>{\beta_n}>>
  B_{n+1} @>{\beta_{n+1}}>> \cdots
\end{CD}$$
be a commutative diagram in the category of topological spaces
such that
\begin{enumerate}[(i)]
\item all morphisms are closed cofibrations;
\item for every $n \in \mathbb{N}$, $A_{n+1} \cap B_n=A_n$,
where $A_n$, $A_{n+1}$ and $B_n$ are seen as subspaces of $B_{n+1}$.
\end{enumerate}
Then $\underset{\underset{n \in \mathbb{N}}{\longrightarrow}}{\lim} A_n \longrightarrow
\underset{\underset{n\in \mathbb{N}}{\longrightarrow}}{\lim} B_n$ is a closed cofibration.
\end{lemme}

\begin{proof}[Proof of Proposition \ref{homotopyCW}:]
We deduce from Corollary \ref{goodspace} that $\mathcal{C}_j$ is a good simplicial
space and it follows from point (iv) of Proposition A.1 of \cite{Segal-cat} that
$|\mathcal{C}_j|$ has the homotopy type of
$\|\mathcal{C}_j\|$. Also, every component of the nerve of
$\mathcal{C}_j$ has the homotopy type of a CW-complex, and we deduce, using point (i) of Proposition A.1 of
\cite{Segal-cat}, that $\|\mathcal{C}_j\|$ has the homotopy type of a CW-complex.
Hence $|\mathcal{C}_j|$ has the homotopy type of a CW-complex. \end{proof}

\noindent From Proposition \ref{homotopyCW}, we deduce that $|\Func(\mathcal{B}H,\gamma(F) \Mod)|
=\underset{j \in J}{\coprod}|\Func_j(\mathcal{BH},\gamma(F) \Mod)|$
has the homotopy type of a CW-complex, hence so does $(\vEc_G^{F,\infty})^H$.
This completes the proof of Theorem \ref{homotopytype}.

\begin{Rems}
\begin{enumerate}[(i)]
\item In a similar fashion, we may prove that
$|\Func(\mathcal{B}H,\gamma^{(m)}(F) \Mod)| $ has the homotopy type of a CW-complex
for every $m\in \mathbb{N}^*$. Since, for $m\in \mathbb{N}^*$, we do not know how to construct an extension functor
as in Section \ref{5.1.3}, this is not really interesting. However, when $G$ is discrete and $H$ is a finite subgroup of it,
this shows that $(\vEc_G^{F,m})^H$ has the homotopy type of a CW-complex for every
$m \in \mathbb{N}^*\cup \{\infty\}$.
\item The proof of Proposition \ref{homotopyCW} may easily be adapted to the case of
$|\Func_j(\mathcal{BH},\gamma(F) \imod)|$: it suffices to replace $\GL_n(F)$ by $U_n(F)$ everywhere. We then deduce that
$(i\vEc_G^{F,\infty})^H$ has the homotopy type of a CW-complex, and $(s\vEc_G^{F,\infty})^H$ also does, since
it is homotopy equivalent to it.
\end{enumerate}
\end{Rems}

\section{A note on $\Gamma$-spaces}\label{C}

When $\underline{A}$ is a $\Gamma-G$-space, $BA$ will denote its thick
geometric realization (where a \emph{colimit} is used instead of a \emph{homotopy colimit}).
For every finite set $S$, we have a $\Gamma-G$-space $\underline{A}(S \times -)$ , and
we let $B\underline{A}(S)$ denote its thick geometric realization.
Then $B\underline{A}$ is a $\Gamma-G$-space, and $B\underline{A}(\mathbf{1})=BA$.

It is only when $\underline{A}(\mathbf{0})$ is a point that there is a well-defined canonical map
$\underline{A}(\mathbf{1}) \rightarrow \Omega BA$. It thus appears that
the map Segal deals with in \cite{Segal-cat} has to be understood as the composite of the canonical map
$\underline{A}(\mathbf{1}) \rightarrow \Omega (BA/\underline{A}(\mathbf{0}))$
with a homotopy inverse of $\Omega BA \rightarrow \Omega(BA/\underline{A}(\mathbf{0}))$.
The existence of such a homotopy inverse relies upon the following lemma (in its non-equivariant form):

\begin{lemme}\label{pointedhomotopy}
Let $G$ be a topological group, $B$ be a well-pointed $G$-space, and $B \hookrightarrow A$ be a closed
$G$-cofibration. Assume $B$ is $G$-contractible. Then $A \rightarrow A/B$ is a pointed equivariant homotopy equivalence.
\end{lemme}

\begin{proof}
Since $B \hookrightarrow A$ is a $G$-cofibration and $B \rightarrow *$ is an equivariant homotopy equivalence,
we may apply the equivariant version of the homotopy theorem for cofibrations to the push-out square
$$\begin{CD}
B @>>> * \\
@VVV @VVV \\
A @>>> A/B,
\end{CD}$$
and therefore deduce that $A \rightarrow A/B$ is an equivariant homotopy equivalence.
Since $B$ is a well-pointed $G$-space and $A/B$ also is - judging from the preceding push-out square -
we deduce that $A \rightarrow A/B$ is a pointed equivariant homotopy equivalence.
\end{proof}

In order to have a map of the form claimed by Segal, it then seems that we need
to assume that $\underline{A}(\mathbf{0})$ is a well-pointed space.
In this case $\underline{A}(\mathbf{0})$ will be both contractible and well-pointed, and we may
deduce from Lemma \ref{pointedhomotopy}
that $\Omega BA \rightarrow \Omega (BA/\underline{A}(\mathbf{0}))$ is a pointed homotopy equivalence.
In the case of $\Gamma-G$-spaces, we assume that $\underline{A}(\mathbf{0})$ is a well-pointed $G$-space,
and obtain the same result in the equivariant context.
If we choose a pointed (equivariant) inverse $s : \Omega (BA/\underline{A}(\mathbf{0})) \rightarrow \Omega BA$ of
$\Omega BA \rightarrow \Omega (BA/\underline{A}(\mathbf{0}))$
up to pointed (equivariant) homotopy, we finally obtain a map
$$\underline{A}(\mathbf{1}) \longrightarrow \Omega (BA/\underline{A}(\mathbf{0})) \longrightarrow \Omega BA$$
which has the properties claimed by Segal in \cite{Segal-cat} (again, in the equivariant context).

\begin{Rem}
We could have simply assumed that $\underline{A}(\mathbf{0})$
is a point (as in the definition of an equivariant $\Gamma$-space given by
Lück and Oliver in \cite{Bob2}). However, if $\underline{A}$ is a $\Gamma-G$-space defined in such a way,
then $B\underline{A}(\mathbf{0})$ is not a point.
\end{Rem}

The following easy lemma makes our definition relevant :

\begin{lemme}
Let $G$ be a topological group and $\underline{A}$ be a $\Gamma-G$-space such that $\underline{A}(\mathbf{0})$
is a well-pointed $G$-space. Then $B^{n}\underline{A}(\mathbf{0})$ is a well-pointed $G$-space for every positive integer $n$.
\end{lemme}

Our next problem is to establish the properties of the various maps $\Omega^n B^n A \rightarrow
\Omega ^{n+1}B^{n+1}A$
that can be deduced from the previous ones. This is essentially treated in Lemma 2.6 of \cite{Bob2},
but the previous remarks make a more rigorous proof necessary (part of the problem being that it is not obvious
why the square appearing in the proof of the aforementioned lemma is commutative up to equivariant homotopy).

Given a $\Gamma-G$-space $\underline{A}$ such that $\underline{A}(\mathbf{0})$ is well-pointed,
we construct the $\Gamma-G$-space $B^{n}\underline{A}$ as follows
for every finite set $S$, we consider the $G$-space
$$\underset{(i_{1},\dots,i_{n})\in \mathbb{N}^n}{\coprod}
\underline{A}(S \times \mathbf{i_{1}} \times \dots \times \mathbf{i_{n}}) \times (\Delta ^{i_{1}} \times \dots \times \Delta ^{i_{n}}),$$
and we then construct the quotient space using the face maps.
This quotient space is no other than $B^{n}\underline{A}(S)$, as defined by Segal in \cite{Segal-cat},
and this can easily be checked by induction on $n$.
In particular, $B^{n}A$ is a quotient space of
$\underset{(i_{1},\dots,i_{n})\in \mathbb{N}^n}{\coprod}
\underline{A}(\mathbf{i_{1}} \times \dots \times \mathbf{i_{n}}) \times
(\Delta ^{i_{1}} \times \dots \times \Delta ^{i_{n}})$,
whilst $B^{n}\underline{A}(\mathbf{0})=\underline{A}(\mathbf{0}) \times B^{n}*$ is equivariantly contractible.

For $n \geq 0$ and $0 \leq k \leq n$,
the identification maps
$$\underline{A}(\mathbf{0} \times \mathbf{i_{1}} \times \dots \times \mathbf{i_{n}}) \times
 (\Delta ^{i_{1}} \times \dots
\times \Delta ^{i_{n}}) \overset{\cong}{\longrightarrow} \underline{A}(\mathbf{i_{1}} \times \dots
\times \mathbf{i_{k}} \times \mathbf{0} \times \dots \times
\mathbf{i_{n}}) \times (\Delta ^{i_{1}} \times \dots \Delta^{i_k} \times \Delta^0 \times \dots \times \Delta ^{i_{n}})$$
give rise to an injection
$$j_{n+1}^{k}: B^{n}\underline{A}(\mathbf{0}) \hookrightarrow B^{n+1}A.$$
which is a closed $G$-cofibration. For every positive integer $n$, we set:
$$B_{0}^{n}A:=\underset{0 \leq k \leq n-1}{\cup}j_{n}^{k} (B^{n-1}\underline{A}(\mathbf{0})) \subset B^{n}A.$$

\begin{prop}\label{unioncofibration}
Let $\underline{A}$ be a $\Gamma-G$-space. Then:
\begin{enumerate}[(a)]
\item The embedding $B_{0}^{n}A \hookrightarrow B^{n}A$
is a closed $G$-cofibration.
\item The space $B_{0}^{n}A$ is $G$-contractible.
\end{enumerate}
\end{prop}

\begin{proof}
We will use the notation ``$*$" to denote the trivial $\Gamma-G$-space
(with a point in every degree). The set $B_{0}^{n}A$ is a closed subspace of $B^n A$,
since it is a finite union of closed subsets.
Also, we remark that $B_{0}^{n}A \cong \underline{A}(\mathbf{0}) \times B_{0}^{n}*$.
Since $\underline{A}(\mathbf{0})$ is $G$-contractible and
$G$ acts trivially on $B_{0}^{n}*$, statement (b) will follow if we
prove that $B_{0}^{n}*$ is contractible. This last proof relies on the following double induction process.\\
The result is clearly true when $n=1$ (since $B_{0}^{1}*=*$). \\
Assume that there is some integer $n>0$ such that, for every $0 \leq k \leq n-1$, the space
$\underset{i=0}{\overset{k}{\cup}}j_{n}^{i}(B^{n-1}*)$ is contractible.
Then $j_{n+1}^{k}(B^{n}*) \cong B^{n}*$ is contractible for every integer $k\in [n]$. In particular,
$j_{n+1}^0(B^{n}*)$ is contractible
 \\
Assume further that, for some $k \in [n-1]$,
$\underset{i=0}{\overset{k}{\cup}}j_{n+1}^{i}(B^{n}*)$ is contractible.
We can then consider the following cocartesian square:

$$\begin{CD}
\underset{i=0}{\overset{k}{\cup}}j_{n}^{i}(B^{n-1}*) @>{\cong}>> j_{n+1}^{k+1}(B^{n}*) \\
@VVV @VVV \\
\underset{i=0}{\overset{k}{\cup}}j_{n+1}^{i}(B^{n}*) @>{\cong}>> \underset{i=0}{\overset{k+1}{\cup}}j_{n+1}^{i}(B^{n}*).
\end{CD}
$$
By assumptions, the two upper spaces are contractible, and it follows
that the upper horizontal map is a homotopy equivalence.
Also, the left-hand vertical map is a cofibration.
We can then apply the homotopy theorem for cofibrations to obtain that
$\underset{i=0}{\overset{k+1}{\cup}}j_{n+1}^{i}(B^{n}*)$ has
the homotopy type of $\underset{i=0}{\overset{k}{\cup}}j_{n+1}^{i}(B^{n}*)$, and deduce that it is contractible.
We conclude that $B_{0}^{n}*$ is contractible for every positive integer $n$, which proves statement (b).

It remains to prove that $B_{0}^{n}A \hookrightarrow B^{n}A$
is a $G$-cofibration. However the maps $j_n^0,j_n^1,\dots,j_n^{n-1}$ are all closed $G$-cofibrations.
Moreover, for every proper
$\sigma \subsetneq \{0,\dots,n-1\}$, and every $i \in \{0,\dots,n-1 \} \setminus \sigma$,
the inclusion $\underset{l \in \sigma }{\cap} j_{n}^{l}(B^{n-1}\underline{A}(\mathbf{0}))
\hookrightarrow \underset{l \in \sigma \cup \{i\}}{\cap} j_{n}^{l}(B^{n-1}\underline{A}(\mathbf{0}))$
is a $G$-cofibration (this is reduced to the easy case of the point
by noticing that $B^{n-1}\underline{A}(\mathbf{0})=\underline{A}(\mathbf{0})\times B^{n-1}*$).
We then apply the equivariant version of Lillig's theorem (cf. \cite{Lillig} for the non-equivariant version),
and deduce that $B^{n}_{0}A \hookrightarrow B^{n}A$
is a $G$-cofibration for every positive integer $n$. \end{proof}

\begin{cor}For every $\Gamma-G$-space and every positive integer $n$, the projection
$B^{n}A \rightarrow B^{n}A/B_{0}^{n}A$ is an equivariant pointed homotopy equivalence.
\end{cor}

For every $k\in [n]$, the composite maps
\begin{multline*}
\left(\underline{A}(\mathbf{i_{1}} \times \dots \times \mathbf{i_{n}}) \times
(\Delta ^{i_{1}} \times \dots \times \Delta ^{i_{n}}) \right ) \times I
\longrightarrow \underline{A}(\mathbf{i_{1}} \times \dots  \times \mathbf{i_{k}} \times \mathbf{1} \times \dots  \times \mathbf{i_{n}})
\times (\Delta ^{i_{1}} \times \dots \times \Delta ^{i_{k}} \times \Delta^{1} \times \dots \times \Delta ^{i_{n}}) \\
\longrightarrow B^{n+1}\underline{A} \longrightarrow B^{n+1}\underline{A}/j_{n+1}^{k}(B^{n}\underline{A}(\mathbf{0})),
\end{multline*}
for $(i_1,\dots,i_n)\in \mathbb{N}^n$,
give rise to a pointed $G$-map
$B^{n}A \times I \longrightarrow B^{n+1}A/j_{n+1}^{k}(B^n A(\mathbf{0}))$.
For every $k\in [n]$, this last map yields a continuous map
$\Sigma B^{n}A \longrightarrow
B^{n+1}A/\left(\underset{k\in[n]}{\cup}j_{n+1}^{k}(B^{n}\underline{A}(\mathbf{0}))\right)$
and, furthermore, a map $B^n A \longrightarrow \Omega (B^{n+1} A/B^{n}_{0}A)$.
If $s$ denotes an homotopy inverse
of the projection $B^{n+1}A \rightarrow B^{n+1}A/B^{n}_{0}A$, composing the preceding map
with $\Omega (s)$ yields, for every $k\in [n]$, a pointed $G$-map
$$\underline{i}_{n}^{k} : B^{n}A \longrightarrow \Omega B^{n+1}A$$
which is uniquely defined up to an equivariant pointed homotopy.

In the spirit of Segal's article \cite{Segal-cat}, the morphism $\underline{i}_{n}^{0}$
is defined solely by considering the continuous map
$B^{n}A \rightarrow \Omega (B^{n+1}A/j_{n+1}^{0}(B^{n}\underline{A}(\mathbf{0})))$
defined earlier in our construction, and by composing it
with a homotopy inverse of $\Omega (B^{n+1}A) \rightarrow
\Omega(B^{n+1}A/j_{n}^{0}(B^{n}\underline{A}(\mathbf{0})))$
(such an inverse exists because $j_{n}^{0}$ is a cofibration and $B^{n}\underline{A}(\mathbf{0})$ is $G$-contractible).
It is now clear that the morphism we have just constructed is the same one as Segal's,
as a morphism in the category $CG_G^{h\bullet}$.

Let $\Sigma_n$ denote the group of permutations of $\{1,\dots,n\}$, and
let  $\sigma \in \Sigma _{n}$. Then $\sigma$ gives rise to
$\underline{\sigma}: \begin{cases}
I^{n}/\partial I^n & \rightarrow I^{n}/\partial I^n \\
(x_1,\dots,x_n) & \mapsto (x_{\sigma (1)},\dots,x_{\sigma (n)})
\end{cases}$, and furthermore, for every pointed $G$-space $X$, to an equivariant homeomorphism
$$\overline{\sigma}_X:
\begin{cases}
\Omega ^n X  & \longrightarrow \Omega ^n X \\
(\varphi : S^n \rightarrow X) & \longmapsto (\varphi \circ (\underline{\sigma})^{-1} : S^{n} \rightarrow X).
\end{cases}
$$

For $0 \leq k \leq n$, we let $\sigma _{k} \in \Sigma _{n+1}$ denote the permutation of $\{1,\dots,n+1\}$ which acts trivially on
$\{1,\dots,n-k\}$, and whose restriction to $\{n-k+1,\dots,n+1 \}$ is the decreasing cycle (i.e. $(n+1,n,n-1,\dots,n-k+1)$).
We finally set
$$i_{n}^{k}:=\overline{({\sigma _{k}})}_{B^{n+1}A} \circ \Omega ^{n}(\underline{i}_{n}^{k}):
\Omega ^{n}B^{n}A \longrightarrow \Omega ^{n+1}B^{n+1}A.$$
From there, the properties of the $i_n^k$'s are similar to those claimed in \cite{Bob2}
and are proven with similar arguments. We shall only restate them here:

\begin{aprop}\label{inkgweak}
For every positive integer $n$ and every $k \in [n]$, the morphism $i_n^k$ is a $G$-weak equivalence.
\end{aprop}

\begin{aprop}\label{independanceonk}
For any fixed positive integer $n$, the maps $i_n^0,i_n^1,\dots,i_n^n$ all
define the same morphism in the category $CG_G^{h \bullet}[W_{G}^{-1}]$.
\end{aprop}

\begin{Rems}
\begin{enumerate}[(i)]
\item For any positive integer $n$, the maps $i_n^k$ are morphisms between equivariant H-spaces.
\item For any Lie group $G$, and any $m\in \mathbb{N}^* \cup \{\infty\}$,
$\underline{\vEc} _{G}^{F,m}(\mathbf{0})=*$. This justifies that the previous results can be applied
to the $\Gamma-G$-space $\underline{\vEc}_G^{F,m}$, as is done in Section \ref{7}.
\end{enumerate}
\end{Rems}

\section*{Acknowledgements}

I would like to thank Bob Oliver for his constant support and the countless good advice he gave me
during my research on this topic.

\bibliographystyle{plain}

\end{document}